\documentclass[11pt,a4paper]{article}
\usepackage[T1]{fontenc}
\usepackage{iftex}
\ifPDFTeX\usepackage[utf8]{inputenc}\fi
\usepackage{lmodern,microtype,mathtools,amssymb,amsthm,mathrsfs}
\usepackage[a4paper,margin=28mm,headheight=14pt]{geometry}
\usepackage{enumitem,array,needspace}
\usepackage[hidelinks]{hyperref}
\numberwithin{equation}{section}
\newtheorem{theorem}{Theorem}[section]
\newtheorem{proposition}[theorem]{Proposition}
\newtheorem{lemma}[theorem]{Lemma}
\newtheorem{corollary}[theorem]{Corollary}
\theoremstyle{definition}

\newtheorem{question}[theorem]{Question}
\theoremstyle{remark}
\newtheorem{remark}[theorem]{Remark}
\newcommand{\Pp}{\mathbb P}
\newcommand{\E}{\mathbb E}
\newcommand{\R}{\mathbb R}
\newcommand{\N}{\mathbb N}
\newcommand{\Z}{\mathbb Z}
\newcommand{\one}{\mathbf 1}
\newcommand{\calG}{\mathcal G}
\newcommand{\D}{\mathcal D}
\newcommand{\dimH}{\operatorname{dim}_{\mathrm H}}
\newcommand{\dimB}{\operatorname{dim}_{\mathrm B}}
\newcommand{\dimUB}{\overline{\operatorname{dim}}_{\mathrm B}}
\newcommand{\law}{\overset{\mathrm{law}}=}
\newcommand{\cH}{\mathscr H}
\DeclareMathOperator{\diam}{diam}
\DeclareMathOperator{\Var}{Var}
\DeclareMathOperator{\Leb}{Leb}
\setlist[enumerate]{label=\textup{(\roman*)},leftmargin=*,itemsep=3pt,topsep=4pt}
\allowdisplaybreaks[2]
\hypersetup{pdftitle={Graph distance and percolation dimension of three-dimensional Brownian motion},pdfauthor={Daisuke Shiraishi}}
\title{Graph distance and percolation dimension\\of three-dimensional Brownian motion}
\author{Daisuke Shiraishi\thanks{Graduate School of Informatics,
  Kyoto University, Kyoto 606-8501, Japan.}}
\date{}
\begin{document}
\maketitle
\begin{abstract}
We prove that the minimum Hausdorff dimension of a continuous curve
joining $W_0$ to $W_1$ within the range $W([0,1])$ of a
three-dimensional Brownian motion is almost surely a deterministic
constant $\gamma_3\in(1,5/3]$.
Moreover, we construct continuum many simple curves with pairwise
distinct images that attain this minimum.
Each curve has box dimension $\gamma_3$ and admits a
parametrization that is $\beta$-H\"older for every
$0<\beta<1/\gamma_3$.
To identify $\gamma_3$, we partition Brownian time into $n$ equal
intervals and form a graph by joining pairs with intersecting ranges.
We show that its endpoint graph distance and diameter have a common
deterministic almost sure growth exponent $\theta$ and that
$\gamma_3=2\theta$.
We also establish $n^{p\theta+o(1)}$ growth for their $p$-th moments,
for every $p>0$.
\end{abstract}
\tableofcontents
\newpage

\section{Introduction}\label{sec:introduction}

We study the smallest Hausdorff dimension of a continuous curve
joining two prescribed points within the range of a three-dimensional
Brownian motion, a version of the percolation dimension problem
introduced by Burdzy~\cite{Burdzy}.

A related problem for planar Brownian motion is also listed among the open
problems in M\"orters and Peres~\cite[Selected open problems, (2)]{MortersPeres}.
In connection with this question,
Pemantle~\cite[Theorem 1.2]{Pemantle} proved that the planar Brownian
range contains no nondegenerate line segment. On the other hand,
Zhan~\cite[Theorem 3.1]{Zhan} constructed, inside the trace of planar
Brownian motion stopped on leaving the unit disc, a curve with the law
of radial Schramm--Loewner evolution (SLE) with parameter~2 (see Schramm~\cite{SchrammSLE} for the introduction of SLE).
By Beffara~\cite[Introduction, main theorem]{Beffara}, this curve has
Hausdorff dimension $5/4$, which is also the growth exponent of planar
loop-erased random walk, introduced by Lawler~\cite{Lawler1980}
(for its precise definition, see \eqref{eq:chronological-loop-erasure}
in Section~\ref{subsec:dimension-proof}).
The latter identity was proved by
Kenyon~\cite[Theorem 3]{Kenyon} and, by a different method, by
Masson~\cite[Theorem 1.1]{Masson}; see also
Lawler~\cite[Theorem 1.1]{LawlerEdge}.
However, these results do not determine the minimum curve dimension
in a planar Brownian range or the exact exponent for the planar
analogue of the graph distance introduced below.

In three dimensions, the infimum of the Hausdorff dimensions of curves
joining the initial and terminal points within a Brownian range is known
to lie strictly between one and two.
Lawler~\cite[Corollary 1.2]{LawlerCut} used cut points to give the
strict lower bound, whereas
Fargason~\cite[Theorem 1.1 and Corollary 1.2]{Fargason} constructed a
connecting curve whose dimension is bounded above by a deterministic
constant strictly smaller than two. For general dimension bounds
based on crossing estimates, we refer to Aizenman and
Burchard~\cite[Theorem 1.3 and Appendix A.4]{AizenmanBurchard};
related work involving Brownian loop soup is discussed in
Gao, Li and Li~\cite[Section 1.2]{GaoLiLi}.
Our approach relates the minimum curve dimension to graph distance:
we divide Brownian time into equal intervals and join pairs whose
ranges intersect. We identify the minimum dimension with twice the
growth exponent of this graph distance and construct continuum many
simple curves attaining the minimum.

Graph distance also arises on the trace of a random walk, whose
vertices and edges are the visited sites and traversed edges.
Burdzy and Lawler~\cite{BurdzyLawler1990} studied graph distance on this trace, while
Shiraishi~\cite[Section 4]{Shiraishi} developed an ergodic approach
for the trace of a non-intersecting two-sided random walk in dimensions
two and three. Related results in dimensions four and higher appear
in~\cite{AdhikariOkada,AdhikariOkadaShiraishi,Croydon,CroydonShiraishi,ShiraishiWatanabe}.

To state our results, let $W=(W_t)_{t\ge0}$ be standard Brownian motion
in $\R^3$ started at the origin $\mathbf0=(0,0,0)$, and set
$K=W([0,1])$. We write $\Pp$ and $\E$ for probability and expectation
under the law of $W$. For an interval $J\subset\R$, a continuous map
$w:J\to\R^3$, and a set $I\subset J$, write
$w(I)=\{w(t):t\in I\}$. A \emph{curve} is a continuous map from an
interval into a metric space, and it is \emph{simple} if it is
injective. A \emph{graph path} is a finite sequence of vertices such that consecutive vertices are adjacent; vertices may be repeated.

The space $C(A,E)$ of continuous maps from a compact metric space $A$
to a metric space $E$ is equipped with the topology of uniform convergence.
We write $\dimH$, $\dimB$ and $\dimUB$ for Hausdorff, box and upper
box dimension, respectively, in the Euclidean metric. For $x,y\in K$, define
\begin{equation}\label{eq:percolation-definition}
 p_K(x,y)=\inf\left\{\dimH\eta([0,1]):
 \begin{array}{l}
 \eta\in C([0,1],K),\\[-1mm]
 \eta(0)=x,\quad\eta(1)=y
 \end{array}\right\},\qquad
 \mathfrak p_3(W)=p_K(W_0,W_1).
\end{equation}
The curves in this infimum need not follow the time order of $W$.

Partition $[0,1]$ into the closed time cells
$I_i^{(n)}=[i/n,(i+1)/n]$, $0\le i \le n-1$, where $n\ge1$ is an integer.
The undirected simple graph $G_n=G_n(W)$
has vertex set $V(G_n)=\{0,\ldots,n-1\}$ and edges given by
\begin{equation}\label{eq:graph-definition}
 i\sim j\quad\Longleftrightarrow\quad
 i\ne j,\quad W(I_i^{(n)})\cap W(I_j^{(n)})\ne\varnothing.
\end{equation}
Consecutive time cells share an endpoint, so their vertices are
adjacent and $G_n$ is connected. We use graph distance with unit edge
lengths: for a finite connected graph $G$, let $d_G(x,y)$ be the minimum
number of edges in a graph path from $x$ to $y$, and set
$\diam G=\max_{x,y\in V(G)}d_G(x,y)$.
We consider the endpoint distance $L_n$ and the diameter $V_n$:
\begin{equation}\label{eq:finite-observables}
 L_n=d_{G_n}(0,n-1),\qquad V_n=\diam G_n.
\end{equation}
Both quantities are positive for $n\ge2$.
In contrast, Brownian motion is almost surely simple in dimensions
$d\ge4$~\cite[Theorem 9.22]{MortersPeres}, and the same graph
construction then gives endpoint distance and diameter $n-1$.

\subsection{Main result}\label{subsec:main-dimension}

The following main theorem of this paper identifies the minimum curve
dimension with twice the common growth exponent of $L_n$ and $V_n$,
and shows that continuum many simple curves with distinct images
attain this minimum.
\begin{theorem}\label{thm:main-growth}
There is a deterministic exponent $\theta\in(1/2,5/6]$ such that,
for every $p>0$,
\begin{equation}\label{eq:main-moments}
 \E[L_n^p]=n^{p\theta+o(1)},\qquad
 \E[V_n^p]=n^{p\theta+o(1)}.
\end{equation}
Set $\gamma_3=2\theta$. On an event of probability one,
\begin{equation}\label{eq:main-pathwise}
 \lim_{n\to\infty}\frac{\log L_n}{\log n}
 =\lim_{n\to\infty}\frac{\log V_n}{\log n}
 =\theta,\qquad \mathfrak p_3(W)=\gamma_3.
\end{equation}
On this event, there are continuum many simple curves $\eta:[0,1]\to K$
with pairwise distinct images, each $\beta$-H\"older for every
$0<\beta<1/\gamma_3$ and satisfying
\begin{equation}\label{eq:main-minimizers}
 \eta(0)=W_0,\qquad \eta(1)=W_1,\qquad
 \dimH\eta([0,1])=\dimB\eta([0,1])=\gamma_3.
\end{equation}
\end{theorem}

The following remark records the known bounds on $\gamma_3$ and their
relation to three-dimensional loop-erased random walk.
\begin{remark}\label{rem:dimension-bounds}
Let $\ell_3$ denote the growth exponent of three-dimensional
loop-erased random walk: the expected number of edges in the
loop-erasure of a simple random walk on $\Z^3$ from the origin,
stopped on leaving the ball of radius $R$, is
$R^{\ell_3+o(1)}$~\cite{Shiraishi}.
Then
\begin{equation}\label{eq:main-lerw}
 1<\gamma_3\le\ell_3\le\frac53.
\end{equation}
The strict lower bound follows from Lawler's theorem on cut
points~\cite[Corollary~1.2]{LawlerCut}. The comparison with $\ell_3$
uses the coupling of Sapozhnikov and
Shiraishi~\cite[Theorems~1.1 and~1.2]{SapozhnikovShiraishi} and the
dimension identity for the loop-erased random walk
limit~\cite{ShiraishiDimension}; see
Section~\ref{subsec:dimension-proof}. Finally, the upper bound
$\ell_3\le5/3$ is due to Lawler~\cite{LawlerLERW}.
\end{remark}

\subsection{Further results}\label{subsec:further-results}

Another main contribution of this paper is Theorem~\ref{thm:main-metric}.
Because its statement requires additional notation, we present it
separately from Theorem~\ref{thm:main-growth}.
It establishes joint subsequential limits of normalized internal graph
distances, gives a modulus of continuity for the limiting pseudometrics,
and gives a gluing formula that reconstructs each limit
from the internal distances and the Brownian self-intersection
relation for every finite dyadic time partition. 
It also proves that these limits are measurable functions
of the Brownian path and that, along the corresponding subsequences,
the normalized distances converge on the original Wiener space in
H\"older norms of every exponent below $\theta$.
For each fixed deterministic $s\in[0,1]$, the limiting pseudometric
almost surely separates $s$ from all other time parameters.
Proposition~\ref{prop:coarse-information} further shows that each
limit is determined by the Brownian self-intersection relation.

We also extend the dimension identity and nonuniqueness in
Theorem~\ref{thm:main-growth} to curves from $W_s$ to every
$y\in K\setminus\{W_s\}$, simultaneously for all such $y$ and almost
surely for each fixed deterministic $s\in[0,1]$; see
Theorem~\ref{thm:main-pinned}. The corresponding dimension lower bound
for all pairs of distinct points on one event remains open.

Finally, we consider spatial discretizations obtained by successive
exits from balls of radius $r$. Theorem~\ref{thm:main-exit} gives
almost sure growth $r^{-\gamma_3+o(1)}$ for the endpoint distance and
diameter of the resulting intersection graph, together with the
corresponding moment exponents. Corollary~\ref{cor:quantitative-lower-bounds} gives
quantitative lower bounds for equal time partitions.

\subsection{Proof of the main theorem}\label{subsec:proof-ideas}

We outline the proof of Theorem~\ref{thm:main-growth}.
We first introduce an auxiliary quantity defined as one plus
the maximum diameter of the subgraphs induced by consecutive
unit time cells. This quantity is defined in \eqref{eq:envelope},
and Brownian scaling relates the underlying graphs to $G_n$
through \eqref{eq:unit-scaling}.
The comparison of two time meshes in
Lemma~\ref{lem:lifting}, together with the tail bound in
Proposition~\ref{prop:upper-tail}, gives an approximate product inequality for
the expectation of this auxiliary quantity. Proposition~\ref{prop:product}
then proves the existence of its growth exponent using
Fekete's lemma. The grid comparison in Lemma~\ref{lem:grid} identifies
this exponent with those of the endpoint distance and graph diameter
in expectation. Combining these estimates with the Paley--Zygmund
inequality gives the growth of all their positive moments; see
Proposition~\ref{prop:finite-exponent}.

To prove $\mathfrak p_3(W)\le2\theta$,
Proposition~\ref{prop:curve-upper} constructs a curve from $W_0$ to
$W_1$ inside $K$ using nested partitions of Brownian time.
The graph diameter estimates control the number of time cells retained
at each scale, while Brownian H\"older regularity bounds their spatial
diameters. The resulting curve has upper box dimension at most
$2\theta$ and is $\beta$-H\"older for every
$0<\beta<1/(2\theta)$. It therefore has finite $p$-variation
for every $p>2\theta$. Lemma~\ref{lem:holder-arc} then gives
a simple curve with the same endpoints and a single parametrization
satisfying all these H\"older bounds.

For the reverse inequality, we construct $\alpha$-H\"older functions
on Brownian time, for $0<\alpha<\theta$, that take equal values
whenever the Brownian positions agree. The starting point is the
graph distance from the initial time cell, divided by a deterministic
scale comparable to the expected endpoint distance.
Proposition~\ref{prop:scalar-seed} uses the moment bounds and
a positive lower bound for the normalized endpoint mean to obtain
a nonconstant subsequential limit with positive probability.
With positive probability, such a function can be made
constant near time one. The transience argument in
Lemma~\ref{lem:escape} then allows it to be extended to all
later times while preserving compatibility.
Brownian scaling and Blumenthal's zero--one law yield,
almost surely, a compatible function $u$ with $u(0)=0$
and $u(1)=1$; see Proposition~\ref{prop:plateau}.
Compatibility lets $u$ descend to a continuous function $\bar u$
on $K$, with $\bar u(W_t)=u(t)$. For any connected set
$C\subset K$ joining $W_0$ to $W_1$, the image $\bar u(C)$ contains
$[0,1]$. The H\"older bound therefore forces the set of Brownian times
whose positions lie in $C$ to have Hausdorff dimension at least
$\alpha$. Kaufman's dimension doubling theorem
gives $\dimH C\ge2\alpha$, as in Lemma~\ref{lem:dimension-test}.
Letting $\alpha\uparrow\theta$ identifies the minimum.

These functions also give the almost sure lower growth bound for
$L_n$ by summing their increments along a graph path.
On the other hand, the tail estimate and the Borel--Cantelli lemma
give the matching upper bound for $V_n$.
Section~\ref{subsec:dimension-proof} combines these arguments and
establishes the comparison with the loop-erased random walk exponent.

To prove nonuniqueness, Lemma~\ref{lem:double-point-avoidance} uses
the absence of pioneer double cut points, proved by
Gao, Li, Liu and Qian~\cite[Theorem~1.2]{GaoLiLiuQian}, to construct
alternative connections in the Brownian range.
In Proposition~\ref{prop:nonunique-curves}, we use these connections
to replace countably many pairwise disjoint subarcs of a minimizing
curve. The local variation estimate in
Lemma~\ref{lem:local-variation}, together with the H\"older
regularity of the original curve, ensures that the changes have
summable $p$-variation for every $p>\gamma_3$.
For each subarc, we may either retain the original subarc or
use its replacement. A selected double point belongs to the
resulting curve's image precisely when the original subarc is retained.
Thus different binary sequences produce continuum many simple curves
with distinct images. Lemma~\ref{lem:holder-arc} then gives each
curve a parametrization that is $\beta$-H\"older for every
$0<\beta<1/\gamma_3$.

\subsection{Organization}\label{subsec:organization}

Section~\ref{sec:growth} begins with the necessary notation and then
proves the graph distance exponent and moment estimates.
Section~\ref{sec:curves} completes the proof
of Theorem~\ref{thm:main-growth}. Sections~\ref{sec:limits}
and~\ref{sec:gluing} construct limiting distances and prove their
reconstruction and measurability, respectively.
Section~\ref{sec:separation} establishes separation from deterministic
times and the corresponding curve dimension results.
Section~\ref{sec:exits} treats ball exit partitions and quantitative
lower bounds. Section~\ref{sec:questions} concludes with several
open problems.

\section{The graph distance exponent}\label{sec:growth}

In this section, we prove that the endpoint distance and graph diameter
have the same expectation and moment exponents, as asserted in
Theorem~\ref{thm:main-growth}. We compare these quantities using the
interval maximum defined in \eqref{eq:envelope}, then choose the
subsequence and normalization needed to construct limits in
Section~\ref{sec:limits}.

\subsection{Notation and interval maxima}\label{subsec:interval-maxima}

We begin with notation, then introduce the interval maxima and prove
their subadditivity and exponential tail bound.
We use $\N=\{1,2,\ldots\}$, $|I|$ for interval length, $|\cdot|$ for
the Euclidean norm, and $B(x,r)$ for the open Euclidean ball. The
filtration $(\mathcal F_t)_{t\ge0}$ is the completed, right-continuous
natural filtration of $W$. We write $\Pp_x,\E_x$ when Brownian motion
starts at $x$, $\Leb$ for Lebesgue measure, and $\one_A$ for the
indicator of $A$. The symbols $\law$ and $\xrightarrow{\mathrm d}$ denote
equality and convergence in distribution, respectively. For paths on
an unbounded time interval we use uniform convergence on compact sets.
All measurability statements refer to the Borel sigma-fields of the
stated spaces, with completion when specified. Subsequences are
deterministic unless stated otherwise.

Unless specified otherwise, the constants $c,C$ are positive, finite
and deterministic, may change from line to line, and have any additional
dependence indicated by subscripts. For positive sequences,
$a_n\asymp b_n$ means that their
ratio is bounded above and below by positive constants, and
$a_n=n^{\alpha+o(1)}$ means $\log a_n/\log n\to\alpha$. For a random
variable $Z$ with values in a normed space $E$ and $p\ge1$, put
$\|Z\|_{L^p(\Pp;E)}=(\E\|Z\|_E^p)^{1/p}$; for real random variables
we simply write $\|Z\|_p$.

For a function $f$ from a compact subset $A$ of a Euclidean
space into a Euclidean space, write
\[
 \|f\|_\infty=\sup_{x\in A}|f(x)|,\qquad
 [f]_{\alpha;A}
 =\sup_{\substack{x,y\in A\\x\ne y}}
       \frac{|f(x)-f(y)|}{|x-y|^\alpha}
       \quad(0<\alpha<1).
\]
When $A=[0,1]$, we abbreviate $[f]_{\alpha;[0,1]}$ to $[f]_\alpha$.

The cell assigned to $t\in[0,1]$ at mesh $1/n$, $n\in\N$, has index
\[
  \iota_n(t)=\min\{\lfloor nt\rfloor,n-1\}.
\]
At a shared endpoint we choose the cell on the right, while $t=1$
belongs to the last cell.

On unit time intervals, define $\calG[a,b]$ for integers $0\le a<b$
to have vertices $a,\ldots,b-1$, with distinct $i,j$ adjacent when
$W([i,i+1])\cap W([j,j+1])\ne\varnothing$. Brownian scaling and translation invariance give
\begin{equation}\label{eq:unit-scaling}
 \calG[0,N]\law G_N,
 \qquad \calG[a,a+N]\law\calG[0,N].
\end{equation}

For comparisons between time scales, we need diameter bounds for every
consecutive group of time cells.
Set
\begin{align}
 \lambda[a,b]&=d_{\calG[a,b]}(a,b-1),\notag\\
 Y_{a,b}&=1+\max_{a\le u<v\le b}\diam\calG[u,v],
 \qquad Y_{a,a}=0,\label{eq:envelope}\\
 Y_N&=Y_{0,N},\qquad \mu_N=\E Y_N,\qquad
 f_N=\E(1+L_N),\qquad g_N=\E(1+V_N).\notag
\end{align}
Note that one plus graph distance counts the cells in a shortest graph path. Thus $1\le f_N\le g_N\le\mu_N\le N$, with $f_1=g_1=\mu_1=1$.

The first lemma gives the subadditivity and endpoint representation of
the interval maximum.
\begin{lemma}\label{lem:envelope}
For integers $0\le a\le c\le b$ and $N\ge1$,
\begin{align}
 Y_{a,b}&\le Y_{a,c}+Y_{c,b},\label{eq:subadd}\\
 Y_N&\le Y_{N+1}\le Y_N+1,\qquad 1\le Y_N\le N,\label{eq:increments}\\
 Y_N&=1+\max_{0\le u<v\le N}\lambda[u,v].\label{eq:endidentity}
\end{align}
\end{lemma}
\begin{proof}
We first prove \eqref{eq:subadd}. Fix integers
$a\le u<v\le b$ and vertices $u\le x\le y<v$.
If $v\le c$ or $u\ge c$, the required bound follows
directly from the definition of $Y_{a,c}$ or $Y_{c,b}$.
Suppose now that $u<c<v$. If $x,y<c$, they can be
joined within $\calG[u,c]$ using at most $Y_{a,c}-1$
edges. Similarly, if $c\le x,y$, they can be joined
within $\calG[c,v]$ using at most $Y_{c,b}-1$ edges.
If $x<c\le y$, join $x$ to $c-1$ within $\calG[u,c]$,
cross the edge $\{c-1,c\}$, and join $c$ to $y$
within $\calG[c,v]$. This path has at most
\[
 (Y_{a,c}-1)+1+(Y_{c,b}-1)
 =Y_{a,c}+Y_{c,b}-1
\]
edges. Thus, in all cases,
\[
 1+d_{\calG[u,v]}(x,y)\le Y_{a,c}+Y_{c,b}.
\]
Taking the maximum over the restrictions and their
vertices proves \eqref{eq:subadd}.

By \eqref{eq:envelope}, we have $Y_N\le Y_{N+1}$.
Applying \eqref{eq:subadd} and using $Y_{N,N+1}=1$,
we obtain $Y_{N+1}\le Y_N+1$.
Since consecutive vertices are adjacent,
$\diam\calG[u,v]\le v-u-1\le N-1$
for $0\le u<v\le N$.
Hence $1\le Y_N\le N$, proving
\eqref{eq:increments}.

Finally, for $0\le u\le x<y<v\le N$, the graph
$\calG[x,y+1]$ is the subgraph of $\calG[u,v]$
induced by the vertices $x,\ldots,y$. Hence
\[
d_{\calG[u,v]}(x,y)
\le d_{\calG[x,y+1]}(x,y)
=\lambda[x,y+1].
\]
Taking maxima gives
\[
Y_N\le 1+\max_{0\le u<v\le N}\lambda[u,v].
\]
The reverse inequality follows from
$\lambda[u,v]\le\diam\calG[u,v]$.
This proves \eqref{eq:endidentity}.
\end{proof}

\vspace{1.5mm}

Subadditivity and the strong Markov property yield the following exponential tail estimate.
\begin{proposition}\label{prop:upper-tail}
There are constants $c,C>0$ such that
\begin{equation}\label{eq:upper-tail}
 \Pp(Y_N\ge t\mu_N)\le Ce^{-ct},\qquad N\ge1,\quad t\ge0.
\end{equation}
For every integer $p\ge1$, we have $\E[Y_N^p]\le p!(C\mu_N)^p$. Furthermore, for
integers $r,s\ge1$, if $Z_1,\ldots,Z_s$ each have the law of $Y_r$,
with arbitrary dependence, then
\begin{equation}\label{eq:max-tail}
 \left\|\max_{1\le i\le s}Z_i\right\|_2
 \le C\log(s+2)\mu_r.
\end{equation}
\end{proposition}

\begin{proof}
Fix $N$ and define stopping times by $T_0=0$ and
\[
 T_{j+1}=\inf\{m\in\Z:m>T_j,\ Y_{T_j,m}\ge2\mu_N\}.
\]
The infimum of the empty set is $+\infty$; if $T_j=+\infty$,
set $T_{j+1}=+\infty$. By monotonicity and Markov's inequality,
$\Pp(T_1\le N)\le\Pp(Y_N\ge2\mu_N)\le1/2$.
On $\{T_j\le N\}$, the strong Markov property at $T_j$ and
monotonicity in the remaining time $N-T_j\le N$ give
\[
 \Pp(T_{j+1}\le N\mid\mathcal F_{T_j})
 \le \Pp(Y_N\ge2\mu_N)\le\tfrac12.
\]
Iterating this conditional estimate gives
$\Pp(T_j\le N)\le2^{-j}$.

By \eqref{eq:subadd} and $Y_{m,m+1}=1$, we have
$Y_{a,m+1}\le Y_{a,m}+1$ for integers $a\le m$.
Hence each completed interval has maximum less than
$2\mu_N+1$. On $\{T_j>N\}$, there are at most $j-1$
such intervals before $N$, followed by one interval whose
maximum is less than $2\mu_N$. Subadditivity therefore gives
\[
 Y_N<j(2\mu_N+1)\le3j\mu_N.
\]
It follows that $\Pp(Y_N\ge3j\mu_N)\le2^{-j}$,
which proves \eqref{eq:upper-tail} after adjusting the
constants. The moment bound follows from \eqref{eq:upper-tail}.
A union bound gives
\[
 \Pp\left(\max_{1\le i\le s}Z_i>t\mu_r\right)
 \le\min\{1,Cs e^{-ct}\}.
\]
Splitting the tail integral at $t=C\log(s+2)$
yields \eqref{eq:max-tail}.
\end{proof}

\subsection{Comparison of two time meshes}\label{subsec:two-partitions}

For positive integers $r,s$, consider the unit time partition
of $[0,rs]$ and the coarser partition into $s$ intervals
$[jr,(j+1)r]$, $0\le j<s$. Let $C_{r,s}$ be the simple graph
with vertices $0,\ldots,s-1$, joining distinct $i$ and $j$
when
\[
 W([ir,(i+1)r])\cap W([jr,(j+1)r])
 \ne\varnothing.
\]
Define $Y(C_{r,s})$ as one plus the maximum diameter
of the induced subgraphs of $C_{r,s}$ on nonempty
consecutive vertex sets. Set
\[
 M_{r,s}=\max_{0\le j<s}Y_{jr,(j+1)r}.
\]
The following lemma bounds $Y_{rs}$ by constructing paths
in the finer graph from paths in the coarser graph.

\begin{lemma}\label{lem:lifting}
For every realization of the Brownian path,
\begin{equation}\label{eq:lifting}
 Y_{rs}\le3M_{r,s}Y(C_{r,s}).
\end{equation}
\end{lemma}

\begin{proof}
We first consider $\calG[kr,\ell r]$, where
$0\le k<\ell\le s$. Fix two vertices $x,y$ of this
graph and choose a shortest path between their
containing intervals in the subgraph of $C_{r,s}$
induced by the vertices $k,\ldots,\ell-1$.
If this path has $q$ edges, it visits $q+1$
intervals, with $q+1\le Y(C_{r,s})$.
For each edge, choose intersecting unit time cells
in the two incident intervals. These choices
specify the cells to be joined inside each visited
interval, using $x$ and $y$ at the ends.
Each internal connection needs at most
$M_{r,s}-1$ edges, and there are $q$ edges
between intervals. Thus the resulting path in
$\calG[kr,\ell r]$ has length at most
\[
 q+(q+1)(M_{r,s}-1)
 =M_{r,s}(q+1)-1
 \le M_{r,s}Y(C_{r,s})-1.
\]

Now consider a restriction $[u,v]\subset[0,rs]$
with integer endpoints. By \eqref{eq:endidentity},
it suffices to bound $1+\lambda[u,v]$.
The restriction consists of complete intervals
of length $r$ and at most two shorter end portions.
Each end portion is a consecutive restriction
within one interval of length $r$. It can be
traversed using at most $M_{r,s}-1$ edges,
and one additional edge connects it to the
adjacent complete interval.
When there is a complete middle part, the preceding
bound therefore gives
\[
 1+\lambda[u,v]
 \le M_{r,s}Y(C_{r,s})+2M_{r,s}
 \le 3M_{r,s}Y(C_{r,s}).
\]
If no complete interval is present, the restriction
lies in at most two consecutive intervals of length
$r$, and $1+\lambda[u,v]\le 2M_{r,s}$.
Taking the maximum in \eqref{eq:endidentity}
proves \eqref{eq:lifting}.
\end{proof}

Combining \eqref{eq:lifting} with Brownian scaling and the moment
bounds in Proposition~\ref{prop:upper-tail} gives the following
approximate product inequality and, in turn, the existence of the growth exponent.
\begin{proposition}\label{prop:product}
For integers $r,s\ge1$, we have
\begin{equation}\label{eq:product}
 \mu_{rs}\le C\min\{1+\log r,1+\log s\}\mu_r\mu_s. 
\end{equation}
Furthermore, there is a deterministic $\theta\in[0,1]$ such that
\begin{equation}\label{eq:envelope-exponent}
 \mu_N=N^{\theta+o(1)}.
\end{equation}
\end{proposition}
\begin{proof}
By Brownian scaling and translation invariance,
$Y(C_{r,s})$ has the law of $Y_s$, and each variable
in the maximum defining $M_{r,s}$ has the law of $Y_r$. Applying
Cauchy--Schwarz to \eqref{eq:lifting}, followed by the estimates in
Proposition~\ref{prop:upper-tail}, gives
\[
 \mu_{rs}\le3\|M_{r,s}\|_2\|Y(C_{r,s})\|_2
 \le C(1+\log s)\mu_r\mu_s.
\]
Repeating the comparison with $r$ intervals of length $s$ gives the
same bound with $1+\log r$. Taking the smaller of the two bounds proves
\eqref{eq:product}.

Increase the constant $C$ in \eqref{eq:product} to be at least one, and
put $\widehat\mu_N=2C(1+\log N)\mu_N$. Since
\[
 1+\log(rs)\le2\max\{1+\log r,1+\log s\},
\]
\eqref{eq:product} gives
$\widehat\mu_{rs}\le\widehat\mu_r\widehat\mu_s$. Thus
$\log\widehat\mu_{2^k}$ is subadditive in $k$. Fekete's lemma gives the
limit of $\log\widehat\mu_{2^k}/(k\log2)$; see \cite[proof of Theorem 6.4.1,
equations (6.4.1)--(6.4.2)]{Durrett}. The same limit holds for $\log\mu_{2^k}/(k\log2)$, since
$\log(1+\log N)=o(\log N)$. Monotonicity of $\mu_N$ then extends
the limit from powers of two to all integers. The bounds $1\le\mu_N\le N$ complete the proof.
\end{proof}

\subsection{Endpoint distance and graph diameter}\label{subsec:endpoint-diameter}

We next show that $f_N$ and $g_N$ have the same growth
exponent $\theta$ as $\mu_N$ in
\eqref{eq:envelope-exponent}. To this end, the following
lemma provides the mesh comparison and grid decomposition
needed to obtain a lower bound on $f_N$.
\begin{lemma}\label{lem:grid}
For integers $1\le M\le N$,
\begin{equation}\label{eq:refinement}
 f_M\le2f_N.
\end{equation}
Furthermore, for integers $2\le q\le N$, let $L=\lceil N/q\rceil$. Then
\begin{equation}\label{eq:grid}
 \mu_N\le C\log(q+2)\mu_L+2q^2f_N.
\end{equation}
\end{lemma}

\begin{proof}
To prove \eqref{eq:refinement}, construct $G_M$ and $G_N$ from the
same Brownian path on $[0,1]$. Their common refinement has all
division points from both partitions. Since $M\le N$, each interval
of length $1/N$ splits into at most two refinement intervals.
Choose an intersection point for each edge of a shortest path
in $G_N$. In each visited cell, connect the two resulting time
representatives through the common refinement; use times $0$ and $1$
at the ends of the path. At most two refinement intervals are used
in each cell. Projecting these intervals to their containing cells
in $G_M$, and deleting consecutive repetitions, gives a graph path from vertex \(0\) to vertex \(M-1\) with at most \(2(1+L_N)\) vertices.
Thus $1+L_M\le2(1+L_N)$, and taking expectations proves
\eqref{eq:refinement}.

For the grid decomposition, choose integers
$0=t_0<\cdots<t_s=N$ with $s\le q$ and $t_{i+1}-t_i\le L$.
Fix a restriction $[u,v]$ in \eqref{eq:endidentity}. If it contains
two or more grid points, let $t_i$ and $t_j$ be the first and last.
Join $u$ to $t_i$, use a shortest path from $t_i$ to $t_j-1$ in
$\calG[t_i,t_j]$, and then join $t_j-1$ to $v-1$.
Each nonempty end portion lies within one grid interval
$[t_k,t_{k+1}]$. Its internal connection uses at most
$Y_{t_k,t_{k+1}}-1$ edges, and at most one additional edge
joins it to the middle part. If there are fewer than two grid points, the whole
restriction lies in at most two grid intervals, so
its endpoint distance is at most
$2\max_{k<s}Y_{t_k,t_{k+1}}-1$.
In either case, maximizing over $[u,v]$ gives
\[
 Y_N\le2\max_{j<s}Y_{t_j,t_{j+1}}
       +\max_{i<j}\bigl(1+\lambda[t_i,t_j]\bigr).
\]
To take expectations, extend each grid interval to length $L$
and dominate its maximum by a variable with the law of $Y_L$.
Equation~\eqref{eq:max-tail} bounds the expectation of the first
maximum by $C\log(q+2)\mu_L$. By \eqref{eq:unit-scaling} and
\eqref{eq:refinement}, each variable in the second maximum has
expectation $f_{t_j-t_i}\le2f_N$. Since there are at most $q^2$ of them,
bounding their maximum by their sum proves \eqref{eq:grid}.
\end{proof}

The following proposition identifies the expectation and moment
exponents by combining the comparison in Lemma~\ref{lem:grid} with the
upper tail estimate in Proposition~\ref{prop:upper-tail}.
\begin{proposition}\label{prop:finite-exponent}
The exponent $\theta$ in \eqref{eq:envelope-exponent}
belongs to $[1/2,1]$, and
\begin{equation}\label{eq:finite-exponent}
 f_N=N^{\theta+o(1)},\qquad g_N=N^{\theta+o(1)}.
\end{equation}
In fact, for every $p>0$,
\[
 \E[L_N^p]=N^{p\theta+o(1)},\qquad
 \E[V_N^p]=N^{p\theta+o(1)}.
\]
\end{proposition}
\begin{proof}
We first show that $\theta\ge1/2$. For unit time cells, the reflection
principle and the Gaussian tail bound~\cite[Theorem 2.21 and
Remark 2.22]{MortersPeres} give
\[
 \Pp\left(\max_{0\le i<N}\diam W([i,i+1])>
                  A\sqrt{\log(N+2)}\right)\longrightarrow0
\]
for a sufficiently large constant $A$. A graph path from vertex $0$ to
vertex $N-1$ in $\calG[0,N]$, with $q$ edges, connects $W_0$ to $W_N$
through $q+1$ cell ranges. On the event that every cell diameter is at most
$A\sqrt{\log(N+2)}$, it follows that
$|W_N-W_0|\le(q+1)A\sqrt{\log(N+2)}$.
Since $\Pp(|W_N-W_0|\ge\sqrt N)$ is a positive constant independent of
$N$, the intersection of these two events has probability bounded away
from zero for large $N$. Hence $f_N\ge c\sqrt{N/\log(N+2)}$.
The inequality $f_N\le\mu_N$ proves $\theta\ge1/2$.

Fix $\varepsilon\in(0,1)$ and take $q=\lceil N^\varepsilon\rceil$ in
\eqref{eq:grid}. The term $C\log(q+2)\mu_L$ is
$N^{(1-\varepsilon)\theta+o(1)}=o(\mu_N)$. Thus
\[
 f_N\ge c\mu_N/q^2=N^{\theta-2\varepsilon+o(1)}.
\]
Let $\varepsilon\downarrow0$, and use $f_N\le g_N\le\mu_N$ to obtain
\eqref{eq:finite-exponent}.

Concerning the moment estimates, integrating
\eqref{eq:upper-tail} gives
$\E[Y_N^p]\le C_p\mu_N^p=N^{p\theta+o(1)}$
for every $p>0$, yielding the required upper bounds. For the matching lower bound, the Paley--Zygmund
inequality~\cite[Exercise 3.5]{MortersPeres} applied to $1+L_N$ gives
\[
 \Pp(1+L_N\ge f_N/2)
 \ge\frac{f_N^2}{4\E[(1+L_N)^2]}
 \ge c(f_N/\mu_N)^2=N^{-o(1)}.
\]
Restricting the expectation to this event gives, for every $p>0$,
\[
 \E[(1+L_N)^p]\ge (f_N/2)^p\Pp(1+L_N\ge f_N/2)
                \ge N^{p\theta-o(1)}.
\]
Finally, $L_N\le1+L_N\le2L_N$ for $N\ge2$,
and the same bounds hold for $V_N$.
Together with the upper bounds and $L_N\le V_N$,
this proves the claimed moment asymptotics.
\end{proof}

\subsection{Choice of normalization}\label{subsec:normalization}

To construct subsequential limits of normalized graph
distances in Section~\ref{sec:limits}, we need a
deterministic subsequence and suitable normalization
constants. The following lemma provides both, together
with uniform upper bounds on the normalized expected
interval maxima at smaller scales and a positive lower
bound on the expected normalized endpoint distance.

\begin{lemma}\label{lem:records}
There is a deterministic increasing sequence of powers of two
$n_j\to\infty$ with the following properties. Set
\begin{equation}\label{eq:record-normalization}
 \kappa_j=\mu_{n_j}.
\end{equation}
There are deterministic exponents
$\theta/2\le\rho_j<\theta$ with $\rho_j\uparrow\theta$
such that
\begin{equation}\label{eq:record-variable}
 \mu_l\le2\kappa_j(l/n_j)^{\rho_j},
       \qquad 1\le l\le n_j,\quad j\ge1.
\end{equation}
For every $0<\rho<\theta$ there is a deterministic $C_\rho<\infty$
such that, for every $j$,
\begin{equation}\label{eq:record-local}
 \mu_l\le C_\rho\kappa_j(l/n_j)^\rho,
       \qquad 1\le l\le n_j.
\end{equation}
Moreover, there is a deterministic $c_*>0$ such that, for every $j$,
\begin{equation}\label{eq:record-positive}
 \E L_{n_j}\ge c_*\kappa_j.
\end{equation}
In particular $\kappa_j=n_j^{\theta+o(1)}\to\infty$, and $\kappa_j$ is
comparable to each of $\E L_{n_j}$ and $\E V_{n_j}$ on this sequence.
\end{lemma}

\begin{proof}
Choose $\rho_j\uparrow\theta$ with $\theta/2\le\rho_j<\theta$.
By \eqref{eq:envelope-exponent}, the sequence
$\mu_{2^k}/2^{k\rho_j}$ is unbounded for each $j$.
Starting from $n_0=1$, choose $n_j=2^{k_j}>n_{j-1}$
sufficiently large that
\[
 \frac{\mu_{2^{k_j}}}{2^{k_j\rho_j}}
 =\max_{0\le k\le k_j}\frac{\mu_{2^k}}{2^{k\rho_j}}.
\]
Given $l\le n_j$, let $2^k$ be the least power of two
not smaller than $l$. The choice of $n_j$ gives
$\mu_l\le\mu_{2^k}\le\kappa_j(2^k/n_j)^{\rho_j}$.
Since $2^k<2l$ and $\rho_j\le1$, this gives
\eqref{eq:record-variable}.

Choose a fixed power of two $q$ such that
$2C\log(q+2)q^{-\theta/2}\le1/2$, where $C$ is the
constant in \eqref{eq:grid}. Since the record indices
can be chosen arbitrarily large, we may assume
$n_j\ge q$ for every $j$. Then $q$ divides $n_j$,
and \eqref{eq:record-variable} bounds the first term
of \eqref{eq:grid} by
$2C\log(q+2)\kappa_jq^{-\rho_j}\le\kappa_j/2$.
The other term must therefore be at least $\kappa_j/2$,
giving $f_{n_j}\ge\kappa_j/(4q^2)$.
Since $\kappa_j\to\infty$, after discarding finitely
many terms and relabeling the sequences, we have
\[
 \E L_{n_j}=f_{n_j}-1\ge\kappa_j/(8q^2)
\]
for every $j$. This proves \eqref{eq:record-positive}.

Finally, for fixed $\rho<\theta$, eventually
$\rho_j\ge\rho$, so \eqref{eq:record-variable}
implies \eqref{eq:record-local}. Increasing $C_\rho$
if necessary covers the finitely many remaining indices.
The exponent and the stated comparisons follow from
\eqref{eq:envelope-exponent}, \eqref{eq:record-positive},
and $f_{n_j}\le g_{n_j}\le\mu_{n_j}$.
\end{proof}

\section{Curves of minimum dimension}\label{sec:curves}

In this section we prove that $\mathfrak p_3(W)=2\theta$ and construct
continuum many simple curves attaining this minimum, with the
regularity asserted in Theorem~\ref{thm:main-growth}. We first build
curves whose images have upper box dimension at most $2\theta$
(Proposition~\ref{prop:curve-upper}), then obtain the matching lower
bound using compatible H\"older functions
(Lemma~\ref{lem:dimension-test} and Proposition~\ref{prop:plateau}).
Section~\ref{subsec:dimension-proof} combines these results to
establish the dimension identity and proves the almost sure growth
exponents of $L_n$ and $V_n$ using the compatible functions and the
estimates from Section~\ref{sec:growth}.
Finally, local replacements of the minimizing curves establish
nonuniqueness (Proposition~\ref{prop:nonunique-curves}).

\subsection{Uniform dimension doubling}\label{subsec:dimension-doubling}

We begin by recalling the dimension doubling theorem that transfers
dimension estimates from Brownian time to the range. Since the time
sets in our constructions depend on $W$, we use the simultaneous form
of Kaufman's theorem~\cite{Kaufman}: on an event of probability one,
\begin{equation}\label{eq:kaufman}
 \dimH W(A)=2\dimH A\qquad
            \text{for every }A\subset[0,1].
\end{equation}
Here we use the formulation in
\cite[Theorem 9.28]{MortersPeres}. Brownian scaling gives the same assertion
on any fixed finite time interval.

\subsection{Construction of connecting curves}\label{subsec:curve-construction}

We first give a lemma that removes loops from a continuous curve while
controlling its $p$-variation. For a continuous curve
$f:[a,b]\to\R^3$ and $p\ge1$, write
\begin{equation}\label{eq:p-variation}
 \|f\|_{p\text{-var};[a,b]}
 =\left(\sup_{a=t_0<\cdots<t_m=b}
       \sum_{i=1}^m|f(t_i)-f(t_{i-1})|^p\right)^{1/p}.
\end{equation}
This seminorm is invariant under nondecreasing surjective
reparametrizations and satisfies the triangle inequality for sums
and concatenations of curves. We omit the interval when it is $[0,1]$.

\begin{lemma}\label{lem:holder-arc}
Let $q\ge1$, and let $f:[0,1]\to\R^3$ be continuous with
$f(0)\ne f(1)$. Suppose that $\|f\|_{p\text{-var}}<\infty$
for every $p>q$. For any sequence $p_n\downarrow q$ with $p_n>q$,
there exists a simple curve $\eta:[0,1]\to f([0,1])$ satisfying
$\eta(0)=f(0)$ and $\eta(1)=f(1)$, and
\[
 \|\eta\|_{p\text{-var}}\le\|f\|_{p\text{-var}}
 \qquad(p\ge1),
\]
\[
 [\eta]_{1/p_n}^{p_n}
 \le 2^n\bigl(1+\|f\|_{p_n\text{-var}}^{p_n}\bigr)
 \qquad(n\ge1).
\]
In particular, $\eta$ is $\beta$-H\"older for every
$0<\beta<1/q$. The conclusion also applies if $f$ is
$\beta$-H\"older for every $0<\beta<1/q$.
\end{lemma}
\begin{proof}
By the loop-cancellation construction of Brazas~\cite{Brazas},
there is a continuous curve $g:[0,1]\to f([0,1])$ satisfying
$g(0)=f(0)$ and $g(1)=f(1)$ with the following property:
if $g(s)=g(t)$ for $s<t$, then $g$ is constant on $[s,t]$.
The curve $g$ is obtained by choosing pairwise disjoint closed
intervals $[a,b]$ with $f(a)=f(b)$ and replacing $f$ on each of
them by the constant value $f(a)$. More precisely, define
$\sigma:[0,1]\to[0,1]$ by $\sigma(r)=a$ when $r\in(a,b)$ for
one of these intervals, and $\sigma(r)=r$ otherwise. Then
$\sigma$ is nondecreasing and $g(r)=f(\sigma(r))$.
For every partition $0=t_0<\cdots<t_m=1$, the sequence
$\sigma(t_0),\ldots,\sigma(t_m)$ is nondecreasing. Removing
repeated times if necessary gives
\[
 \sum_{i=1}^m|g(t_i)-g(t_{i-1})|^p
 =\sum_{i=1}^m|f(\sigma(t_i))-f(\sigma(t_{i-1}))|^p
 \le\|f\|_{p\text{-var}}^p.
\]
Consequently,
\[
 \|g\|_{p\text{-var}}\le\|f\|_{p\text{-var}}
 \qquad(p\ge1).
\]

Fix a sequence $p_n\downarrow q$ with $p_n>q$, and put
\[
 v_n(t)=\|g\|_{p_n\text{-var};[0,t]}^{p_n},\qquad v_n(0)=0.
\]
For $q<r<p_n$, the oscillation of $g$ on $[s,t]$ gives
\[
 \|g\|_{p_n\text{-var};[s,t]}^{p_n}
 \le \left(\sup_{u,v\in[s,t]}|g(u)-g(v)|\right)^{p_n-r}
       \|g\|_{r\text{-var}}^r.
\]
Since $g$ is continuous, the right-hand side tends to zero uniformly as
$t-s\downarrow0$. The inequality
\[
 0\le \|g\|_{p_n\text{-var};[0,t]}
       -\|g\|_{p_n\text{-var};[0,s]}
 \le \|g\|_{p_n\text{-var};[s,t]}
 \qquad(s<t)
\]
shows that $v_n$ is continuous. Combining partitions also gives
\[
 v_n(t)-v_n(s)\ge
 \|g\|_{p_n\text{-var};[s,t]}^{p_n}
 \ge |g(t)-g(s)|^{p_n}\qquad(s<t).
\]
In particular, $v_n$ is constant wherever $g$ is constant. The
uniformly convergent series
\[
 h(t)=\sum_{n\ge1}2^{-n}\frac{v_n(t)}{1+v_n(1)}
\]
is continuous and nondecreasing, with $h(0)=0$ and $0<h(1)\le1$.
If $g(s)=g(t)$ for $s<t$, the property of $g$ established above
implies that all $v_n$ are constant on $[s,t]$, and hence
$h(s)=h(t)$. Conversely, $h(s)=h(t)$ implies
$v_n(s)=v_n(t)$ for every $n$, so the preceding increment bound
gives $g(s)=g(t)$.

Set $\phi(t)=h(t)/h(1)$, which is continuous, nondecreasing,
and maps $[0,1]$ onto $[0,1]$. For $u\in[0,1]$, choose a time
$t$ with $\phi(t)=u$ and define $\eta(u)=g(t)$.
This defines a function because $\phi(s)=\phi(t)$ implies
$g(s)=g(t)$. For any closed set $F\subset\R^3$,
\[
 \eta^{-1}(F)=\phi(g^{-1}(F))
\]
is compact, so $\eta$ is continuous. If $\eta(u)=\eta(v)$,
choose $s,t$ with $\phi(s)=u$ and $\phi(t)=v$.
Then $g(s)=g(t)$, which implies $\phi(s)=\phi(t)$ and hence
$u=v$. Thus $\eta$ is injective. Furthermore,
$\eta(0)=f(0)$, $\eta(1)=f(1)$, and
$\eta([0,1])=g([0,1])\subset f([0,1])$.

For $u<v$, choose $s\le t$ with $\phi(s)=u$ and $\phi(t)=v$.
Then
\[
\begin{aligned}
 |\eta(v)-\eta(u)|^{p_n}
 &\le v_n(t)-v_n(s)\\
 &\le 2^n(1+v_n(1))\bigl(h(t)-h(s)\bigr)\\
 &\le 2^n\bigl(1+\|f\|_{p_n\text{-var}}^{p_n}\bigr)|v-u|.
\end{aligned}
\]
This proves the stated H\"older estimate for every $n$.
Since $p_n\downarrow q$, for any $0<\beta<1/q$ we can choose
$n$ with $\beta<1/p_n$, and then
$[\eta]_\beta\le[\eta]_{1/p_n}<\infty$.

To compare $p$-variation, take any partition
$0=u_0<\cdots<u_m=1$. We can choose
$0=t_0\le\cdots\le t_m=1$ with $\phi(t_i)=u_i$, and hence
\[
 \sum_{i=1}^m|\eta(u_i)-\eta(u_{i-1})|^p
 =\sum_{i=1}^m|g(t_i)-g(t_{i-1})|^p
 \le\|g\|_{p\text{-var}}^p.
\]
Conversely, applying $\phi$ to a partition for $g$ and
removing repeated values gives the reverse inequality.
Therefore $\|\eta\|_{p\text{-var}}=\|g\|_{p\text{-var}}$,
which proves the required bound.
Finally, if $f$ is $\beta$-H\"older with $\beta p>1$, then
$\|f\|_{p\text{-var}}\le[f]_\beta$.
Thus the H\"older hypothesis in the last assertion implies
finite $p$-variation for every $p>q$.
\end{proof}

We construct connecting curves by successively refining Brownian time
cells and using shortest paths in the intersection graphs of their
subcells. The next
proposition gives the resulting dimension and H\"older bounds uniformly
over the endpoints.
\begin{proposition}\label{prop:curve-upper}
On an event of probability one, for every $x,y\in K$, there exist
a compact set $A\subset[0,1]$ and a continuous curve
$\eta:[0,1]\to K$ such that
\begin{equation}\label{eq:curve-upper}
 \begin{gathered}
 \eta(0)=x,\qquad\eta(1)=y,\qquad \eta([0,1])\subset W(A),\\
 \dimUB A\le\theta,\qquad
 \dimUB\eta([0,1])\le2\theta.
 \end{gathered}
\end{equation}
The curve $\eta$ is $\beta$-H\"older for every
$0<\beta<(2\theta)^{-1}$.
For each such $\beta$, there is a nonnegative random variable
$H_\beta$ with $\E[H_\beta^p]<\infty$ for every $p>0$
such that $[\eta]_\beta\le H_\beta$ uniformly over $x,y\in K$.
When $x\ne y$, $\eta$ can be chosen simple with these properties.
\end{proposition}

\begin{proof}
To make the estimates hold at sufficiently fine levels, set
$r_j=\lceil e^{\sqrt j}\rceil$, $N_0=1$ and
$N_j=\prod_{l=1}^j r_l$. Since $\mu_r=r^{\theta+o(1)}$, choose
deterministic $\varepsilon_j\downarrow0$ such that
\[
 \mu_{r_j}\le r_j^{\theta+\varepsilon_j/2},\qquad
 \varepsilon_j\sqrt j\ge6\log(j+1).
\]
Then $\log N_j\sim(2/3)j^{3/2}$ and
$r_j^{\varepsilon_j/2}\ge(j+1)^3$, so
\begin{equation}\label{eq:upper-scales}
 CN_{j-1}\exp\{-c r_j^{\varepsilon_j/2}\}
       \le C e^{-cj^3}.
\end{equation}
At level $j$, divide each interval of the preceding partition into $r_j$
equal subintervals. Within any fixed interval, their intersection graph
has the law of $G_{r_j}$. Equation~\eqref{eq:upper-tail} and a union bound
over the $N_{j-1}$ intervals show that the probability that any of these
graphs has diameter plus one greater than $r_j^{\theta+\varepsilon_j}$ is
at most $Ce^{-cj^3}$. These probabilities are summable, so the
Borel--Cantelli lemma gives all these bounds for sufficiently large
$j$, on a single event of probability one that does not depend on $x,y$.
Let $J$ be the last level at which a bound fails, with $J=0$ if none
fails. Then $\Pp(J\ge j)\le Ce^{-cj^3}$ and
\begin{equation}\label{eq:curve-failure-moments}
 \E[N_J^p]<\infty\qquad(p>0).
\end{equation}

Fix a Brownian path with $J<\infty$ that is H\"older continuous
with every exponent below $1/2$, and choose $s,t\in[0,1]$ with
$W_s=x$ and $W_t=y$. We define
$f^{(0)}(v)=W_{(1-v)s+vt}$ for $0\le v\le1$,
where $v$ is the parameter of the connecting curve.
At each subsequent step, the parameter interval is divided into
finitely many subintervals, each assigned to a Brownian time cell
so that the corresponding curve segment lies in its range.
Initially the whole parameter interval is assigned to $[0,1]$.
Distinct parameter subintervals may be assigned to the same time
cell, and these assignments are counted separately.

At step $j$, consider a parameter subinterval assigned to a time
cell $I$ of length $N_{j-1}^{-1}$, and denote the values of
$f^{(j-1)}$ at its endpoints by $z,z'$.
Representatives of $z,z'$ are chosen in $I$, together with the
smaller cells containing them. A shortest path in the intersection
graph of the $r_j$ subcells of $I$ determines the subcells to be
used. For each edge, we choose a common point of the two incident
Brownian ranges. Within each visited subcell, the curve follows
$W$ between the chosen points, reversing Brownian time if necessary.
Concatenating these pieces gives a curve from $z$ to $z'$ contained
in $W(I)$. The original parameter subinterval is divided into
equal parts, each assigned to a visited subcell, and the
corresponding curve segments are parametrized on these parts.
Repeating this construction on every parameter subinterval
defines $f^{(j)}$, whose values at all previously chosen division
points remain unchanged.

We first bound the number of parameter subintervals, and hence the
number of time cells used. Let $M_j$ count the parameter subintervals
after step $j$.
At every sufficiently large $j$, each old segment is replaced by at
most $r_j^{\theta+\varepsilon_j}$ new segments, since a shortest graph
path visits one more vertex than its number of edges. Hence
\[
 M_j\le N_J\prod_{l=1}^j r_l^{\theta+\varepsilon_l},
 \qquad
 \limsup_{j\to\infty}\frac{\log M_j}{\log N_j}\le\theta.
\]
The exceptional levels contribute at most $N_J$. The limit follows
from $\varepsilon_l\to0$ and
$\sum_{l\le j}\log r_l=\log N_j$.

The construction also gives nested time sets: if $A_j$ is the union
of the selected closed time cells at level $j$, then
$A_{j+1}\subset A_j$. For a parameter in a level-$j$ curve segment,
every later curve value stays in the Brownian range of its assigned
level-$j$ time cell.
At an endpoint of such a segment, the value is fixed at every later
level.
Consequently, for $l,l'\ge j$,
\[
 \|f^{(l)}-f^{(l')}\|_\infty
 \le\sup_{\substack{s,t\in[0,1]\\|s-t|\le N_j^{-1}}}
 |W_s-W_t|\longrightarrow0.
\]
The limit $f$ is continuous, with endpoints $x,y$. Set
$A=\bigcap_jA_j$. Compactness and nestedness give
\[
 f([0,1])\subset\bigcap_jW(A_j)=W(A).
\]
To pass from the cell count to dimension, let $\mathcal N(A,r)$ be the
least number of sets of diameter at most $r$ needed to cover $A$.
At level $j$, this set is covered by at most
$M_j$ intervals of length $N_j^{-1}$.
If $N_{j+1}^{-1}\le r<N_j^{-1}$, then
$\mathcal N(A,r)\le M_{j+1}$. Since
$\log N_{j+1}/\log N_j\to1$, the bound on $M_j$ gives
\[
 \dimUB A=\limsup_{r\downarrow0}
 \frac{\log\mathcal N(A,r)}{\log(1/r)}
 \le\limsup_{j\to\infty}\frac{\log M_{j+1}}{\log N_j}
 \le\theta.
\]

We use the equal divisions of the curve parameter to control its
H\"older regularity. Write
\[
 b_j=\prod_{l=1}^j r_l^{-(\theta+\varepsilon_l)},\qquad b_0=1.
\]
Each level-$j$ parameter interval has length at least
$N_J^{-1}b_j$. Indeed, at each level $l$ where the diameter
bound holds, an interval is divided into at most
$r_l^{\theta+\varepsilon_l}$ subintervals, while the product
of the subdivision factors at the exceptional levels is
at most $N_J$.
Also $b_j=N_j^{-\theta-o(1)}$.

Let $[W]_\xi$ be the H\"older seminorm on $[0,1]$, where $0<\xi<1/2$.
For parameter times $u,v$ with $|u-v|\le N_J^{-1}b_j$, the interval
between them is contained in the union of at most two adjacent
level-$j$ parameter intervals. Their images lie in the ranges of
the assigned Brownian time cells, so
\[
 |f(u)-f(v)|\le2[W]_\xi N_j^{-\xi}.
\]
Fix $0<\beta<1/(2\theta)$ and choose $\beta\theta<\xi<1/2$.
Since $b_j=N_j^{-\theta-o(1)}$ and
$\log N_{j+1}/\log N_j\to1$, we have
\[
 C_\beta:=
 \max\left\{1,\,
 2\sup_{j\ge0}N_j^{-\xi}b_{j+1}^{-\beta}\right\}
 <\infty.
\]
Choosing $j$ with
$N_J^{-1}b_{j+1}<|u-v|\le N_J^{-1}b_j$
therefore yields
\begin{equation}\label{eq:curve-holder}
 |f(u)-f(v)|\le K_\beta|u-v|^\beta,\qquad
 K_\beta=C_\beta[W]_\xi N_J^\beta.
\end{equation}
For $|u-v|>N_J^{-1}$, the same bound follows from
$\operatorname{diam}K\le[W]_\xi$.
Gaussian increment moments and dyadic chaining give every positive
moment of $[W]_\xi$; together with
\eqref{eq:curve-failure-moments} and the Cauchy--Schwarz inequality,
this gives $\E[K_\beta^p]<\infty$ for every $p>0$.
The bound $K_\beta$ is independent of $x,y$, and the construction
of $f$ does not depend on $\beta$.
These H\"older bounds also give
$\dimUB f([0,1])\le2\theta$.

It remains to obtain a simple curve when $x\ne y$.
The H\"older bounds above imply that
$\|f\|_{p\text{-var}}<\infty$ for every $p>2\theta$.
Fix a deterministic sequence $p_n\downarrow2\theta$ with $p_n>2\theta$.
Lemma~\ref{lem:holder-arc} gives a simple curve $\eta$ with
$\eta(0)=x$, $\eta(1)=y$ and
$\eta([0,1])\subset f([0,1])\subset W(A)$.
Its upper box dimension is therefore at most $2\theta$.
For each $n$, choose a deterministic exponent
$1/p_n<\beta_n<1/(2\theta)$. Since
$\|f\|_{p_n\text{-var}}\le[f]_{\beta_n}\le K_{\beta_n}$,
the estimate in Lemma~\ref{lem:holder-arc} gives
\[
 [\eta]_{1/p_n}
 \le\bigl[2^n(1+K_{\beta_n}^{p_n})\bigr]^{1/p_n}.
\]
For every $0<\beta<1/(2\theta)$, there is $n$ with
$\beta<1/p_n$. As the parameter interval is $[0,1]$,
the same bound applies to $[\eta]_\beta$. It is uniform over
$x,y$ and has moments of every positive order.
When $x=y$, take the constant curve.
\end{proof}

The random variable $\mathfrak p_3(W)$ is measurable with respect
to the completion of the sigma-field generated by $W$.
Indeed, for continuous paths $w,\eta$, the conditions
$\eta(0)=w(0)$, $\eta(1)=w(1)$ and
$\eta([0,1])\subset w([0,1])$ define a closed subset of the
product path space. The map
$\eta\mapsto\dimH\eta([0,1])$ is Borel measurable, as can be
seen from finite covers by balls with rational centres and radii.
For each real $a$, the set of paths $w$ admitting such a curve
$\eta$ with $\dimH\eta([0,1])<a$ is therefore the projection
of a Borel set. This projection is analytic and hence measurable
with respect to the completion of Wiener measure
\cite[Exercise 14.3 and Theorem 21.10]{Kechris}.

\subsection{Compatible H\"older functions}\label{subsec:holder-functions}

To prove $\mathfrak p_3(W)\ge2\theta$, we use $\alpha$-H\"older
functions on Brownian time, with $0<\alpha<\theta$, that take equal
values whenever $W$ visits the same spatial point. We first derive a
dimension lower bound from such functions and then construct one that
separates $W_0$ and $W_1$ with positive probability.
A function $u:[0,T]\to\R$ is compatible with a continuous path
$w:[0,T]\to\R^3$ if
\begin{equation}\label{eq:compatible}
 w(s)=w(t)\quad\Longrightarrow\quad u(s)=u(t).
\end{equation}
For a compatible function $u$, define $\bar u:w([0,T])\to\R$ by
$\bar u(x)=u(t)$ for any $t\in[0,T]$ with $w(t)=x$.
Equation~\eqref{eq:compatible} ensures that this does not depend
on the choice of $t$. If $u$ is continuous, then so is $\bar u$.
Indeed, for any closed set $F\subset\R$,
\[
 \bar u^{-1}(F)=w(u^{-1}(F)),
\]
and the right-hand side is compact, hence closed in $w([0,T])$.

On the event \eqref{eq:kaufman}, a compatible $\alpha$-H\"older
function whose induced function $\bar u$ is nonconstant on a connected
set $C\subset K$ forces $\dimH C\ge2\alpha$:
\begin{lemma}\label{lem:dimension-test}
Suppose \eqref{eq:kaufman} holds, $0<\alpha<1$, and
$u:[0,1]\to\R$ is compatible with $W$ and $\alpha$-H\"older.
Let $\bar u(W_t)=u(t)$. If $C\subset K$ is connected and
$\bar u$ is nonconstant on $C$, then
\begin{equation}\label{eq:dimension-test}
 \dimH C\ge2\alpha.
\end{equation}
\end{lemma}
\begin{proof}
Set $A=\{t\in[0,1]:W_t\in C\}$. Since $C$ is connected and
$\bar u$ is continuous and nonconstant on $C$, the image
$u(A)=\bar u(C)$ is a nondegenerate interval.
The dimension bound for H\"older maps
\cite[Remark 4.12]{MortersPeres} gives
\[
 1=\dimH u(A)\le\frac{\dimH A}{\alpha}.
\]
As $W(A)=C$, equation~\eqref{eq:kaufman} yields
$\dimH C=2\dimH A\ge2\alpha$.
\end{proof}

We can apply this criterion once a compatible function separates the
endpoints. The following proposition provides such a function with
positive probability by taking limits of normalized distances from the
initial cell.
\begin{proposition}\label{prop:scalar-seed}
For every $0<\alpha<\theta$ there are deterministic constants
$c_0,M,p_0>0$ such that, with probability at least $p_0$,
there exists a continuous function $u:[0,1]\to\R$ compatible
with $W$ and satisfying
\begin{equation}\label{eq:scalar-seed}
 u(0)=0,\qquad u(1)\ge c_0,\qquad [u]_{\alpha;[0,1]}\le M.
\end{equation}
\end{proposition}
\begin{proof}
Fix $\alpha<\rho<\theta$ and take $(n_j)$ and $(\kappa_j)$ from
Lemma~\ref{lem:records}. For each $j$, define a piecewise linear
function $u_j:[0,1]\to\R$ by
\[
 u_j(i/n_j)=\kappa_j^{-1}d_{G_{n_j}}(0,i),\quad 0\le i<n_j,
 \qquad u_j(1)=u_j((n_j-1)/n_j).
\]
Note that $u_j(0)=0$ and $u_j(1)=L_{n_j}/\kappa_j$.

For integers $0\le a<b<n_j$, the graph triangle inequality gives
\[
 |u_j(b/n_j)-u_j(a/n_j)|
 \le\kappa_j^{-1}d_{G_{n_j}}(a,b).
\]
The graph distance on the right is at most the diameter of the
subgraph induced by the consecutive vertices $a,\ldots,b$.
By Brownian scaling and translation invariance, the interval
maximum for this restriction has the law of $Y_{b-a+1}$.
Proposition~\ref{prop:upper-tail} and \eqref{eq:record-local}
therefore give, for every $p\ge1$,
\[
 \|u_j(b/n_j)-u_j(a/n_j)\|_{L^p}
 \le C_{p,\rho}\left(\frac{b-a+1}{n_j}\right)^\rho.
\]
For $s,t\in[0,1]$ with $h=|t-s|\ge n_j^{-1}$, linear
interpolation expresses $u_j(s)$ and $u_j(t)$ as convex combinations
of values $u_j(i/n_j)$ from a block of at most
$\min\{n_j,\lceil n_jh\rceil+3\}$ consecutive vertices.
The preceding estimate therefore gives
\[
 \|u_j(t)-u_j(s)\|_{L^p}\le C_{p,\rho}h^\rho.
\]
If $h<n_j^{-1}$, consecutive vertices of $G_{n_j}$ are adjacent,
so the linear interpolation has Lipschitz constant at most
$n_j/\kappa_j$. Equation~\eqref{eq:record-local} with $l=1$
gives $\kappa_j\ge c_\rho n_j^\rho$, and hence
\[
 |u_j(t)-u_j(s)|\le C_\rho n_j^{1-\rho}h
 \le C_\rho h^\rho.
\]
We have proved, for all $s,t\in[0,1]$ and $p\ge1$,
\begin{equation}\label{eq:scalar-increments}
 \E|u_j(t)-u_j(s)|^p\le C_{p,\rho}|t-s|^{p\rho},
\end{equation}
with constants independent of $j$.

At dyadic level $k$, equation~\eqref{eq:scalar-increments} gives
\[
 \left\|\max_{0\le l<2^k}
 |u_j((l+1)2^{-k})-u_j(l2^{-k})|\right\|_{L^p}
 \le C_{p,\rho}2^{-k(\rho-1/p)}.
\]
Dyadic approximation of pairs of times and telescoping of the
increments bound $[u_j]_{\alpha;[0,1]}$ by a constant times
\[
 \sum_{k\ge0}2^{k\alpha}
 \max_{0\le l<2^k}
 |u_j((l+1)2^{-k})-u_j(l2^{-k})|.
\]
Choosing $p>1$ with $\rho-1/p>\alpha$ and applying Minkowski's
inequality, we obtain
\[
 \sup_j\E\bigl[[u_j]_{\alpha;[0,1]}^p\bigr]<\infty.
\]
Moreover, \eqref{eq:record-positive} and the bound on second moments
from Proposition~\ref{prop:upper-tail} give
\[
 \E u_j(1)\ge c_*>0,\qquad \sup_j\E[u_j(1)^2]<\infty.
\]
The Paley--Zygmund inequality therefore gives deterministic
$c_0,p_1>0$ such that $\Pp(u_j(1)\ge c_0)\ge p_1$ for every $j$.
By the uniform moment bound on $[u_j]_{\alpha;[0,1]}$, we may
choose a deterministic $M<\infty$ such that
\[
 \Pp\bigl(u_j(1)\ge c_0,\ [u_j]_{\alpha;[0,1]}\le M\bigr)
 \ge p_1/2\qquad(j\ge1).
\]
For $j\ge1$, set
$E_j=\{u_j(1)\ge c_0,\ [u_j]_{\alpha;[0,1]}\le M\}$.
Then $\Pp(\limsup_j E_j)\ge\limsup_j\Pp(E_j)\ge p_1/2$.
On each path in $\limsup_j E_j$, the functions from the
corresponding subsequence vanish at $0$, have H\"older seminorm
at most $M$, and are therefore uniformly bounded and
equicontinuous. By the Arzel\`a--Ascoli theorem
\cite[Theorem 7.25]{Rudin}, a further subsequence converges
uniformly to a function $u$ satisfying \eqref{eq:scalar-seed}.
The subsequence may depend on $W$.

It remains to verify compatibility. If $W_s=W_t$, the cells with
indices $\iota_{n_j}(s)$ and $\iota_{n_j}(t)$ have intersecting
Brownian ranges and hence are either identical or adjacent in
$G_{n_j}$. Consecutive grid values of $u_j$ differ by at most $1/\kappa_j$.
Thus, by linear interpolation and the definition of $u_j(1)$,
\[
 |u_j(r)-u_j(\iota_{n_j}(r)/n_j)|\le1/\kappa_j
 \qquad(r\in[0,1]).
\]
The graph triangle inequality now gives
\[
 |u_j(s)-u_j(t)|\le3/\kappa_j\longrightarrow0.
\]
This estimate holds simultaneously for all $s,t\in[0,1]$ with
$W_s=W_t$. Since $u_j$ converges uniformly along the selected
subsequence, the limit $u$ is compatible with $W$.
Taking $p_0=p_1/2$ proves the proposition.
\end{proof}

\subsection{Pseudometrics and compatible H\"older functions}\label{subsec:holder-chains}

Before applying a zero--one law, we define a pseudometric by taking
the infimum of costs over finite chains. We use it to express the
existence of compatible H\"older functions as a measurable event in
Section~\ref{subsec:scalar-zero-one}.
For $T>0$, $0<\alpha<1$ and $w\in C([0,T],\R^3)$, put
\begin{equation}\label{eq:holder-chain}
 \delta_{\alpha,T}^w(s,t)=
 \inf\left\{\sum_{l=1}^m|a_l-b_l|^\alpha:
 \begin{array}{l}
 m\ge1,\quad a_1=s,\quad b_m=t,\quad a_l,b_l\in[0,T],\\
 w(b_l)=w(a_{l+1})\quad(1\le l<m)
 \end{array}\right\}.
\end{equation}
The times in a chain need not be in increasing order. We write
$\delta_\alpha^w=\delta_{\alpha,1}^w$.

The following lemma expresses the pseudometric in terms of
compatible H\"older functions and establishes its measurability.
\begin{lemma}\label{lem:chain-dual}
The function $\delta_{\alpha,T}^w$ is a continuous pseudometric
on $[0,T]$. It satisfies
$\delta_{\alpha,T}^w(s,t)\le|s-t|^\alpha$ and
$\delta_{\alpha,T}^w(s,t)=0$ whenever $w(s)=w(t)$, and is the
largest pseudometric with these properties. Moreover,
\begin{equation}\label{eq:chain-dual}
 \delta_{\alpha,T}^w(s,t)
 =\sup\{|u(s)-u(t)|:[u]_{\alpha;[0,T]}\le1,\ u\text{ compatible with }w\}.
\end{equation}
The map $(w,s,t)\mapsto\delta_{\alpha,T}^w(s,t)$ is Borel measurable
on $C([0,T],\R^3)\times[0,T]^2$, where the path space has the uniform norm.
\end{lemma}
\begin{proof}
Reversing and concatenating chains prove symmetry and the triangle
inequality; a pair $(s,s)$ gives zero distance on the diagonal.
A single pair gives $\delta_{\alpha,T}^w(s,t)\le|s-t|^\alpha$,
which implies continuity by the triangle inequality. When $w(s)=w(t)$,
the two pairs $(s,s),(t,t)$ have total cost zero. For any function in
the supremum in \eqref{eq:chain-dual}, telescoping along a chain gives
$|u(s)-u(t)|\le\delta_{\alpha,T}^w(s,t)$. Conversely,
$u(v)=\delta_{\alpha,T}^w(s,v)$ is compatible and has H\"older seminorm
at most one, so it gives equality. The same telescoping argument for
any pseudometric bounded by $|s-t|^\alpha$ and zero on equal path
positions proves maximality.
For measurability, fix $m\ge1$. For each $c\in\R$, the set of
$(w,s,t)$ admitting a chain with $m$ pairs and total cost at most $c$
is closed. Indeed, the time parameters belong to the compact set
$[0,T]^{2m}$, and both the constraints and the cost are continuous
in these parameters and in $(w,s,t)$. Thus the minimum cost among
chains with $m$ pairs is Borel measurable. Taking the infimum over
$m\ge1$ proves the last assertion.
\end{proof}

In particular, the existence event in Proposition~\ref{prop:scalar-seed}
is $\{\delta_\alpha^W(0,1)\ge c_0/M\}$, so it is measurable.
To require a function to equal one on an entire terminal interval,
for $0<t\le T$ define
\[
  b_{\alpha,T}^w(t)=\inf_{r\in[t,T]}\delta_{\alpha,T}^w(0,r).
\]
Allowing the final time of the chain to vary over the compact interval
$[t,T]$ in the preceding proof shows that $b_{\alpha,T}^w(t)$ is Borel.
Write $b=b_{\alpha,T}^w(t)$. Then $b>0$ if and only if there is a compatible
$\alpha$-H\"older function $v:[0,T]\to[0,1]$ with $v(0)=0$ and $v=1$ on
$[t,T]$. Indeed, such a function with H\"older constant $H$ gives $b\ge1/H$
by \eqref{eq:chain-dual}. Conversely, if $b>0$, take
$v(s)=\min\{1,\delta_{\alpha,T}^w(0,s)/b\}$. Denote this existence event
for $w=W$ by $A_\alpha^T(t)$.

Similarly, let $A_\alpha(t)$ be the event that there is a compatible function
$v:[0,\infty)\to[0,1]$ with a finite global $\alpha$-H\"older
constant, $v(0)=0$, and $v=1$ on $[t,\infty)$.
To verify measurability, note that the pseudometrics
$\delta_{\alpha,M}^W$ decrease as the integer $M$ increases, because
more time pairs are allowed. For fixed $s,u\ge0$, let $M\to\infty$ through integers with
$M\ge\max\{s,u\}$, and write the limit as
$\delta_{\alpha,\infty}^W(s,u)$. It is a pseudometric bounded by
$|s-u|^\alpha$, and
\[
 \inf_{u\ge t}\delta_{\alpha,\infty}^W(0,u)
 =\inf_{M\in\N,\,M\ge t}b_{\alpha,M}^W(t).
\]
If this value is $b_\infty>0$, the function
$v(s)=\min\{1,\delta_{\alpha,\infty}^W(0,s)/b_\infty\}$ has the
required properties. Conversely, any such function with H\"older constant
$H$ gives $b_{\alpha,M}^W(t)\ge1/H$ for every integer $M\ge t$.
Consequently,
\begin{equation}\label{eq:infinite-plateau}
 A_\alpha(t)=
 \left\{\inf_{M\in\N,\,M\ge t}b_{\alpha,M}^W(t)>0\right\}.
\end{equation}
Thus $A_\alpha(t)$ is measurable.

\subsection{Transience and Blumenthal's zero--one law}\label{subsec:scalar-zero-one}

Proposition~\ref{prop:scalar-seed} gives a compatible H\"older function
with positive probability. We use the following avoidance lemma and
Brownian scaling to prove almost sure existence.
\begin{lemma}\label{lem:escape}
On an event of probability one, for every $0\le T<\infty$,
the complement of $W([0,T])$ in $\R^3$ is connected.
For every such $T$ and every $x\notin W([0,T])$, an independent
Brownian motion started at $x$ has positive probability of avoiding
$W([0,T])$ for all future times.
In particular, for deterministic $0<t<1$,
\begin{equation}\label{eq:escape}
 \Pp\bigl(W([1,\infty))\cap W([0,t])=\varnothing
                       \mid\mathcal F_1\bigr)>0\quad\text{a.s.}
\end{equation}
\end{lemma}
\begin{proof}
The projection of $W([0,N])$ onto a fixed coordinate plane has area
zero almost surely for every integer $N$; see
\cite[Theorem 2.24]{MortersPeres}. The same therefore holds for
every $K_0=W([0,T])$ with $T<\infty$. Fix $x\notin K_0$. In a ball about $x$ disjoint from $K_0$, there
is a point $y$ whose projection does not belong to the projection of
$K_0$. The segment from $x$ to $y$ lies outside $K_0$, and the line
through $y$ perpendicular to the projection plane also avoids $K_0$.
Hence $x$ can be joined, without meeting $K_0$, to the exterior of a
ball containing $K_0$. This proves connectedness.

If $K_0\subset B(0,R)$, there is a polygonal route in the complement
from $x$ to a point outside $B(0,2R)$. Brownian motion can follow this
route within a sufficiently narrow tube with positive probability.
This follows, for example, by subdividing the route into short pieces,
using positive Gaussian transition densities and the positive probability
of small oscillations over fixed time intervals. From outside
$B(0,2R)$ the motion avoids $B(0,R)$ with probability at least $1/2$, by
\cite[Corollary 3.19]{MortersPeres}. Finally, time reversal and polarity
of points~\cite[Corollary 2.26]{MortersPeres} give
$W_1\notin W([0,t])$ almost surely. The Markov property at time one
then proves \eqref{eq:escape}.
\end{proof}

The avoidance estimate, Brownian scaling and Blumenthal's zero--one
law give the following almost sure existence result.
\begin{proposition}\label{prop:plateau}
For every $0<\alpha<\theta$ and $t>0$,
\begin{equation}\label{eq:plateau}
 \Pp(A_\alpha(t))=1.
\end{equation}
\end{proposition}
\begin{proof}
Let $c_0,M$ be as in Proposition~\ref{prop:scalar-seed}, and fix a
deterministic $\varepsilon\in(0,1)$ with $M\varepsilon^\alpha<c_0/2$.
The event
$E=\{\delta_\alpha^W(0,1)\ge c_0/M\}$ has positive probability.
On $E$, the function $u(s)=M\delta_\alpha^W(0,s)$, $0\le s\le1$,
is compatible with $W$, satisfies $u(0)=0$, $u(1)\ge c_0$ and
$[u]_\alpha\le M$, by Lemma~\ref{lem:chain-dual}.
For $1-\varepsilon\le s\le1$,
$u(s)\ge c_0-M\varepsilon^\alpha>c_0/2$. Therefore
\[
 v(s)=\min\{1,2u(s)/c_0\},\qquad 0\le s\le1,
\]
is compatible, takes values in $[0,1]$, satisfies $v(0)=0$, and
equals one on $[1-\varepsilon,1]$. In particular,
$\Pp(A_\alpha^1(1-\varepsilon))>0$.

The event $A_\alpha^1(1-\varepsilon)$ is $\mathcal F_1$-measurable.
By Lemma~\ref{lem:escape}, the event
\[
 A_\alpha^1(1-\varepsilon)\cap
 \{W([1,\infty))\cap W([0,1-\varepsilon])=\varnothing\}
\]
has positive probability. On this event, a function $v$ satisfying
the conditions of $A_\alpha^1(1-\varepsilon)$ extends to
$[0,\infty)$ by setting
$v(s)=1$ for $s\ge1$. Compatibility is preserved because any
intersection between the future and $W([0,1])$ must occur at a point
visited during $[1-\varepsilon,1]$, where $v=1$.
The extension has a finite global H\"older constant: if
$s<1-\varepsilon$ and $r\ge1$, then
\[
 |v(r)-v(s)|=|v(1-\varepsilon)-v(s)|
 \le [v]_\alpha(1-\varepsilon-s)^\alpha
 \le[v]_\alpha(r-s)^\alpha;
\]
for $s\ge1-\varepsilon$ the difference is zero.
Consequently $p_\alpha:=\Pp(A_\alpha(1))>0$.

Brownian scaling gives $\Pp(A_\alpha(t))=p_\alpha$ for every $t>0$,
since a time rescaling preserves compatibility and multiplies a finite
global H\"older constant by a finite factor. Moreover,
$A_\alpha(t')\subset A_\alpha(t)$ when $t'<t$. Consequently
\[
 A_\alpha^*:=\bigcap_{k\ge1}A_\alpha(2^{-k})
 \quad\text{has probability }p_\alpha.
\]
The Brownian germ $\sigma$-field is
$\mathcal F_{0+}:=\bigcap_{T>0}\mathcal F_T$; it consists of events
measurable from the path on every initial time interval $[0,T]$.
To apply Blumenthal's zero--one law, we show that $A_\alpha^*$ belongs
to this $\sigma$-field. Transience and polarity
of points~\cite[Theorem 3.20 and
Corollary 2.26]{MortersPeres} give
$\inf_{u\ge T}|W_u|>0$ almost surely for each $T>0$. Continuity at time
$0$ then implies
\begin{equation}\label{eq:initial-disjoint-scalar}
 W([0,2^{-k}])\cap W([T,\infty))=\varnothing
                 \quad\text{for all sufficiently large }k.
\end{equation}
Fix $T>0$. On a realization satisfying
\eqref{eq:initial-disjoint-scalar}, the events
$A_\alpha^T(2^{-k})$ and $A_\alpha(2^{-k})$ agree for all
sufficiently large $k$.
Indeed, a function defined on the whole trajectory restricts to one on
$[0,T]$. In the other direction, a function for $A_\alpha^T(2^{-k})$
already equals one on $[2^{-k},T]$. Extending it by one after $T$
preserves compatibility, since the part before $2^{-k}$ has no
intersection with the path after $T$. It also preserves a finite
H\"older constant. Consequently, up to a null set,
\[
 A_\alpha^*=\liminf_{k\to\infty}A_\alpha^T(2^{-k})\in\mathcal F_T.
\]
Taking a countable sequence $T\downarrow0$ therefore places
$A_\alpha^*$ in $\mathcal F_{0+}$, with completion understood. Since
$p_\alpha>0$, Blumenthal's zero--one law
\cite[Theorem 2.7]{MortersPeres} gives $p_\alpha=1$.
\end{proof}

\subsection{Percolation dimension and distance growth}\label{subsec:dimension-proof}

Here we prove the dimension identity and almost sure distance exponents
in \eqref{eq:main-pathwise}, as well as \eqref{eq:main-lerw}.
We also obtain a simple curve attaining the minimum; the construction
of continuum many such curves is completed in
Section~\ref{subsec:nonuniqueness}.

For a finite nearest-neighbor path $\omega=(\omega_0,\ldots,\omega_m)$,
its \emph{chronological loop-erasure} is defined by the indices
\begin{equation}\label{eq:chronological-loop-erasure}
 \begin{aligned}
 j_0&=\max\{i:0\le i\le m,\ \omega_i=\omega_0\},\\
 j_{k+1}&=\max\{i:j_k+1\le i\le m,\ \omega_i=\omega_{j_k+1}\}
 \quad(j_k<m).
 \end{aligned}
\end{equation}
When $j_q=m$, the sequence $(\omega_{j_0},\ldots,\omega_{j_q})$
is the loop-erasure of $\omega$; it is self-avoiding and has the
same endpoints. Applying this construction to simple random walk
gives \emph{loop-erased random walk}~\cite{Lawler1980}.

\begin{proof}[Proof of \eqref{eq:main-pathwise} and \eqref{eq:main-lerw}]
The expectation and moment exponents follow from
Proposition~\ref{prop:finite-exponent}, so we begin with the dimension
lower bound.
Let $(\alpha_j)$ be a sequence of exponents increasing to $\theta$,
with $0<\alpha_j<\theta$. By Proposition~\ref{prop:plateau},
for every pair $j,m\ge1$, almost surely there is a compatible
$\alpha_j$-H\"older function on $[0,1]$ equal to zero at $0$
and to one on $[1/m,1]$. These countably many statements and
\eqref{eq:kaufman} hold on a common event of probability one.
For any $y\in K\setminus\{W_0\}$, there is a time $t\in(0,1]$
with $W_t=y$, and an integer $m$ with $1/m\le t$.
Applying Lemma~\ref{lem:dimension-test} for every $j$ and letting
$j\to\infty$ gives
\begin{equation}\label{eq:connected-dimension-lower}
 \dimH C\ge2\theta
\end{equation}
simultaneously for every $y\in K\setminus\{W_0\}$ and every
connected $C\subset K$ containing $W_0,y$. In particular, since
$W_1\ne W_0$ almost surely, $\mathfrak p_3(W)\ge2\theta$.
Proposition~\ref{prop:curve-upper} gives a simple curve whose upper
box dimension is at most $2\theta$. For every bounded set $E\subset\R^3$,
\begin{equation}\label{eq:box-comparison}
 \dimH E\le\underline{\operatorname{dim}}_{\mathrm B}E\le\dimUB E,
\end{equation}
where the middle term is the lower box dimension. Together with
\eqref{eq:connected-dimension-lower}, this gives Hausdorff and box
dimensions $2\theta$ for the curve. Thus $\mathfrak p_3(W)=2\theta$,
with the minimum attained by a simple curve.
Lawler's lower bound from cut points~\cite[Corollary 1.2]{LawlerCut} then gives
$\gamma_3:=2\theta>1$.

To prove \eqref{eq:main-lerw}, let
$\tau=\inf\{t>0:|W_t|=1\}$.
On an enlarged probability space, Sapozhnikov and Shiraishi
\cite[Theorems 1.1 and 1.2]{SapozhnikovShiraishi} couple
$W([0,\tau])$ with a simple arc contained in this range and
joining $W_0$ to $W_\tau$. The arc has the law of the loop-erased
random walk limit constructed by Kozma~\cite{Kozma}, and has
Hausdorff dimension $\ell_3$ by
\cite{ShiraishiDimension}.
The event $\{0<\tau<1\}$ has positive probability. On this event
the arc lies in $K$, so \eqref{eq:connected-dimension-lower} applies with
$y=W_\tau\ne W_0$. Thus $2\theta\le\ell_3$ on an event of
positive probability. Both exponents are deterministic, so
$\gamma_3\le\ell_3$. Lawler's upper bound~\cite{LawlerLERW} gives
$\ell_3\le5/3$ and hence $\theta\le5/6$.

It remains to prove the almost sure distance exponent. By Brownian
scaling, $1+V_N$ is stochastically dominated by $Y_N$. For each
$\varepsilon>0$, the relation $\mu_N=N^{\theta+o(1)}$ and
\eqref{eq:upper-tail} therefore give, for all sufficiently large $N$,
\[
 \Pp(V_N>N^{\theta+\varepsilon})
 \le C\exp\{-cN^{\varepsilon/2}\}.
\]
The Borel--Cantelli lemma yields
$\limsup_{N\to\infty}\log V_N/\log N\le\theta$ almost surely.
For the matching lower limit, fix $0<\alpha<\theta$ and use
Proposition~\ref{prop:plateau} to choose a compatible function $u$
on $[0,1]$ with $u(0)=0$, $u(1)=1$ and H\"older constant $H$.
Consider a shortest graph path from the first to the last cell,
with $q=L_N$ edges and hence $q+1$ visited cells.
For each transition between consecutive cells, an intersection
point gives two time representatives at which $u$ has the same value,
by compatibility. Within each visited cell, the entry and exit times
differ by at most $N^{-1}$, so the corresponding values of $u$
differ by at most $HN^{-\alpha}$. Summing these changes along
the path, including the initial and terminal time representatives,
and using $u(1)-u(0)=1$, we obtain
\[
 1\le H(q+1)N^{-\alpha}.
\]
Hence $L_N\ge(2H)^{-1}N^\alpha$ for all sufficiently large $N$.
Together with $L_N\le V_N$, a countable sequence
$\alpha\uparrow\theta$ proves \eqref{eq:main-pathwise}.
\end{proof}

We record the corresponding threshold for compatible H\"older functions
in the following corollary: the chain pseudometric separates the
endpoints below $\theta$ and vanishes above it.
\begin{corollary}\label{cor:chain-threshold}
The following two assertions hold simultaneously with probability one:
\begin{equation}\label{eq:chain-threshold}
 \delta_\alpha^W(0,1)>0\quad(0<\alpha<\theta),\qquad
 \delta_\alpha^W\equiv0\quad(\theta<\alpha<1).
\end{equation}
Consequently $\theta$ is the supremum of the exponents
$\alpha\in(0,1)$ for which a nonconstant compatible $\alpha$-H\"older
function exists.
\end{corollary}
\begin{proof}
The assertion for $\alpha<\theta$ follows from
Proposition~\ref{prop:plateau} and Lemma~\ref{lem:chain-dual}.
For the upper threshold, a shortest graph
path between the cells containing $s,t$ yields
\begin{equation}\label{eq:chain-graph}
 \delta_\alpha^W(s,t)
 \le n^{-\alpha}\bigl(d_{G_n}(\iota_n(s),\iota_n(t))+1\bigr).
\end{equation}
Between consecutive intersections, the two representative times belong
to the same cell, so their contribution is at most $n^{-\alpha}$.
On an event of probability one on which \eqref{eq:main-pathwise} holds,
\[
 \sup_{s,t\in[0,1]}\delta_\alpha^W(s,t)
 \le n^{-\alpha}(V_n+1)
 =n^{\theta-\alpha+o(1)}\longrightarrow0
 \qquad(\alpha>\theta).
\]
This proves the upper assertion simultaneously for all
$\alpha\in(\theta,1)$. Below $\theta$, a countable sequence of exponents
increasing to $\theta$ and monotonicity of $\delta_\alpha^W$ in $\alpha$
give simultaneity. The characterization by compatible functions follows from
\eqref{eq:chain-dual}.
\end{proof}

\subsection{Nonuniqueness of minimizing curves}\label{subsec:nonuniqueness}

We prove nonuniqueness by replacing disjoint subarcs of a minimizing
curve near $W_0$. The following lemma bounds the $p$-variation of
curves used in these replacements.

\begin{lemma}\label{lem:local-variation}
Almost surely, for every $0\le u<v\le1$ there is a continuous curve
$f_{u,v}:[0,1]\to W([u,v])$ joining $W_u$ to $W_v$ such that
\begin{equation}\label{eq:local-variation}
 \|f_{u,v}\|_{p\text{-var}}\le C_p|v-u|^{1/4}
 \qquad(p>\gamma_3).
\end{equation}
For each $p>\gamma_3$, the constant $C_p=C_p(W)$ is finite
almost surely and does not depend on $u,v$. The curve $f_{u,v}$
can be chosen independently of $p$, so all these estimates hold
for that one curve.
\end{lemma}
\begin{proof}
Choose $\beta_j\uparrow1/\gamma_3$ with $0<\beta_j<1/\gamma_3$.
For each time cell $I=[k2^{-m},(k+1)2^{-m}]\subset[0,1]$,
Brownian scaling and Proposition~\ref{prop:curve-upper} give one curve
$f_I:[0,1]\to W(I)$ joining the Brownian positions at the endpoints
of $I$, with random bounds $H_{I,j}$ such that
\[
 [f_I]_{\beta_j}\le H_{I,j},\qquad
 \E[H_{I,j}^a]\le C_{j,a}|I|^{a/2}\qquad(a>0).
\]
Since the collection of dyadic intervals and the sequence
$(\beta_j)$ are countable, the curves $f_I$ and their bounds
$H_{I,j}$ can be chosen simultaneously on an event of probability one. A union bound and Markov's inequality give, for $a>4$,
\[
 \Pp\left(\max_{|I|=2^{-m}}H_{I,j}>2^{-m/4}\right)
 \le C_{j,a}2^{m(1-a/4)}.
\]
Since these probabilities are summable in $m$, the Borel--Cantelli
lemma gives almost surely finite $A_j$ with
$H_{I,j}\le A_j|I|^{1/4}$ for every $I,j$. If $p>\gamma_3$, choose
$j$ with $\beta_jp>1$. Then
\begin{equation}\label{eq:cell-variation}
 \|f_I\|_{p\text{-var}}\le[f_I]_{\beta_j}
       \le A_j|I|^{1/4}.
\end{equation}

For arbitrary $u<v$, decompose $[u,v]$ into maximal dyadic intervals
contained in $[u,v]$. Their interiors are disjoint, at most two have
each given length, and their union contains $(u,v)$. Thus
$\sum_I|I|^{1/4}\le C|v-u|^{1/4}$.
On each $I=[a,b]$, use the curve $f_I$ with parameter
$(r-a)/(b-a)$, $r\in I$, and concatenate these curves in increasing
time order. Their endpoint values agree with $W$ at the common cell
endpoints. Setting the values at $u,v$ to $W_u,W_v$ gives a continuous
curve with image in $W([u,v])$: the dyadic intervals can accumulate
only at $u,v$, and $f_I([0,1])\subset W(I)$.

To bound its $p$-variation, retain finitely many consecutive dyadic
intervals and join the two remaining ends by straight segments.
These auxiliary curves converge uniformly to the concatenated curve,
and the $p$-variation of the two straight segments tends to zero.
The triangle inequality, \eqref{eq:cell-variation} and lower
semicontinuity of $p$-variation give
\[
 \|f_{u,v}\|_{p\text{-var}}
 \le A_j\sum_I|I|^{1/4}
 \le C A_j|v-u|^{1/4},\qquad \beta_jp>1.
\]
Rescaling its parameter interval affinely to $[0,1]$ proves
\eqref{eq:local-variation}, with a suitable random constant $C_p$.
The estimates for all dyadic intervals hold simultaneously on an
event of probability one, independently of $u,v$. Hence the
construction applies to intervals chosen after observing $W$.
The curves $f_I$ are fixed independently of $p$, so the resulting
curve $f_{u,v}$ satisfies the bound for every $p>\gamma_3$.
\end{proof}

\Needspace{7\baselineskip}
A double point of $W$ on $[0,q]$ is a point visited at two distinct
times in that interval; a triple point is visited at three distinct
times. A double point $z=W_s=W_t$, $s<t$, is a
\emph{pioneer double cut point} if
$W([0,s))\cap W((s,t))=\varnothing$; see
\cite[Section~1.1]{GaoLiLiuQian}. The following lemma uses
Gao, Li, Liu and Qian's theorem excluding such points to construct
connections that avoid $z$ and $W_0$.
\begin{lemma}\label{lem:double-point-avoidance}
Almost surely, simultaneously for every positive dyadic rational
$q\le1$ and every double point $z$ of $W$ on $[0,q]$, any two points
of $K_q\setminus\{W_0,z\}$, where $K_q=W([0,q])$, can be joined by
a continuous curve $g:[0,1]\to K_q\setminus\{W_0,z\}$ with
\begin{equation}\label{eq:avoidance-variation}
 \|g\|_{p\text{-var}}\le3C_pq^{1/4}\qquad(p>\gamma_3),
\end{equation}
where $C_p$ is the random constant in \eqref{eq:local-variation}.
\end{lemma}
\begin{proof}
Three-dimensional Brownian motion has no triple points
\cite[Theorem~9.22]{MortersPeres}. A fixed point is polar
\cite[Corollary~2.26]{MortersPeres}. Indeed, for every $m\ge1$,
$W_{1/m}\ne W_0$ almost surely, and the Markov property and polarity
rule out further visits to $W_0$ after $1/m$.
Consequently, $W_0$ is almost surely not revisited after time zero.
For any fixed $q>0$, the reversed path
$(W_q-W_{q-t})_{0\le t\le q}$ has the law of a Brownian motion
started at the origin. Applying the same polarity statement to this
reversed path shows that $W_q$ is almost surely not visited before
$q$. These conclusions hold simultaneously for all positive dyadic
rationals $q\le1$ because there are countably many of them.
Fix a realization on which these statements,
Lemma~\ref{lem:local-variation}, and
\cite[Theorem~1.2]{GaoLiLiuQian} for $W$ and for each of the
reversed paths all hold. Such realizations form an event of
probability one.

Fix such a dyadic $q$ and a double point $z$ of $W$ on $[0,q]$.
As there are no triple points, $z$ is visited at exactly two
times $s<t$; the endpoint properties above give $0<s<t<q$.
By \cite[Theorem~1.2]{GaoLiLiuQian}, these two visits cannot
satisfy $W([0,s))\cap W((s,t))=\varnothing$.
Applying the same theorem to the reversed path, and using the absence
of further visits to $W_0$ and $W_q$, yields
\[
 W((0,s))\cap W((s,t))\ne\varnothing,\qquad
 W((s,t))\cap W((t,q])\ne\varnothing.
\]
Since there are no further visits to $z$,
\[
 K_q\setminus\{W_0,z\}
 =W((0,s))\cup W((s,t))\cup W((t,q]).
\]
The two nonempty intersections above allow any two points of
$K_q\setminus\{W_0,z\}$ to be joined using at most three curves,
one in each of the time ranges $(0,s)$, $(s,t)$ and $(t,q]$.
For each connection, take a constant curve if its endpoints coincide.
Otherwise, choose times representing its endpoints in the relevant
time interval and apply Lemma~\ref{lem:local-variation} between those
times, reversing the resulting curve if necessary. Each closed time
interval used is contained in $(0,s)$, $(s,t)$ or $(t,q]$, respectively;
its Brownian range therefore avoids $W_0$ and $z$, and its length is
at most $q$. Concatenation gives \eqref{eq:avoidance-variation}
for the same curve $g$ and every $p>\gamma_3$.
\end{proof}

\Needspace{7\baselineskip}
The next proposition constructs the minimizing curves with distinct
images required in Theorem~\ref{thm:main-growth}.
\begin{proposition}\label{prop:nonunique-curves}
On an event of probability one, for every $y\in K\setminus\{W_0\}$
there are continuum many simple curves from $W_0$ to $y$ in $K$ with
pairwise distinct images and finite $p$-variation for every $p>\gamma_3$.
In particular, after the reparametrization in
Lemma~\ref{lem:holder-arc}, these curves with $y=W_1$ satisfy
\eqref{eq:main-minimizers}, giving the nonuniqueness assertion of
Theorem~\ref{thm:main-growth}.
\end{proposition}
\begin{proof}
Fix a realization in the event of probability one on which
Lemmas~\ref{lem:local-variation} and~\ref{lem:double-point-avoidance},
Proposition~\ref{prop:curve-upper}, \eqref{eq:kaufman}, and
\eqref{eq:connected-dimension-lower} all hold.
Fix $y\in K\setminus\{W_0\}$
and choose a simple curve $\eta:[0,1]\to K$ from $W_0$ to $y$ as in
Proposition~\ref{prop:curve-upper},
in its parametrization that is $\beta$-H\"older for every
$0<\beta<1/\gamma_3$. Its image $A=\eta([0,1])$ has Hausdorff
dimension $\gamma_3$ by \eqref{eq:curve-upper} and
\eqref{eq:connected-dimension-lower}.

Every nondegenerate closed subarc $J\subset A\cap K_q$, with
$q\in(0,1]$ dyadic, contains a double point of $W$ on $[0,q]$.
Indeed, if no point of $J$ were visited twice during $[0,q]$, then
\[
 W:\{t\in[0,q]:W_t\in J\}\longrightarrow J
\]
would be a continuous bijection. Its domain is compact, so this map
would be a homeomorphism. Since $J$ is homeomorphic to a closed
interval, its preimage would be a nondegenerate closed interval in
$[0,q]$. Equation~\eqref{eq:kaufman} would then give $\dimH J=2$,
contrary to $\dimH J\le\gamma_3<2$.

We now choose double points $z_n=\eta(c_n)$ and simple curves
$\lambda_n:[a_n,b_n]\to K$ inductively. Start with $a_0=1$.
Having chosen the first $n-1$ curves, put
\[
 r_n=\min\{2^{-n},a_{n-1}\},\qquad
 E_n=\eta([r_n,1])\cup\bigcup_{j<n}\lambda_j([a_j,b_j]).
\]
Since the compact set $E_n$ does not contain $W_0$, Brownian
continuity allows us to choose a positive dyadic $q_n\le2^{-n}$ with
$K_{q_n}\cap E_n=\varnothing$. Because $W_0$ is never revisited, $W([q_n,1])$ is also compact
and does not contain $W_0$. A sufficiently short initial subarc of
$\eta$ therefore lies in a ball about $W_0$ disjoint from
$W([q_n,1])$, and so lies in $K_{q_n}$. Choosing a nondegenerate
closed subarc strictly inside it gives a double point
$z_n=\eta(c_n)$ of $W$ on $[0,q_n]$.

Choose one point before $z_n$ and one after it along this initial
subarc. Lemma~\ref{lem:double-point-avoidance} joins them, in that order,
by a curve $g$ in
$K_{q_n}\setminus\{W_0,z_n\}$ satisfying
\eqref{eq:avoidance-variation}. Let $v$ be the first parameter time
at which $g$ meets $\eta([c_n,1])$, and let $u$ be its last visit to
$\eta([0,c_n])$ before $v$. The sets $\eta([0,c_n])$ and $\eta([c_n,1])$ are compact and
intersect only at $z_n$, which $g$ avoids. Thus $u,v$ are well defined,
$u<v$, and $g((u,v))$ does not meet $A$.
Apply Lemma~\ref{lem:holder-arc} to the restricted curve $g|_{[u,v]}$.
The resulting simple curve has endpoints $\eta(a_n),\eta(b_n)$ with
\begin{equation}\label{eq:detour-intervals}
 0<a_n<c_n<b_n<r_n.
\end{equation}
Parametrize it on $[a_n,b_n]$ and call it $\lambda_n$. By construction,
\begin{equation}\label{eq:detour-properties}
 \begin{gathered}
 \lambda_n([a_n,b_n])\subset K_{q_n}\setminus\{W_0,z_n\},\\
 \lambda_n([a_n,b_n])\cap A=\{\eta(a_n),\eta(b_n)\},\\
 \|\lambda_n\|_{p\text{-var};[a_n,b_n]}
       \le3C_pq_n^{1/4}\qquad(p>\gamma_3).
 \end{gathered}
\end{equation}
It is disjoint from every earlier $\lambda_j$, since it lies in
$K_{q_n}$ and $K_{q_n}\cap E_n=\varnothing$.

To sum the changes made by these replacements, define
$h_n:[0,1]\to\R^3$ to equal $\lambda_n-\eta$ on $[a_n,b_n]$
and zero elsewhere. The matching endpoint values make $h_n$ continuous.
For $p>\gamma_3$, choose $\beta<1/\gamma_3$ with $\beta p>1$.
Extending by zero does not change the variation of a curve with zero
endpoints, so \eqref{eq:detour-intervals}--\eqref{eq:detour-properties}
give
\begin{equation}\label{eq:detour-summability}
 \|h_n\|_{p\text{-var}}
 \le3C_p2^{-n/4}+[\eta]_\beta2^{-n\beta},\qquad
 \sum_{n\ge1}\|h_n\|_{p\text{-var}}<\infty.
\end{equation}
A binary sequence $\varepsilon=(\varepsilon_n)_{n\ge1}\in\{0,1\}^{\N}$
now determines a curve by
\[
 f_\varepsilon=\eta+\sum_{n\ge1}\varepsilon_nh_n.
\]
By \eqref{eq:detour-summability}, the series converges uniformly, and
the triangle inequality and lower semicontinuity of $p$-variation give
\[
 \|f_\varepsilon\|_{p\text{-var}}
 \le\|\eta\|_{p\text{-var}}+\sum_{n\ge1}\|h_n\|_{p\text{-var}}
 <\infty\qquad(p>\gamma_3).
\]
The intervals $[a_n,b_n]$ are pairwise disjoint by
\eqref{eq:detour-intervals}. By \eqref{eq:detour-properties}, each
$\lambda_n$ meets $A$ only at its two endpoints, and the images of
different $\lambda_n$ are disjoint. Since $\eta$ and all the
$\lambda_n$ are simple, $f_\varepsilon$ is simple, joins $W_0$ to
$y$, and takes values in $K$.

The point $z_n=\eta(c_n)$ is omitted from $\lambda_n$. It does not
belong to an earlier replacement because $z_n\in K_{q_n}$ and
$K_{q_n}\cap E_n=\varnothing$. Nor does it belong to a later
replacement $\lambda_m$, $m>n$, since $r_m<c_n$ and thus
$z_n\in\eta([r_m,1])\subset E_m$, while
$\lambda_m([a_m,b_m])\cap E_m=\varnothing$.
As $\eta$ is injective and $a_n<c_n<b_n$, it follows that
\[
 z_n\in f_\varepsilon([0,1])\quad\Longleftrightarrow\quad
 \varepsilon_n=0.
\]
Hence distinct binary sequences produce distinct curve images.

Since each $f_\varepsilon$ is already simple, the reparametrization
in Lemma~\ref{lem:holder-arc} preserves its image and makes it
$\beta$-H\"older for every $0<\beta<1/\gamma_3$. These bounds give
upper box dimension at most $\gamma_3$. The lower bound
\eqref{eq:connected-dimension-lower} for connected sets gives Hausdorff
dimension at least $\gamma_3$. By \eqref{eq:box-comparison}, the
Hausdorff and box dimensions therefore both equal $\gamma_3$.
\end{proof}

\section{Limits of internal distances}\label{sec:limits}

In this section we prove tightness of the normalized internal graph
distances, the pseudometric properties in \eqref{eq:main-local}, and
the regularity assertion in Theorem~\ref{thm:main-metric}(i).
The remaining assertions of that theorem are proved later:
the reconstruction formula in (ii) is Theorem~\ref{thm:exact-gluing};
the measurability and convergence assertions in (iii) are proved in
Theorem~\ref{thm:measurable} and
Corollary~\ref{cor:strong-convergence}, respectively; and the
separation assertion in (iv) is Theorem~\ref{thm:deterministic-root}.

\subsection{Definition and properties of the limits}\label{subsec:metric-statement}

We first define the internal distances and state the properties of their
limits. A distance is called \emph{internal} to an interval $I$ when its
graph paths use only cells contained in $I$. We consider
the countable family
\[
 \mathscr J=
 \bigl\{[k2^{-m},\ell2^{-m}]:
 m\in\{0,1,\ldots\},\ k,\ell\in\Z,\ 0\le k<\ell\bigr\}.
\]
We call an element of $\mathscr J$ a dyadic interval; its endpoints
may be any two nonnegative dyadic rational numbers.
The notation $C(I^2)$, with $I=[a,b]$, means $C(I\times I,\R)$.
For $0<\alpha<1$, let $C^\alpha(I^2)$ be the subspace of functions
with finite $\alpha$-H\"older seminorm, equipped with the norm
\[
 \|f\|_{C^\alpha(I^2)}=\|f\|_\infty+[f]_{\alpha;I^2}.
\]
We also write, for $T>0$,
\[
 C_0([0,T],\R^3)=
 \{w\in C([0,T],\R^3):w(0)=\mathbf0\},
\]
with the uniform norm. Let $n_j=2^{k_j}$, where $(k_j)$ is a strictly
increasing sequence of nonnegative integers, and let $\kappa_j>0$ be
deterministic normalization constants.

For $I=[a,b]\in\mathscr J$ and all sufficiently large $j$, both $an_j$ and
$bn_j$ are integers. Let $G_j^I$ be the intersection graph with vertices
$an_j,\ldots,bn_j-1$ and cells $[k/n_j,(k+1)/n_j]$, with the edge rule
in \eqref{eq:graph-definition}. Define
\begin{equation}\label{eq:local-kernel}
 \iota_j^I(t)=\min\{\lfloor n_jt\rfloor,bn_j-1\},\qquad
 d_j^I(s,t)=\kappa_j^{-1}
        d_{G_j^I}(\iota_j^I(s),\iota_j^I(t)),\quad s,t\in I.
\end{equation}
Let $\widetilde d_j^I$ agree with $d_j^I$ at every grid point
$(k/n_j,\ell/n_j)\in I^2$, and make it linear separately in $s$
and $t$ on each grid square (bilinear interpolation). In either
argument, the value of $d_j^I$ at $b$ equals that at $b-1/n_j$,
since both times are assigned to the last cell of $G_j^I$.
Moving either vertex to a neighbour changes graph distance by at most
one, so
\begin{equation}\label{eq:interpolation}
 \|\widetilde d_j^I-d_j^I\|_\infty\le2/\kappa_j.
\end{equation}
The function $d_j^I$ is a pseudometric on time labels, whereas its
interpolation need not be one. The uniform error in
\eqref{eq:interpolation} will transfer the pseudometric properties to
the limit. For the finitely many indices at which the endpoints are not
grid points, set both functions equal to zero.

The following theorem describes the resulting limits, their regularity
and their dependence on the Brownian path.
\begin{theorem}\label{thm:main-metric}
Let $\theta$ be the exponent in Theorem~\ref{thm:main-growth}, and
choose the deterministic sequences $(n_j)$ and $(\kappa_j)$ from
Lemma~\ref{lem:records}. Then $\kappa_j=n_j^{\theta+o(1)}$ is comparable
to both $\E L_{n_j}$ and $\E V_{n_j}$ by constants independent of $j$,
and the laws of
$(W,(\widetilde d_j^I)_{I\in\mathscr J})$ are tight. Here we use uniform
convergence on compact time intervals for $W$ and uniform convergence
on $I^2$ for each distance function. The distance coordinates are
also tight in $C^\alpha(I^2)$ for every $0<\alpha<\theta$.
Let $(W,(\D^I)_{I\in\mathscr J})$ be any joint subsequential limit
in distribution. Almost surely, its distance coordinates are continuous
pseudometrics satisfying
\begin{equation}\label{eq:main-local}
 \begin{aligned}
 W_s=W_t&\Longrightarrow\D^I(s,t)=0
       &&(I\in\mathscr J,\ s,t\in I),\\
 \D^J(s,t)&\le\D^I(s,t)
       &&(I,J\in\mathscr J,\ I\subset J,\ s,t\in I).
 \end{aligned}
\end{equation}
The following assertions hold for this limiting law.
\begin{enumerate}
\item For every $I\in\mathscr J$, there are a nonnegative random variable
$H_I$ and a deterministic $c_I>0$ such that
$\E\exp\{c_IH_I^{1/(1-\theta)}\}<\infty$ and, almost surely,
\[
 \D^I(s,t)\le H_I|s-t|^\theta
       \left(\log\frac{e|I|}{|s-t|}\right)^{1-\theta}
       \qquad(s,t\in I,\ s\ne t).
\]
\item On an event of probability one, for every finite partition
$I=I_1\cup\cdots\cup I_m$, where $I,I_1,\ldots,I_m\in\mathscr J$
and the interiors of the $I_i$ are disjoint, and all $s,t\in I$,
\[
 \D^I(s,t)=\inf\sum_{r=1}^q\D^{I_{i_r}}(a_r,b_r).
\]
The infimum is over integers $q\ge1$, indices
$i_r\in\{1,\ldots,m\}$, and times $a_r,b_r\in I_{i_r}$ with
$a_1=s$, $b_q=t$, and $W_{b_r}=W_{a_{r+1}}$ for $r<q$.
No chronological ordering is imposed on the pairs $(a_r,b_r)$:
$a_r$ may exceed $b_r$, and the same partition interval may occur
more than once among $I_{i_1},\ldots,I_{i_q}$.
\item For every $I=[a,b]\in\mathscr J$, there is a Borel map
$\Psi_I:C_0([0,b-a],\R^3)\to C(I^2)$ such that
$\D^I=\Psi_I((W_{a+u}-W_a)_{0\le u\le b-a})$ almost surely.
If the joint convergence is along indices $j_k$, then, on the original
probability space of $W$,
\[
 \E\left\|\widetilde d_{j_k}^I-
       \Psi_I\bigl((W_{a+u}-W_a)_{0\le u\le b-a}\bigr)
       \right\|_{C^\alpha(I^2)}^p\longrightarrow0
       \qquad(p>0,\ 0<\alpha<\theta).
\]
In particular, convergence holds in each such H\"older norm in probability.
\item For each fixed deterministic $s\in[0,1]$, on an event of
probability one,
\begin{equation}\label{eq:main-root-separation}
 \inf_{\substack{t\in[0,1]\\|t-s|\ge\varepsilon}}
       \D^{[0,1]}(s,t)>0\qquad\text{for every }\varepsilon>0,
\end{equation}
where the infimum of the empty set is $+\infty$.
\end{enumerate}
\end{theorem}

Proposition~\ref{prop:coarse-information} refines the measurability
assertion in Theorem~\ref{thm:main-metric}(iii): $\D^{[0,1]}$ is
measurable with respect to the completed sigma-field generated by the
compact self-intersection relation
\[
 \mathcal I_{[0,1]}(W)=\{(s,t)\in[0,1]^2:W_s=W_t\}.
\]

The event of probability one in Theorem~\ref{thm:main-metric}(iv)
may depend on $s$. It remains open whether every subsequential limit
satisfies, almost surely,
\begin{equation}\label{eq:spatial-separation-question}
 \D^{[0,1]}(s,t)=0\quad\Longrightarrow\quad W_s=W_t
 \qquad\text{for all }s,t\in[0,1].
\end{equation}
By \eqref{eq:main-local}, $\D^{[0,1]}$ induces a pseudometric on the
Brownian range $K=W([0,1])$ through
$d_K(W_s,W_t)=\D^{[0,1]}(s,t)$. If
\eqref{eq:spatial-separation-question} holds, $d_K$ is a metric on $K$.
The question of full separation is revisited in
Question~\ref{qu:critical-separation}.

In the remainder of this section, Lemma~\ref{lem:kernel-moments} and
Proposition~\ref{prop:local-field} establish tightness and the
properties in \eqref{eq:main-local}; Lemma~\ref{lem:critical-modulus}
proves Theorem~\ref{thm:main-metric}(i).
Lemma~\ref{lem:local-independence} records independence of internal
limits on disjoint intervals, and Proposition~\ref{prop:quotient-geometry}
describes the metric quotient. Parts (ii)--(iv) of
Theorem~\ref{thm:main-metric} are proved in Sections~\ref{sec:gluing}
and~\ref{sec:separation}, as specified above.

\subsection{Tightness and joint limits}\label{subsec:joint-limits}

Uniform convergence of paths need not preserve self-intersections:
even if $w_j\to w$ uniformly and $w(s)=w(t)$, there need not be times
$s_j\to s$ and $t_j\to t$ with $w_j(s_j)=w_j(t_j)$. For this reason,
we include the self-intersection relations as additional coordinates
in the joint convergence. Their Hausdorff convergence will provide
the exact intersections in the approximating paths needed to show that
$\D^I(s,t)=0$ whenever $W_s=W_t$; see the proof of
Proposition~\ref{prop:local-field}. Throughout this subsection we use
the sequences from Lemma~\ref{lem:records}.

For $I=[a,b]\in\mathscr J$ and $j$ such that $an_j,bn_j$ are integers,
write $G_j^I[u,v]$ for the subgraph of $G_j^I$ induced by the
consecutive vertices $u,\ldots,v-1$, where
$an_j\le u<v\le bn_j$. Set
\[
 Z_j^I=\frac1{\kappa_j}\left(1+
   \max_{an_j\le u<v\le bn_j}\diam G_j^I[u,v]\right).
\]
Thus every diameter is computed using only paths within the
corresponding consecutive restriction. Note that Brownian scaling gives
$Z_j^I\law\kappa_j^{-1}Y_{(b-a)n_j}$.
For the finitely many indices at which the endpoints are not grid
points, set $Z_j^I=0$. Along with the distances, we retain the compact set
\begin{equation}\label{eq:intersection-relation}
 \mathcal I_I(W)=\{(s,t)\in I^2:W_s=W_t\}.
\end{equation}
For nonempty compact sets $A,B\subset I^2$, define their Hausdorff
distance by
\[
 d_{\mathrm H}(A,B)=\max\left\{
   \sup_{x\in A}\inf_{y\in B}|x-y|,\;
   \sup_{y\in B}\inf_{x\in A}|x-y|\right\}.
\]
We use this metric on the nonempty compact subsets of $I^2$.
The map $W\mapsto\mathcal I_I(W)$ is Borel measurable for this
topology. Indeed, $\mathcal I_I(W)$ meets a fixed compact rectangle
$R\subset I^2$ exactly when
$\min_{(s,t)\in R}|W_s-W_t|=0$; taking countable unions over closed
rectangles inside open sets gives the measurable hitting events.

For a bounded pseudometric $q$ on $I$, write
$\diam q=\sup_{s,t\in I}q(s,t)$.
For $0<\alpha<1$, each $\widetilde d_j^I$ is a Borel random element
of $C^\alpha(I^2)$: for fixed $I$ and $j$, its defining graph takes
values in a finite set of labelled graphs, and the edge events are
measurable.

\Needspace{7\baselineskip}
\begin{lemma}\label{lem:kernel-moments}
Fix $I=[a,b]\in\mathscr J$, $0<\alpha<\rho<\theta$, and $p>0$.
The laws of $\widetilde d_j^I$ are tight in $C^\alpha(I^2)$, and
$\sup_j\E\bigl[[\widetilde d_j^I]_{\alpha;I^2}^{\,p}\bigr]<\infty$.
If $|I|\le1$, then
\begin{equation}\label{eq:local-envelope-moments}
 \E[(Z_j^I)^p]\le C_{p,\rho}|I|^{p\rho},
\end{equation}
where $C_{p,\rho}$ depends only on $p$ and $\rho$, not on $I$
or $j$. For every $M>1$ there is a
constant $C_{M,p}<\infty$ such that
\[
 \sup_{\substack{I\in\mathscr J\\1<|I|\le M}}
 \sup_{j\ge1}\E[(Z_j^I)^p]\le C_{M,p}.
\]
\end{lemma}
\begin{proof}
Since $d_j^I$, $\widetilde d_j^I$, and $Z_j^I$ are zero when
the endpoints of $I$ are not grid points, we may assume
$an_j,bn_j\in\Z$. To bound $d_j^I(s,t)$, write
\[
 r=\min\{\iota_j^I(s),\iota_j^I(t)\},\qquad
 v=1+\max\{\iota_j^I(s),\iota_j^I(t)\}.
\]
Both assigned vertices belong to the consecutive restriction
$G_j^I[r,v]$, which contains $v-r\le n_j|s-t|+2$ vertices.
Their distance in $G_j^I$ is at most their distance in this
restriction, and hence at most its diameter. Applying
Proposition~\ref{prop:upper-tail} and \eqref{eq:record-local} gives,
for $|I|\le M$,
\begin{equation}\label{eq:local-distance-moment}
 \E[(d_j^I(s,t))^p]\le
 C_{M,p,\rho}(|s-t|+n_j^{-1})^{p\rho}.
\end{equation}
If the restriction contains more than $n_j$ cells, divide it into at
most $\lceil M\rceil+1$ consecutive parts and use
\eqref{eq:subadd} before applying these estimates.
For $|I|\le1$, the law of $Z_j^I$, the moment bound following
from \eqref{eq:upper-tail}, and \eqref{eq:record-local} with
$l=|I|n_j$ give
\[
 \E[(Z_j^I)^p]\le C_p\left(\frac{\mu_{|I|n_j}}{\kappa_j}\right)^p
    \le C_{p,\rho}|I|^{p\rho},
\]
proving \eqref{eq:local-envelope-moments} with no dependence on $I$.
For $1<|I|\le M$, divide the $|I|n_j$ cells into at most
$\lceil M\rceil+1$ consecutive blocks of at most $n_j$ cells.
Subadditivity \eqref{eq:subadd}, monotonicity of $\mu_l$, and
$\mu_{n_j}=\kappa_j$ then give
$\E[(Z_j^I)^p]\le C_{M,p}$ uniformly in $I$ and $j$.
For any $p>0$ and nonnegative $x_1,\ldots,x_m$, we use
\[
 \left(\sum_{r=1}^m x_r\right)^p
 \le m^{\max\{p-1,0\}}\sum_{r=1}^m x_r^p.
\]

The graph triangle inequality compares distances at two pairs of
vertices:
\[
 |d(x,y)-d(x',y')|\le d(x,x')+d(y,y').
\]
For $z,z'\in I^2$ with $|z-z'|\ge n_j^{-1}$, this inequality,
\eqref{eq:local-distance-moment}, and the interpolation error
\eqref{eq:interpolation} give
\[
 \E|\widetilde d_j^I(z)-\widetilde d_j^I(z')|^p
       \le C_{I,p,\rho}|z-z'|^{p\rho}.
\]
For shorter increments, write $z=(s,t)$ and $z'=(s',t')$.
With the other coordinate fixed, moving between adjacent grid
points changes $\widetilde d_j^I$ by at most $1/\kappa_j$.
Bilinear interpolation therefore gives
\[
 |\widetilde d_j^I(s,t)-\widetilde d_j^I(s',t')|
 \le\frac{n_j}{\kappa_j}(|s-s'|+|t-t'|).
\]
By \eqref{eq:record-local} with $l=1$,
$\kappa_j\ge c_\rho n_j^\rho$. When $|z-z'|<n_j^{-1}$,
this inequality also yields the preceding moment bound.

Choose $\alpha<\beta<\rho$ and $q>p$ with
$\rho-2/q>\beta$. At dyadic level $k$ in the rectangle $I^2$,
there are at most $C2^{2k}$ adjacent grid pairs. The increment
moment bound therefore gives an $L^q$ bound of
$C2^{-k(\rho-2/q)}$ for their maximum.
Telescoping along successive dyadic approximations, as in the proof of
Proposition~\ref{prop:scalar-seed}, and summing $2^{k\beta}$ times
these increments gives a uniform $p$-th moment bound for the
$\beta$-H\"older seminorm. Since $\widetilde d_j^I(a,a)=0$, the $C^\beta(I^2)$ norms also have
uniformly bounded $p$-th moments, as do the $\alpha$-H\"older seminorms.
Bounded balls in $C^\beta(I^2)$ have compact closure in $C^\alpha(I^2)$
by the Arzel\`a--Ascoli theorem
\cite[Theorem 7.25]{Rudin} and the interpolation inequality
\eqref{eq:holder-interpolation}. This proves tightness.
\end{proof}

The tightness of $\widetilde d_j^I$ and the moment bounds for
$Z_j^I$ from Lemma~\ref{lem:kernel-moments} allow joint subsequential
convergence over the countable family $\mathscr J$, together with
$W$ and its self-intersection relations.
\begin{proposition}\label{prop:local-field}
Every subsequence of $(n_j)$ has a further deterministic subsequence
along which
\begin{equation}\label{eq:joint-local-limit}
 \left(W,\bigl(\mathcal I_I(W),\widetilde d_j^I,Z_j^I\bigr)_{I\in\mathscr J}\right)
 \xrightarrow{\mathrm d}
 \left(W,\bigl(\mathcal I_I(W),\D^I,Z^I\bigr)_{I\in\mathscr J}\right).
\end{equation}
The convergence in distribution in \eqref{eq:joint-local-limit} uses
uniform convergence on compact time intervals for the Brownian path,
the uniform topology on $C(I^2)$ for each distance function, the
Hausdorff metric for the compact sets $\mathcal I_I(W)$, and the
usual topology on $\R$ for $Z_j^I$. Almost surely, simultaneously
over all $I,J\in\mathscr J$, the following hold:
\begin{align}
 W_s=W_t&\Longrightarrow\D^I(s,t)=0,\qquad s,t\in I,
                      \label{eq:local-identifications}\\
 \D^J(s,t)&\le\D^I(s,t),\qquad I\subset J,\quad s,t\in I.
                      \label{eq:domain-monotonicity}
\end{align}
Every $\D^I$ is a continuous pseudometric and, for $0<\alpha<\theta$,
\begin{equation}\label{eq:local-holder}
 \D^I(s,t)\le H_{I,\alpha}|s-t|^\alpha,
       \qquad \E[H_{I,\alpha}^p]<\infty\quad(p>0).
\end{equation}
For $|I|\le1$ and $p>0$,
\begin{equation}\label{eq:diameter-moments}
 \E[(\diam\D^I)^p]\le\E[(Z^I)^p]\le C_p|I|^{p\theta}.
\end{equation}
In fact, these moment bounds follow from the stronger uniform tail
estimate: there are constants $c,C>0$, independent of $I$, such that
\begin{equation}\label{eq:diameter-tail}
 \Pp(\diam\D^I>u|I|^\theta)
 \le\Pp(Z^I>u|I|^\theta)
 \le C\exp\{-cu^{1/(1-\theta)}\}
       \qquad(|I|\le1,\ u\ge0).
\end{equation}
Also,
\begin{equation}\label{eq:limit-positive-mean}
 \E\D^{[0,1]}(0,1)\ge c_*>0.
\end{equation}
\end{proposition}
\begin{proof}
Since $W$ and its self-intersection relations
$(\mathcal I_I(W))_{I\in\mathscr J}$ do not depend on $j$, their joint
law is tight. Lemma~\ref{lem:kernel-moments} gives tightness of
$(\widetilde d_j^I)_j$ and $(Z_j^I)_j$ for each $I\in\mathscr J$.
Choosing compact coordinate sets with summable error probabilities gives
tightness in the countable product. Prokhorov's theorem
\cite[Theorem 5.1]{Billingsley} gives a weakly convergent subsequence.
The product space is Polish, so the Skorokhod representation theorem
\cite[Theorem 6.7]{Billingsley} gives copies of these random variables
on a space where the convergence is almost sure. On this space write
$W^{(j)}$ for the Brownian path at index $j$, and $W$ for its limit.
The finite quantities $d_j^I,\widetilde d_j^I,Z_j^I$ are now computed
from $W^{(j)}$. The joint law of
$(W^{(j)},(\mathcal I_I(W^{(j)}))_{I\in\mathscr J})$ does not depend
on $j$, so the same law is preserved in the weak limit.
In particular, the compact sets in the limit are
$\mathcal I_I(W)$ for all $I\in\mathscr J$ almost surely.

Fix $I\in\mathscr J$ and $s,t\in I$ with $W_s=W_t$.
Hausdorff convergence of $\mathcal I_I(W^{(j)})$ to $\mathcal I_I(W)$
gives times $s_j\to s$, $t_j\to t$ with $W^{(j)}_{s_j}=W^{(j)}_{t_j}$.
Their cells in $G_j^I$ coincide or are joined by an edge, so
$d_j^I(s_j,t_j)\le1/\kappa_j$. Uniform convergence of the
interpolations and \eqref{eq:interpolation} give
\eqref{eq:local-identifications}. Symmetry, nonnegativity, the zero
diagonal and the triangle inequality pass from $d_j^I$ to the limit.
For $I\subset J$, the graph $G_j^I$ is the subgraph of $G_j^J$
induced by the cells contained in $I$. For $t\in I$, the assigned
vertices in the two graphs agree, except possibly at the right
endpoint of $I$, where they are adjacent. Consequently,
\[
 d_j^J(s,t)\le d_j^I(s,t)+2/\kappa_j\qquad(s,t\in I).
\]
Passing to the limit gives \eqref{eq:domain-monotonicity}.

Lemma~\ref{lem:kernel-moments}, Fatou's lemma and lower semicontinuity
of the H\"older seminorm give \eqref{eq:local-holder}, first for a
sequence of exponents increasing to $\theta$ and then for every smaller
exponent. We can choose
$H_{I,\alpha}=\sup_{s,t\in I,\,s\ne t}\D^I(s,t)/|s-t|^\alpha$
measurably. Since $d_j^I(s,t)\le Z_j^I$ for all $s,t\in I$,
the joint convergence and \eqref{eq:interpolation} give
$\diam\D^I\le Z^I$.

Equation~\eqref{eq:upper-tail} gives constants $c_0>0$ and $B_*>1$
such that $\E\exp(c_0Y_N/\mu_N)\le B_*$ for every $N\ge1$.
Fix $I$ with $L=|I|\le1$ and divide it into $m$ equal intervals,
where $m$ is a power of two. For sufficiently large $j$,
$l=Ln_j/m$ is an integer. By \eqref{eq:subadd}, $\kappa_jZ_j^I$
is at most the sum of the interval maxima on these $m$ parts.
These maxima have the law of $Y_l$ and are independent, since each
depends only on the Brownian increments in its own part.
By \eqref{eq:record-variable},
$\mu_l\le2\kappa_j(L/m)^{\rho_j}$, so the exponential moment
bound applies with parameter $c_0m^{\rho_j}/(2\kappa_jL^{\rho_j})$.
Independence and exponential Markov's inequality give
\[
 \Pp(Z_j^I>uL^\theta)
 \le\exp\{m\log B_*-(c_0/2)uL^{\theta-\rho_j}m^{\rho_j}\}.
\]
Keeping $m$ fixed, we pass to the limit $j\to\infty$. The portmanteau
inequality for the open set $(uL^\theta,\infty)$, together with
$\rho_j\to\theta$, yields
\[
 \Pp(Z^I>uL^\theta)
 \le\exp\{m\log B_*-(c_0/2)um^\theta\}.
\]
Since $\theta<1$, for sufficiently large $u$ we can choose a power
of two $m$ satisfying
\[
 m\le\left(\frac{c_0u}{4\log B_*}\right)^{1/(1-\theta)}<2m.
\]
The exponent is then at most $-m\log B_*\le-cu^{1/(1-\theta)}$.
This proves \eqref{eq:diameter-tail}, after adjusting the constants
for bounded $u$. Integrating the tail probability, we obtain, for $p>0$,
\[
 \E[(Z^I)^p]
 =p|I|^{p\theta}\int_0^\infty u^{p-1}
   \Pp(Z^I>u|I|^\theta)\,\mathrm du
 \le C_p|I|^{p\theta}.
\]
Together with $\diam\D^I\le Z^I$, this proves
\eqref{eq:diameter-moments}. Finally,
$\widetilde d_j^{[0,1]}(0,1)=L_{n_j}/\kappa_j$, whose expectation
is at least $c_*$ by \eqref{eq:record-positive}. The uniform
moment bounds from \eqref{eq:upper-tail} imply uniform
integrability, so convergence in distribution of the endpoint
distances gives \eqref{eq:limit-positive-mean}.
\end{proof}

Applying the diameter tail \eqref{eq:diameter-tail} to dyadic
subdivisions controls $\D^I(s,t)$ simultaneously over all time pairs
in a fixed interval. The following lemma gives the resulting modulus
of continuity.
\begin{lemma}\label{lem:critical-modulus}
For each $I\in\mathscr J$ there is a nonnegative random variable $H_I$
and constants $c_I,C_I>0$ such that
\begin{equation}\label{eq:critical-modulus-tail}
 \Pp(H_I>u)\le C_I\exp\{-c_Iu^{1/(1-\theta)}\},\qquad u\ge0.
\end{equation}
On an event of probability one, simultaneously for every $I\in\mathscr J$,
\begin{equation}\label{eq:critical-modulus}
 \D^I(s,t)\le H_I|s-t|^\theta
       \left(\log\frac{e|I|}{|s-t|}\right)^{1-\theta}
       \qquad(s,t\in I,\ s\ne t).
\end{equation}
When $|I|\le1$, the constants in \eqref{eq:critical-modulus-tail}
can be chosen independently of $I$.
\end{lemma}
\begin{proof}
Fix $I=[a,b]$, write $L=|I|$, and let
\[
 J_{m,k}=[a+kL2^{-m},a+(k+1)L2^{-m}],
       \qquad m\ge0,\quad 0\le k<2^m.
\]
These intervals belong to $\mathscr J$.
If $L\le1$, \eqref{eq:diameter-tail} applies to every $J_{m,k}$.
For $L>1$, choose $m_0$ with $L2^{-m_0}\le1$. Each coarser interval
$J_{m,k}$, $m<m_0$, is the union of $q=2^{m_0-m}$ level-$m_0$
intervals of length $\ell=L2^{-m_0}$. Domain monotonicity and the
triangle inequality bound its diameter by the sum of their internal diameters. If this sum exceeds
$u(q\ell)^\theta$, at least one diameter exceeds
$uq^{\theta-1}\ell^\theta$. Thus \eqref{eq:diameter-tail} and a union
bound give
\[
 \Pp(\diam\D^{J_{m,k}}>u|J_{m,k}|^\theta)
 \le Cq\exp\{-cu^{1/(1-\theta)}/q\}.
\]
Since $q\le2^{m_0}$, this and \eqref{eq:diameter-tail} give constants
$a_I,A_I>0$ such that
\[
 \Pp(\diam\D^{J_{m,k}}>u(L2^{-m})^\theta)
 \le A_I\exp\{-a_Iu^{1/(1-\theta)}\}
       \qquad(m\ge0,\ 0\le k<2^m,\ u\ge0).
\]
For $L\le1$, these constants are uniform in $I$. To control the
internal diameters at all dyadic scales, set
\[
 M_I=\sup_{m\ge0}\frac{1}{(m+1)^{1-\theta}}
       \max_{0\le k<2^m}
       \frac{\diam\D^{J_{m,k}}}{(L2^{-m})^\theta}.
\]
A union bound over these dyadic intervals gives, for sufficiently large
$u$,
\[
 \Pp(M_I>u)
 \le A_I\sum_{m\ge0}2^m
       \exp\{-a_Iu^{1/(1-\theta)}(m+1)\}
 \le2A_I\exp\{-a_Iu^{1/(1-\theta)}\}.
\]
In particular, $M_I<\infty$ almost surely.

Given $s,t\in I$ with $h=|s-t|>0$, choose $m$ so that
$h\le L2^{-m}<2h$. The interval between $s$ and $t$ lies in at most
two adjacent level-$m$ intervals. The triangle inequality and
domain monotonicity give
\[
 \D^I(s,t)\le2M_I(m+1)^{1-\theta}(L2^{-m})^\theta
       \le C_\theta M_I h^\theta[\log(eL/h)]^{1-\theta}.
\]
Thus $H_I=C_\theta M_I$ satisfies \eqref{eq:critical-modulus}; increasing
$C_I$ extends its tail bound to all $u\ge0$. Countability of
$\mathscr J$ gives a common event of probability one.
\end{proof}

From Lemma~\ref{lem:local-independence} through
Section~\ref{subsec:metric-measurability}, whenever finite graph
approximations are used, we work with the almost surely convergent
copies from Proposition~\ref{prop:local-field}. We reindex the chosen
subsequence by $j$. The graphs and distances at index $j$ are
constructed from $W^{(j)}$, while $W$ is the limiting Brownian path;
the dependence on $W^{(j)}$ is suppressed in the graph notation.
If a limit law of $(W,\widetilde d_j^{[0,1]})$ has already been
specified along a subsequence, we may pass to a further subsequence
to obtain joint convergence of all coordinates in
\eqref{eq:joint-local-limit}. The specified limit law is unchanged.

\Needspace{7\baselineskip}
The following lemma records independence of the Brownian increment
paths together with their internal distance limits on intervals with
disjoint interiors.
\begin{lemma}\label{lem:local-independence}
Let $I_i=[a_i,b_i]\in\mathscr J$, $1\le i\le m$, have disjoint
interiors. The random elements
\[
 \left((W_{a_i+u}-W_{a_i})_{0\le u\le b_i-a_i},
       (\D^J)_{J\in\mathscr J,\ J\subset I_i}\right),
       \qquad 1\le i\le m,
\]
are independent. For a dyadic $b>0$, the process
$(W_{b+u}-W_b)_{u\ge0}$ is independent of $W|_{[0,b]}$ together
with $(\D^J)_{J\in\mathscr J,\ J\subset[0,b]}$.
When two intervals have the same length, the corresponding random
elements have the same law after their time indices are translated
to the same interval.
\end{lemma}
\begin{proof}
At each finite index, the internal distances on an interval depend
only on its Brownian increments. Independence, including that of the
future increments, therefore passes to the limit by
\cite[Theorem 2.8(ii)]{Billingsley}, first for finite collections and
then for the countable collections by a monotone class argument.
Equality in law under time shifts by dyadic rationals passes to the
limit in the same way.
\end{proof}

\subsection{The quotient metric space}\label{subsec:quotient-space}

We examine the metric space obtained by identifying times at zero
limiting distance. Set $\D=\D^{[0,1]}$, let $\pi$ denote the quotient
map for the relation $\D(s,t)=0$, and write
\begin{equation}\label{eq:quotient}
 F=[0,1]/\{\D=0\},\qquad
 d_F(\pi(s),\pi(t))=\D(s,t),\qquad \nu=\pi_*(\Leb|_{[0,1]}).
\end{equation}
The map $\pi$ is continuous, and $(F,d_F)$ is a compact metric space.
By \eqref{eq:local-identifications}, $W_s=W_t$ implies
$\pi(s)=\pi(t)$. Thus $\pi$ factors through $W$, defining a map
\begin{equation}\label{eq:spatial-map}
 \Phi:K\longrightarrow F,\qquad \Phi(W_t)=\pi(t).
\end{equation}
The map $\Phi$ is continuous because $W:[0,1]\to K$ is a
quotient map.

\Needspace{7\baselineskip}
\begin{proposition}\label{prop:quotient-geometry}
Almost surely, the space $(F,d_F)$ is geodesic: every two points
$x,y\in F$ are joined by an isometric map from $[0,d_F(x,y)]$ into $F$.
For every $z\in F$, the inverse image $\Phi^{-1}(\{z\})\subset K$
is compact and connected.
\end{proposition}
\begin{proof}
Use the almost surely convergent representation constructed in
the proof of Proposition~\ref{prop:local-field}, and fix
$s,t\in[0,1]$.
For each $j$, choose the middle vertex of a shortest graph path between
their assigned cells, and let $u_j$ be the left endpoint of its time cell.
Its graph distances to the two endpoint cells differ by at most one.
Along a subsequence, $u_j\to u\in[0,1]$. By \eqref{eq:interpolation}
and uniform convergence of the interpolated distances,
\[
 \D(s,u)=\D(u,t)=\tfrac12\D(s,t).
\]
Thus every pair in $F$ has a midpoint. Completeness now gives a
geodesic by \cite[Theorem 2.4.16(1)]{BuragoBuragoIvanov}.

For the connectedness assertion, let $x=W_s$ and $y=W_t$ satisfy
$\Phi(x)=\Phi(y)$.
In the graph built from $W^{(j)}$, choose a shortest graph path between
the cells assigned to $s,t$. Let $C_j$ be the union of
$W^{(j)}([k/n_j,(k+1)/n_j])$ over its visited vertices $k$. The sets $C_j$
are nonempty, compact and connected because consecutive ranges
intersect. They contain $W_s^{(j)}$ and $W_t^{(j)}$, while uniform path
convergence places them in one compact ball. Blaschke's theorem
\cite[Theorem 7.3.8]{BuragoBuragoIvanov} gives a Hausdorff subsequential
limit $C\subset K$ containing $x,y$. The Hausdorff limit $C$ is connected.

For $z\in C$, choose $z_j\in C_j$ tending to $z$, a visited vertex
$k_j$, and $u_j\in[k_j/n_j,(k_j+1)/n_j]$ with
$W^{(j)}_{u_j}=z_j$. Along a further subsequence, $u_j\to u$ and
$W_u=z$. Since $|u_j-k_j/n_j|\le n_j^{-1}$, also $k_j/n_j\to u$.
The vertex $k_j$ lies on the chosen shortest graph path, so
$d_j^{[0,1]}(s,k_j/n_j)\le d_j^{[0,1]}(s,t)\to0$.
Uniform convergence therefore gives $\D(s,u)=0$.
Thus $C\subset\Phi^{-1}(\{\Phi(x)\})$. Any two points of this
inverse image belong to a connected subset of it, so it is connected;
compactness follows from continuity of $\Phi$. These conclusions hold
for all pairs on the chosen realization.
\end{proof}

\section{Gluing and measurability}\label{sec:gluing}

We first reconstruct the global limiting distance from internal
distances on the pieces of a time partition. The difficulty is to
control graph edges linking different pieces. Section~\ref{subsec:contact-functions}
provides the necessary approximations, which lead in
Section~\ref{subsec:reconstruction} to the gluing formula of
Theorem~\ref{thm:exact-gluing}. This proves Theorem~\ref{thm:main-metric}(ii).

The gluing formula also permits a conditional independence argument.
Using Lemma~\ref{lem:local-independence}, we show in
Section~\ref{subsec:metric-measurability} that each distance value has
zero conditional variance given the Brownian path
(Theorem~\ref{thm:measurable}). In particular, each internal limit is
a measurable function of the Brownian increments on its interval.
Section~\ref{subsec:wiener-convergence} then establishes convergence on
the original Wiener space in every $C^\alpha$ norm for $0<\alpha<\theta$,
including convergence of all positive moments of the error
(Corollary~\ref{cor:strong-convergence}). Together these results prove
Theorem~\ref{thm:main-metric}(iii). Finally,
Section~\ref{subsec:intersection-information} identifies the Brownian
self-intersection relation as sufficient to determine each global
limit (Proposition~\ref{prop:coarse-information}). Throughout we fix a
joint limit from Proposition~\ref{prop:local-field}.

\subsection{Contact times and Lipschitz functions}\label{subsec:contact-functions}

A graph edge may join cells in different partition intervals when
their Brownian ranges intersect. To control such edges, we use the
set of times involved in a self-intersection:
\[
 E_W=\{s\in[0,1]:\text{there exists }t\in[0,1]\setminus\{s\}
       \text{ with }W_s=W_t\}.
\]
For every integer $k\ge1$, the pairs $(s,t)\in[0,1]^2$ with
$|s-t|\ge1/k$ and $W_s=W_t$ form a compact set, as does its projection
onto the first coordinate. The countable union of these projections
is $E_W$, which is therefore Borel measurable. Its dimension satisfies
\begin{equation}\label{eq:double-times}
 \dimH E_W\le\frac12\qquad\text{almost surely}.
\end{equation}
Indeed, Rosen's theorem~\cite[Theorem 6]{Rosen} gives Hausdorff
dimension $1/2$ for the set of pairs $(s,t)\in[0,\infty)^2$ with
$s\ne t$ and $W_s=W_t$. Since the coordinate projections are Lipschitz, this gives
\eqref{eq:double-times} and the same upper bound on every finite time
interval simultaneously.

Fix a finite dyadic partition
\begin{equation}\label{eq:partition}
 \mathcal P:\quad 0=t_0<t_1<\cdots<t_m=1,
 \qquad I_i=[t_{i-1},t_i],
\end{equation}
and define the contact times in $I_i$ by
\begin{equation}\label{eq:contact-times}
 \mathcal T_i=\{u\in I_i:W_u\in\bigcup_{k\ne i}W(I_k)\}.
\end{equation}
Each $\mathcal T_i$ is compact and, except for partition endpoints,
is contained in $E_W$. Thus $\dimH\mathcal T_i\le1/2$ by
\eqref{eq:double-times}.
It follows that, if $f:I_i\to\R$ is $\alpha$-H\"older with
$\alpha>1/2$, then the compact set $f(\mathcal T_i)$ has dimension at
most $1/(2\alpha)<1$ and hence has Lebesgue measure zero.

\Needspace{7\baselineskip}
We will modify functions near the contact values by composing them
with a nondecreasing contraction that is locally constant there and
close to the identity.

\begin{lemma}\label{lem:flatten}
If $S\subset\R$ is compact and has Lebesgue measure zero, then for every $\varepsilon>0$
there are a bounded open set $O\supset S$ and a nondecreasing $1$-Lipschitz
function $g:\R\to\R$ such that
\begin{equation}\label{eq:flatten}
 \Leb(O)<\varepsilon,\qquad
 \sup_{r\in\R}|g(r)-r|<\varepsilon,
\end{equation}
and $g$ is constant on every connected component of $O$.
\end{lemma}
\begin{proof}
Choose a bounded open set $O\supset S$ with $\Leb(O)<\varepsilon$, and define
$g(r)=r-\int_0^r\one_O(v)\,dv$. Its derivative is $1-\one_O$
almost everywhere, which gives the assertions.
\end{proof}

\Needspace{11\baselineskip}
We also need to approximate a function that is Lipschitz for a limiting
internal distance by a function on the vertices of the corresponding
finite graph. The infimum construction below preserves the bound along
edges.
\begin{lemma}\label{lem:inf-extension}
Suppose that $\|d_j^{I_i}-\D^{I_i}\|_\infty\to0$ as $j\to\infty$,
and let $f:I_i\to\R$ satisfy
$|f(u)-f(v)|\le\D^{I_i}(u,v)$ for all $u,v\in I_i$.
Since $I_i=[t_{i-1},t_i]$ has dyadic endpoints and $n_j$ is a power
of two, $n_jt_{i-1}$ and $n_jt_i$ are integers for all sufficiently
large $j$. For each such $j$ and $k\in V(G_j^{I_i})$, define
\begin{equation}\label{eq:inf-extension}
 f_{j,i}(k)=\inf_{u\in I_i}\{f(u)+d_j^{I_i}(u,k/n_j)\}.
\end{equation}
Then, for every edge $\{k,\ell\}$ of $G_j^{I_i}$,
\[
 |f_{j,i}(k)-f_{j,i}(\ell)|\le\kappa_j^{-1}.
\]
Moreover, $t\mapsto f_{j,i}(\iota_j^{I_i}(t))$ converges uniformly
to $f$ on $I_i$.
\end{lemma}
\begin{proof}
For adjacent vertices $k,\ell$ of $G_j^{I_i}$,
$d_j^{I_i}(k/n_j,\ell/n_j)=\kappa_j^{-1}$. The triangle inequality gives
\[
 d_j^{I_i}(u,k/n_j)\le d_j^{I_i}(u,\ell/n_j)+\kappa_j^{-1}
 \qquad(u\in I_i).
\]
Taking the infimum over $u$ in \eqref{eq:inf-extension} yields
$f_{j,i}(k)\le f_{j,i}(\ell)+\kappa_j^{-1}$.
Interchanging $k$ and $\ell$ proves the asserted edge inequality.

For uniform convergence, put
$e_j=\|d_j^{I_i}-\D^{I_i}\|_\infty\to0$.
The Lipschitz property of $f$ gives
$f(u)+d_j^{I_i}(u,k/n_j)\ge f(k/n_j)-e_j$ for every $u\in I_i$.
Taking $u=k/n_j$ gives the upper bound, so
\[
 f(k/n_j)-e_j\le f_{j,i}(k)\le f(k/n_j).
\]
Since $|t-\iota_j^{I_i}(t)/n_j|\le n_j^{-1}$,
\[
 \sup_{t\in I_i}|f_{j,i}(\iota_j^{I_i}(t))-f(t)|
 \le e_j+\sup_{\substack{u,v\in I_i\\|u-v|\le n_j^{-1}}}|f(u)-f(v)|
 \longrightarrow 0
\]
by uniform continuity of $f$. This also covers the right endpoint of
$I_i$.
\end{proof}

\subsection{Reconstruction from internal distances}\label{subsec:reconstruction}

The two approximation lemmas now allow us to pass from Lipschitz
bounds for the internal distances to a bound for the global distance.
In Theorem~\ref{thm:local-global}, edges within a partition interval
are controlled by Lemma~\ref{lem:inf-extension}, while
Lemma~\ref{lem:flatten} deals with edges between intervals. We then
apply the resulting criterion to the chain distance in the proof of
Theorem~\ref{thm:exact-gluing}.
\begin{theorem}\label{thm:local-global}
With probability one, for every partition \eqref{eq:partition}, if
a continuous function $f:[0,1]\to\R$ is compatible with $W$ and satisfies
\begin{equation}\label{eq:local-lipschitz}
 |f(s)-f(t)|\le\D^{I_i}(s,t),\qquad s,t\in I_i,\quad 1\le i\le m,
\end{equation}
then
\begin{equation}\label{eq:global-lipschitz}
 |f(s)-f(t)|\le\D^{[0,1]}(s,t),\qquad s,t\in[0,1].
\end{equation}
\end{theorem}
\begin{proof}
Consider the almost surely convergent Skorokhod copies
constructed in the proof of Proposition~\ref{prop:local-field}.
The uniform convergence of $\widetilde d_j^I$ and
\eqref{eq:interpolation} imply $d_j^I\to\D^I$ uniformly for every
$I\in\mathscr J$. We restrict to the event where this convergence,
\eqref{eq:local-holder}, and \eqref{eq:double-times} hold.
Fix a partition $\mathcal P$ as in \eqref{eq:partition} and a
continuous compatible function $f$ satisfying
\eqref{eq:local-lipschitz}. For $m=1$ the assertion is immediate,
so assume $m\ge2$ and take $1/2<\alpha<\theta$.
By \eqref{eq:local-holder} and \eqref{eq:local-lipschitz}, $f$ is
$\alpha$-H\"older on every $I_i$. Thus
\[
 S=\bigcup_{i=1}^m f(\mathcal T_i)
\]
is compact and has Lebesgue measure zero. Fix $\varepsilon>0$ and
choose $O,g$ from Lemma~\ref{lem:flatten} for this set $S$.

To compare $f$ with the graph distances, use the approximation
$f_{j,i}$ from Lemma~\ref{lem:inf-extension} on the vertices of each
$G_j^{I_i}$. As the partition endpoints are dyadic, every time cell of
$G_{n_j}$ lies in exactly one $I_i$ for all sufficiently large $j$.
Define $f_{j,\mathcal P}(k)=f_{j,i}(k)$ when the cell of vertex $k$ lies in
$I_i$. If $\{k,\ell\}$ is an edge whose two cells lie in the same $I_i$,
Lemma~\ref{lem:inf-extension} gives
$|f_{j,\mathcal P}(k)-f_{j,\mathcal P}(\ell)|\le\kappa_j^{-1}$.

It remains to control edges joining different intervals. We claim that,
for all sufficiently large $j$, the values of $f_{j,\mathcal P}$ at the two endpoints
of every such edge lie in the same component of $O$. Otherwise, choose
infinitely many indices $j$ and an edge $\{k_j,\ell_j\}$ at each index
for which the claim fails. Since there are finitely many partition
intervals,
we may pass to a subsequence on which the two incident cells lie in
fixed intervals $I_i$ and $I_k$, with $i\ne k$.
By the edge rule there are times
$u_j\in[k_j/n_j,(k_j+1)/n_j]$ and
$v_j\in[\ell_j/n_j,(\ell_j+1)/n_j]$ such that
$W^{(j)}_{u_j}=W^{(j)}_{v_j}$.
A further subsequence has $u_j\to u\in I_i$ and $v_j\to v\in I_k$.
Uniform convergence of the Brownian paths gives $W_u=W_v$, so
$u\in\mathcal T_i$, $v\in\mathcal T_k$ and $f(u)=f(v)$.
The approximation in Lemma~\ref{lem:inf-extension}, together with
$|k_j/n_j-u_j|,|\ell_j/n_j-v_j|\le n_j^{-1}$, gives
\[
 f_{j,\mathcal P}(k_j)\longrightarrow f(u)=f(v)
       \longleftarrow f_{j,\mathcal P}(\ell_j).
\]
The common limit belongs to $S\subset O$, so an open interval about
it lies in one component of $O$. Both endpoint values eventually lie in
that component, giving a contradiction. This also covers adjacent cells separated by a partition endpoint,
since that endpoint belongs to both corresponding contact sets.

The composition $g\circ f_{j,\mathcal P}$ now has equal values at the endpoints
of every edge between different intervals. On every other edge it
changes by at most $\kappa_j^{-1}$, because $g$ is $1$-Lipschitz.
Summing along a shortest graph path gives
\[
 |g(f_{j,\mathcal P}(x))-g(f_{j,\mathcal P}(y))|\le\kappa_j^{-1}d_{G_{n_j}}(x,y)
\]
for all vertices $x,y$. Passing to the limit at cells representing
$s,t$ gives
$|g(f(s))-g(f(t))|\le\D^{[0,1]}(s,t)$.
Since $\sup_r|g(r)-r|<\varepsilon$, letting
$\varepsilon\downarrow0$ proves \eqref{eq:global-lipschitz}.
The argument holds for every $f$ on the chosen realization, and the
family of finite dyadic partitions is countable.
\end{proof}

To reconstruct the global distance $\D^{[0,1]}$ from the internal
distances $\D^{I_i}$, define $\Delta_{\mathcal P}(s,t)$ for the
partition $\mathcal P$ as the infimum of
\begin{equation}\label{eq:gluing-chain}
 \sum_{r=1}^q\D^{I_{i_r}}(a_r,b_r)
\end{equation}
over integers $q\ge1$, indices $i_r\in\{1,\ldots,m\}$ and pairs
$a_r,b_r\in I_{i_r}$ satisfying $a_1=s$, $b_q=t$ and
$W_{b_r}=W_{a_{r+1}}$ for $r<q$. We allow $a_r=b_r$.
The chains may revisit intervals and use the times in any order.
For a finite dyadic partition $\mathcal P$ of any $J\in\mathscr J$,
we use the same definition with $s,t\in J$ and with the internal
distances of the partition intervals.

Concatenation and reversal show that $\Delta_{\mathcal P}$ is a
pseudometric. To obtain a finite chain from $s$ to $t$, divide the interval between
them at the partition endpoints and use the internal distance on each
resulting portion. Moreover, $\Delta_{\mathcal P}$ is dominated by $\D^{I_i}$ on $I_i^2$
and vanishes whenever the two Brownian positions agree. Splitting $[u,v]$ at the partition points and using
$\sum_{r=1}^{q}\ell_r^\alpha\le q^{1-\alpha}
(\sum_{r=1}^{q}\ell_r)^\alpha$ for $q\le m$ give, for
$0<\alpha<\theta$ and $0\le u<v\le1$,
\begin{equation}\label{eq:gluing-continuity}
 \Delta_{\mathcal P}(u,v)
 \le\max_iH_{I_i,\alpha}\,m^{1-\alpha}|u-v|^\alpha,
\end{equation}
so $\Delta_{\mathcal P}$ is continuous.

\Needspace{7\baselineskip}
The criterion of Theorem~\ref{thm:local-global} identifies the chain
distance with the global distance, giving Theorem~\ref{thm:main-metric}(ii).
\begin{theorem}\label{thm:exact-gluing}
Almost surely, for every finite dyadic partition $\mathcal P$ as in
\eqref{eq:partition} and all $s,t\in[0,1]$,
\begin{equation}\label{eq:exact-gluing}
 \D^{[0,1]}(s,t)=\Delta_{\mathcal P}(s,t),\qquad s,t\in[0,1].
\end{equation}
For every $J\in\mathscr J$ and every finite dyadic partition
$\mathcal P$ of $J$, the analogous identity
$\D^J(s,t)=\Delta_{\mathcal P}(s,t)$ holds for all $s,t\in J$.
\end{theorem}
\begin{proof}
For any admissible chain from $s$ to $t$, its cost is the sum
$\sum_{r=1}^q\D^{I_{i_r}}(a_r,b_r)$ in
\eqref{eq:gluing-chain}. Domain monotonicity
\eqref{eq:domain-monotonicity}, the identifications
\eqref{eq:local-identifications}, and the triangle inequality imply
\[
 \D^{[0,1]}(s,t)
 \le\sum_{r=1}^q\D^{I_{i_r}}(a_r,b_r).
\]
Taking the infimum gives $\D^{[0,1]}\le\Delta_{\mathcal P}$.
For the reverse inequality, fix $s$ and apply
Theorem~\ref{thm:local-global} to
$f(u)=\Delta_{\mathcal P}(s,u)$. This function is compatible with $W$
because $\Delta_{\mathcal P}$ vanishes at equal Brownian positions.
The pseudometric triangle inequality and domination by $\D^{I_i}$
give \eqref{eq:local-lipschitz}, while \eqref{eq:gluing-continuity}
gives continuity. Since $f(s)=0$, we obtain
$\Delta_{\mathcal P}(s,t)\le\D^{[0,1]}(s,t)$.
The arguments hold for every $s,t$ on the same probability-one event.
For an arbitrary $J\in\mathscr J$, the same proof uses the internal
distances on $J$ and its subintervals. Countability of $\mathscr J$
and its finite dyadic partitions gives the simultaneous assertion.
\end{proof}

We will use the chain formula in Theorem~\ref{thm:measurable} after
conditioning on the Brownian path. For this purpose, its value must
depend measurably on the path and on the internal distances. More
precisely, fix $s,t$ and replace $W$ and the internal distances in
\eqref{eq:gluing-chain} by $w\in C([0,1],\R^3)$ and nonnegative
functions $d_i\in C(I_i^2)$. For a fixed chain length and fixed interval
indices, the time parameters lie in a compact product of intervals.
Both the endpoint requirements and the matching conditions
$w(b_r)=w(a_{r+1})$ are closed, whereas the sum of the $d_i$-costs is
continuous. It follows that the minimum cost, defined as $+\infty$
when no such chain exists, has closed sublevel sets as a function of
$(w,(d_i))$. Taking the infimum over the countably many possible chain
lengths and interval indices proves that the chain value is Borel
measurable.

\subsection{Measurability of the limiting distances}\label{subsec:metric-measurability}

The gluing formula reduces measurability to a variance estimate: for
fixed $s,t\in[0,1]$, we will prove
$\Var(\D^{[0,1]}(s,t)\mid W)=0$. Here and below, conditioning on $W$
means conditioning on $W|_{[0,1]}$. We use the Efron--Stein variance inequality for functions of
independent random variables; see
Boucheron, Lugosi and Massart~\cite[Section 3.1]{BoucheronLugosiMassart}.
After conditioning on the Brownian path, the internal distances are
independent, so this inequality can be applied to their conditional
product law. Gwynne and Miller~\cite[Section 5, equation (5.7)]{GwynneMiller}
use a similar conditional resampling argument in their study of random
metrics. To apply the inequality here, we first bound the change in
the gluing formula when one internal pseudometric is replaced.

\begin{lemma}\label{lem:one-piece}
Fix a continuous path $w:[a,b]\to\R^3$ with $a<b$ and a finite
partition $[a,b]=\bigcup_{k=1}^m I_k$ into nondegenerate closed
intervals with disjoint interiors. Let $d_k$ be a bounded pseudometric on $I_k$. Replace
$d_i$ by another bounded pseudometric $d_i'$, leaving $w$ and the
other pseudometrics unchanged. Define $\Delta$ and $\Delta'$ by
the chain formula \eqref{eq:gluing-chain}, with these two collections
in place of $(\D^{I_k})$ and $w$ in place of $W$. Then
\begin{equation}\label{eq:one-piece}
 -\diam d_i\le\Delta'(s,t)-\Delta(s,t)\le\diam d_i'
\end{equation}
for all $s,t\in[a,b]$.
\end{lemma}
\begin{proof}
Take a finite chain whose cost, measured with the original distances,
is at most $\Delta(s,t)+\varepsilon$.
If the chain does not use interval $i$, its cost is unchanged.
Otherwise, let $r_1$ and $r_2$ be the first and last indices of pairs
assigned to interval $i$. Replace the block of pairs from $r_1$ to
$r_2$, including those assigned to other intervals, by the single
pair $(a_{r_1},b_{r_2})\in I_i^2$, with cost measured by $d_i'$.
The constraints at the ends of the replaced block are unchanged, so
the new chain remains admissible. The removed cost is nonnegative,
whereas the replacement costs at most $\diam d_i'$.
Letting $\varepsilon\downarrow0$ proves the upper bound in
\eqref{eq:one-piece}; interchanging the collections proves the lower
bound.
\end{proof}

For the conditional argument, write
$B_i(u)=W_{t_{i-1}+u}-W_{t_{i-1}}$ for
$0\le u\le t_i-t_{i-1}$ and $Q_i=\D^{I_i}$, using the partition
\eqref{eq:partition}. By Lemma~\ref{lem:local-independence}, the
pairs $(B_i,Q_i)$ are independent.
Take the regular conditional law of $Q_i$ given $B_i$
\cite[Theorems 4.1.17--4.1.18]{Durrett}. Independence of the pairs
implies that the conditional law of $(Q_1,\ldots,Q_m)$ given all the
$B_i$ is the product of those laws. Since $W_0=\mathbf0$, the increment
paths $B_1,\ldots,B_m$ determine $W|_{[0,1]}$, and conversely.
Consequently, the $Q_i$ are conditionally independent given
$W|_{[0,1]}$, as required for the variance estimate.

\Needspace{7\baselineskip}
\begin{theorem}\label{thm:measurable}
For each joint subsequential law in Proposition~\ref{prop:local-field},
there is a Borel map
$\Psi:C_0([0,1],\R^3)\to C([0,1]^2)$ such that
\begin{equation}\label{eq:measurable}
 \D^{[0,1]}=\Psi(W)\qquad\text{almost surely}.
\end{equation}
For each $I=[a,b]\in\mathscr J$, there is also a Borel map
$\Psi_I:C_0([0,b-a],\R^3)\to C(I^2)$ such that
\[
 \D^I=\Psi_I\bigl((W_{a+u}-W_a)_{0\le u\le b-a}\bigr)
 \qquad\text{almost surely}.
\]
\end{theorem}
\begin{proof}
Fix $s,t$ and partition $[0,1]$ into $m=2^k$ equal intervals. By
Theorem~\ref{thm:exact-gluing}, $Z=\D^{[0,1]}(s,t)$ is a Borel function
of $W$ and the conditionally independent pseudometrics $Q_1,\ldots,Q_m$.
We may therefore resample the internal distances one at a time:
conditionally on $W$, let $Q_i'$ be an independent copy of $Q_i$, also
independent of the other pseudometrics, and let $Z^{(i)}$ be the chain
value obtained by replacing $Q_i$ with $Q_i'$. The Efron--Stein
inequality, applied under the regular conditional law given $W$ and
then integrated, gives
\begin{equation}\label{eq:conditional-variance}
 \E[\Var(Z\mid W)]\le\frac12\sum_{i=1}^m\E[(Z-Z^{(i)})^2].
\end{equation}
Both $Z$ and $Z^{(i)}$ are square integrable: chains passing through
consecutive partition intervals bound them by sums of internal
diameters, whose second moments satisfy
\eqref{eq:diameter-moments}. To bound the right-hand side,
Lemma~\ref{lem:one-piece} gives
\[
 |Z-Z^{(i)}|\le\max\{\diam Q_i,\diam Q_i'\},\qquad
 \E[(Z-Z^{(i)})^2]\le2\E[(\diam Q_i)^2].
\]
Substituting these estimates and \eqref{eq:diameter-moments} into
\eqref{eq:conditional-variance} gives
\begin{equation}\label{eq:variance-vanishes}
 0\le\E[\Var(Z\mid W)]\le C m^{1-2\theta}\longrightarrow0,
\end{equation}
since $\theta>1/2$. As $Z=\D^{[0,1]}(s,t)$ and $W$ do not depend on
$m$, neither does $\E[\Var(Z\mid W)]$, the quantity on the left-hand
side of \eqref{eq:variance-vanishes}. It must therefore be zero, and
$Z=\E[Z\mid W]$ almost surely. Taking $s,t$ in a countable dense
subset of $[0,1]$ and using continuity of $\D^{[0,1]}$, we see that
$\D^{[0,1]}$ is measurable with respect to the completed sigma-field
generated by $W$. The Doob--Dynkin lemma
\cite[Lemma 1.13]{Kallenberg}, applied to these path and function
spaces, yields a Borel map $\Psi$ satisfying \eqref{eq:measurable}.

For the local assertion, fix $J=[a,b]\in\mathscr J$, and condition
instead on the increment path
$B_J=(W_{a+u}-W_a)_{0\le u\le b-a}$. Partition $J$ into $m=2^k$
equal pieces and repeat the argument for $Z_J=\D^J(s,t)$.
For $m$ large enough that $|J|/m\le1$, the diameter moments in
\eqref{eq:diameter-moments} give
\[
 0\le\E[\Var(Z_J\mid B_J)]
 \le C|J|^{2\theta}m^{1-2\theta}\longrightarrow0.
\]
Applying the same argument to a countable dense subset of $J^2$
produces the map $\Psi_J$ stated in the theorem. Because $\mathscr J$
is countable, all these representations hold on one event of
probability one.
\end{proof}

\subsection{Convergence on the Wiener space}\label{subsec:wiener-convergence}

The measurability just proved allows us to pass from a joint weak
limit to convergence on the original Wiener space. We use the fact
that joint weak convergence implies convergence in probability when
the limit is a measurable function of a coordinate that remains
unchanged along the sequence.
\begin{lemma}\label{lem:weak-strong}
Let $S$ be a Polish space and $E$ a separable Banach space. Suppose that
$Z$ and $(Y_k)$ are defined on one probability space, take values in $S$
and $E$, respectively, and
$(Z,Y_k)\xrightarrow{\mathrm d}(Z,f(Z))$ for a Borel map $f:S\to E$. Then
$\|Y_k-f(Z)\|_E\to0$ in probability on that space.
\end{lemma}
\begin{proof}
By separability of $E$, approximate $f$ in probability under the law
of $Z$ by a finitely valued Borel map. The indicators of the inverse
images of its values can then be approximated in $L^1$ by bounded
continuous real functions, using regularity of the law of $Z$
\cite[Theorem 17.10]{Kechris}. Thus there is a bounded continuous
$g:S\to E$ with
$\E[1\wedge\|f(Z)-g(Z)\|_E]$ arbitrarily small. Joint weak convergence
gives
\[
 \E[1\wedge\|Y_k-g(Z)\|_E]\longrightarrow
 \E[1\wedge\|f(Z)-g(Z)\|_E].
\]
The triangle inequality therefore gives
\[
 \limsup_k\E[1\wedge\|Y_k-f(Z)\|_E]
 \le2\E[1\wedge\|f(Z)-g(Z)\|_E].
\]
Letting the approximation error tend to zero proves convergence in
probability.
\end{proof}

\Needspace{7\baselineskip}
For the normalized distances, put $X_j(W)=\widetilde d_j^{[0,1]}$.
Lemma~\ref{lem:weak-strong} gives uniform convergence in probability;
the moment estimates of Lemma~\ref{lem:kernel-moments} then allow us
to strengthen it to H\"older convergence.
\begin{corollary}\label{cor:strong-convergence}
If $(W,X_{j_k})\xrightarrow{\mathrm d}(W,\Psi(W))$, then on the original
Wiener space, for every $0<\alpha<\theta$,
\begin{equation}\label{eq:strong-convergence}
 \|X_{j_k}-\Psi(W)\|_{C^\alpha([0,1]^2)}\longrightarrow0
       \quad\text{in probability},\qquad
 \E\|X_{j_k}-\Psi(W)\|_{C^\alpha([0,1]^2)}^p
       \longrightarrow0\quad(p>0).
\end{equation}
For $1\le p<\infty$ and $0<\alpha<\theta$, the family $(X_j)_{j\ge1}$
is relatively compact in $L^p(\Pp;C^\alpha([0,1]^2))$.
For any finite collection $I_1,\ldots,I_r\in\mathscr J$, both the
convergence assertions and the $L^p$ relative compactness statement
hold jointly for the vectors
$(\widetilde d_j^{I_1},\ldots,\widetilde d_j^{I_r})$, with the product
norm formed from the $C^\alpha(I_\ell^2)$ norms.
\end{corollary}
\begin{proof}
Theorem~\ref{thm:measurable} and Lemma~\ref{lem:weak-strong} give
uniform convergence in probability. For every $q>\max\{p,1\}$,
Lemma~\ref{lem:kernel-moments} bounds
$\sup_j\E\|X_j\|_\infty^q$, and the same bound for
$\|\Psi(W)\|_\infty^q$ follows by weak convergence and lower
semicontinuity. The $p$-th powers of the uniform errors are therefore
uniformly integrable. Their expectations tend to zero by
\cite[Theorem 4.6.3]{Durrett}.

To pass to H\"older norms, we use the interpolation inequality valid
for $0<\alpha<\beta<1$ and $f\in C^\beta(I^2)$:
\begin{equation}\label{eq:holder-interpolation}
 [f]_{\alpha;I^2}
 \le (2\|f\|_\infty)^{1-\alpha/\beta}
       [f]_{\beta;I^2}^{\alpha/\beta}.
\end{equation}
This follows by interpolating the bounds $2\|f\|_\infty$ and
$[f]_{\beta;I^2}|z-z'|^\beta$ for each increment.
Choose $\alpha<\beta<\theta$ and put $f=X_{j_k}-\Psi(W)$.
Lemma~\ref{lem:kernel-moments} and Fatou's lemma bound every moment
of $[f]_{\beta;[0,1]^2}$ uniformly in $k$.
H\"older's inequality applied to \eqref{eq:holder-interpolation} gives
\[
 \E\bigl[[f]_{\alpha;[0,1]^2}^{p}\bigr]
 \le 2^{p(1-\alpha/\beta)}
       \bigl(\E\|f\|_\infty^p\bigr)^{1-\alpha/\beta}
       \bigl(\E\bigl[[f]_{\beta;[0,1]^2}^p\bigr]\bigr)^{\alpha/\beta}
 \longrightarrow0.
\]
Together with convergence of the uniform error, this proves
\eqref{eq:strong-convergence}.

To obtain $L^p$ compactness, we verify strong measurability in the
H\"older space. Each $X_j$ takes only finitely many values as a random function.
Hence the closed linear span $E_\alpha$ of all these values is a
separable subspace of $C^\alpha([0,1]^2)$. Although the full H\"older
space need not be separable, the $C^\alpha$ norm of
$X_{j_k}-\Psi(W)$ is measurable: its seminorm is the supremum over
distinct pairs in a fixed countable dense subset of $[0,1]^2$.
Convergence of its moments therefore permits a further subsequence
on which this norm tends to zero almost surely. It follows that
$\Psi(W)$ belongs to $E_\alpha$ almost surely and is strongly
measurable as an $E_\alpha$-valued random variable.
Proposition~\ref{prop:local-field} applies to every subsequence, so the
convergence just proved gives relative compactness in $L^p$.
Applying the argument to a finite collection of intervals gives the
last assertion.
\end{proof}

\subsection{Determination by the intersection relation}\label{subsec:intersection-information}

Every finite graph distance is determined by the Brownian
self-intersection relation $\mathcal I_{[0,1]}(W)$. The convergence in
Corollary~\ref{cor:strong-convergence} then shows that each
subsequential limit on the original Wiener space is measurable with
respect to that relation. We will also show that the relation itself
can be recovered from the labelled graphs $G_{2^m}$, $m\ge0$.

Let $\cH_m=\sigma(G_{2^m}(W))$ for $m\ge0$, including the vertex labels.
Group vertices $2i,2i+1$ into vertex $i$, retaining an edge between
two groups when an edge joins their vertices. This recovers $G_{2^m}$
from $G_{2^{m+1}}$, so $(\cH_m)$ is increasing. As the sigma-fields $\cH_m$ increase, conditioning on $\cH_m$
will also give approximations to $X_j$ with errors uniform in $j$.

\begin{proposition}\label{prop:coarse-information}
The completion of $\sigma(\bigcup_{m\ge0}\cH_m)$ is the completed
sigma-field generated by $\mathcal I_{[0,1]}(W)$. Every subsequential
limit of $X_j=\widetilde d_j^{[0,1]}$, realized on the original Wiener
space as in Corollary~\ref{cor:strong-convergence}, is measurable with
respect to this sigma-field. For $1\le p<\infty$ and
$0<\alpha<\theta$,
\begin{equation}\label{eq:coarse-information}
 \lim_{m\to\infty}\sup_j
 \E\|X_j-\E[X_j\mid\cH_m]\|_{C^\alpha([0,1]^2)}^p=0.
\end{equation}
The conditional expectations may be taken in $C^\alpha([0,1]^2)$;
they agree with their $C([0,1]^2)$ versions.
Each internal limit $\D^I$, $I\in\mathscr J$, realized on the original
Wiener space in the same way, is measurable with respect to the
completion of $\sigma(\mathcal I_I(W))$.
\end{proposition}
\begin{proof}
An edge of $G_{2^m}$ is determined by whether $\mathcal I_{[0,1]}(W)$ meets
the product of the two closed cells. Conversely, let $R_m$ be the union
of all rectangles $I_i^{(2^m)}\times I_k^{(2^m)}$ for which $i=k$ or
$\{i,k\}$ is an edge of $G_{2^m}$. Then
\[
 \mathcal I_{[0,1]}(W)=\bigcap_{m\ge0}R_m.
\]
Indeed, if $(s,t)$ belongs to every $R_m$, there are times
$s_m,t_m$ within $2^{-m}$ of $s,t$ with $W_{s_m}=W_{t_m}$.
Continuity then gives $W_s=W_t$. The sets $R_m$ are nonempty compact
sets with $R_{m+1}\subset R_m$, so they converge in the Hausdorff
metric to $\mathcal I_{[0,1]}(W)$. Hence the intersection relation is
measurable with respect to $\sigma(\bigcup_m\cH_m)$.
Conversely, each edge of $G_{2^m}$ is determined by the relation,
which proves equality of the completed sigma-fields.
Each $X_j$ is $\sigma(\mathcal I_{[0,1]}(W))$-measurable.
Any probability limit in \eqref{eq:strong-convergence} is measurable
with respect to the completion of this sigma-field, by passing to an
almost surely convergent further subsequence
\cite[Theorem 2.3.2]{Durrett}.
For an internal interval $I$, the finite functions
$\widetilde d_j^I$ depend only on $\mathcal I_I(W)$. Their convergence
in probability in $C(I^2)$, again by
Corollary~\ref{cor:strong-convergence}, proves the corresponding
claim for $\D^I$.

For the uniform approximation, we use the common separable subspace
$E_\alpha\subset C^\alpha([0,1]^2)$ constructed in the proof of
Corollary~\ref{cor:strong-convergence}. If $Y$ is an integrable
$E_\alpha$-valued random variable, its Bochner conditional expectation
$\E[Y\mid\cH_m]$ is characterized by
\[
 \int_A\E[Y\mid\cH_m]\,d\Pp=\int_A Y\,d\Pp
 \qquad(A\in\cH_m).
\]
Existence and the $L^p$ contraction property are given in
\cite[Chapter V, Section 1, Theorem 4]{DiestelUhl}.
Since the embedding $E_\alpha\hookrightarrow C([0,1]^2)$ is bounded
and linear, the conditional expectation computed in $E_\alpha$
agrees with the one computed in $C([0,1]^2)$.

For $Y\in L^p(\Pp;E_\alpha)$ measurable with respect to the
completion of $\sigma(\bigcup_m\cH_m)$, put
$P_mY=\E[Y\mid\cH_m]$. L\'evy's upward theorem
\cite[Theorem 4.6.8]{Durrett}, first for simple functions and then by
$L^p$ contraction, gives $P_mY\to Y$ in $L^p$. To make this
convergence uniform over $X_j$, we use relative compactness of
$\{X_j\}$ in $L^p$ from Corollary~\ref{cor:strong-convergence}.
Every element of this closure is measurable with respect to the
completion of $\sigma(\bigcup_m\cH_m)$, since each $X_j$ is.
Choose a finite $\varepsilon$-net $Y_1,\ldots,Y_M$ in that closure. If
$\|X_j-Y_i\|_{L^p}<\varepsilon$, then
\[
 \|X_j-P_mX_j\|_{L^p}\le2\varepsilon+\|Y_i-P_mY_i\|_{L^p}.
\]
Taking the maximum over the finite net, letting $m\to\infty$, and
then letting $\varepsilon\downarrow0$ proves
\eqref{eq:coarse-information}.
\end{proof}

\section{Separation from deterministic times}\label{sec:separation}

The dimension identity and nonuniqueness established in
Theorem~\ref{thm:main-growth} also hold for curves starting at $W_s$
when $s\in[0,1]$ is fixed in advance. Our result is uniform over the
other endpoint.
\begin{theorem}\label{thm:main-pinned}
Fix a deterministic $s\in[0,1]$. On an event of probability one,
\begin{equation}\label{eq:main-pinned}
 p_K(W_s,y)=\gamma_3\qquad\text{for every }y\in K\setminus\{W_s\},
\end{equation}
and, for each such $y$, there are continuum many simple curves
$\eta\in C([0,1],K)$ with pairwise distinct images, each
$\beta$-H\"older for every $0<\beta<1/\gamma_3$ and satisfying
\[
 \eta(0)=W_s,\qquad \eta(1)=y,\qquad
 \dimH\eta([0,1])=\dimB\eta([0,1])=\gamma_3.
\]
For every such $y$, each connected set $C\subset K$
containing $W_s,y$ has $\dimH C\ge\gamma_3$.
\end{theorem}

For each deterministic $s$, the event in Theorem~\ref{thm:main-pinned}
may depend on $s$, but its conclusion holds for all $y$ simultaneously.
We also prove the same conclusions for every $s$ in a random Borel
set $T_*\subset[0,1]$ of full Lebesgue measure, on a single event of
probability one. The set $T_*$ is first constructed in
Theorem~\ref{thm:deterministic-root}. In the proof of
Theorem~\ref{thm:main-pinned} in
Section~\ref{subsec:pinned-dimensions}, we replace it by a Borel
subset of full Lebesgue measure on which the curve statements
also hold simultaneously for every $s\in T_*$.

Proposition~\ref{prop:curve-upper} already supplies simple curves with
the required upper box dimension bound between every pair of distinct
points in $K$. For the matching lower bound, we use the criterion in
Lemma~\ref{lem:dimension-test}, with the limiting distance providing a
compatible H\"older function; this step is carried out in
Proposition~\ref{prop:point-to-point}. The necessary separation is
proved for time zero in Proposition~\ref{prop:endpoint-separation}.
Lemma~\ref{lem:two-sided-positive} treats the origin of a Brownian
motion indexed by the whole real line, and
Theorem~\ref{thm:deterministic-root} transfers the result to each
fixed starting time. Separation of all distinct Brownian positions
by the limiting distance on a single event remains open; see
Question~\ref{qu:critical-separation}.

\subsection{Separation from the initial time}\label{subsec:initial-separation}

We first study the distance from time zero, working with a
fixed joint limiting law from Proposition~\ref{prop:local-field}.
For positive dyadic times $0<t<T$, define
\begin{equation}\label{eq:root-distances}
 b(t,T)=\inf_{u\in[t,T]}\D^{[0,T]}(0,u),\qquad
 b(t)=\inf_{M\in\N,\,M>t}b(t,M).
\end{equation}
For each realization, the function
$u\mapsto\D^{[0,T]}(0,u)$ is continuous on the compact interval
$[t,T]$, so its infimum is a minimum. It is also a random variable:
taking this minimum is a continuous operation on
$C([0,T]^2)$ in the uniform norm. Domain monotonicity
\eqref{eq:domain-monotonicity} shows that $b(t,T)$ is nonincreasing
in $T$, and hence the infimum defining $b(t)$ can be taken along
any increasing sequence of integers tending to infinity. As $t$
increases, the target times are restricted, so $b(t)$ is
nondecreasing in $t$.

For comparison with the finite graphs, whenever $t,T$ are mesh
points we define
\begin{equation}\label{eq:finite-root-distances}
 b_j(t,T)=\inf_{u\in[t,T]}d_j^{[0,T]}(0,u).
\end{equation}
In the almost surely convergent representation used in the
proof of Proposition~\ref{prop:local-field},
\eqref{eq:interpolation} and uniform convergence of the
interpolated distances give, for each fixed $t,T$,
\[
 |b_j(t,T)-b(t,T)|
 \le \bigl\|d_j^{[0,T]}-\D^{[0,T]}\bigr\|_\infty
 \le \frac{2}{\kappa_j}
      +\bigl\|\widetilde d_j^{[0,T]}-\D^{[0,T]}\bigr\|_\infty
 \longrightarrow0.
\]

Although $b(t)$ is defined using intervals of arbitrarily
large length, the next lemma shows that restricting the path to
$[0,T]$ does not change its value when $W([0,t])$ and $W([T,\infty))$ are disjoint.
\begin{lemma}\label{lem:root-localization}
If $0<t<T$ are dyadic and
\begin{equation}\label{eq:root-avoidance}
 W([0,t])\cap W([T,\infty))=\varnothing,
\end{equation}
then $b(t)=b(t,T)$. With probability one, for every positive
dyadic $T$, this equality holds for all sufficiently small positive
dyadic $t$.
\end{lemma}
\begin{proof}
Use the almost surely convergent representation constructed in
the proof of Proposition~\ref{prop:local-field}. For an integer
$M>T$, condition \eqref{eq:root-avoidance} implies
\[
 \operatorname{dist}\bigl(W([0,t]),W([T,M])\bigr)>0,
\]
since these are disjoint compact subsets of $\R^3$.
Both $t$ and $T$ are mesh points for all sufficiently large $j$.
The uniform convergence of $W^{(j)}$ on $[0,M]$ therefore ensures
that no edge of $G_j^{[0,M]}$ joins a cell in $[0,t]$ to a cell
in $[T,M]$.

A path in $G_j^{[0,M]}$ from the initial cell to a cell with left
endpoint at least $t$ may be stopped when it first visits such a
cell. Every preceding cell has left endpoint less than $t$ and is
therefore contained in $[0,t]$. By the exclusion of edges just
proved, the cell reached at the stopping point must have left
endpoint less than $T$. The stopped path consequently uses only
cells in $[0,T]$, which gives
$b_j(t,T)\le b_j(t,M)$. The reverse inequality follows because
$G_j^{[0,T]}$ is a subgraph of $G_j^{[0,M]}$. Thus
\[
 b_j(t,M)=b_j(t,T)
 \quad\text{for all sufficiently large }j,
\]
and passage to the limit gives $b(t,M)=b(t,T)$.
This holds simultaneously for every integer $M>T$, since there are
only countably many such integers. If $t<M\le T$, domain
monotonicity gives $b(t,M)\ge b(t,T)$, so taking the infimum over
all integers $M>t$ proves $b(t)=b(t,T)$.

For the second assertion, transience and polarity of the origin
give $m_T:=\inf_{u\ge T}|W_u|>0$ almost surely for every fixed
$T>0$. By continuity at time zero,
$\sup_{0\le u\le t}|W_u|<m_T$ for all sufficiently small $t$,
and hence \eqref{eq:root-avoidance} holds. Countability of the
dyadic times $T$ gives the asserted simultaneous conclusion.
\end{proof}

Grouping time cells in pairs changes graph distances by at most a
constant factor. Brownian scaling therefore gives the same
probability of positive limiting separation at scales $t$ and $2t$.
\begin{lemma}\label{lem:root-zero-scaling}
For every positive dyadic $t$,
\begin{equation}\label{eq:root-zero-scaling}
 \Pp(b(2t)>0)=\Pp(b(t)>0).
\end{equation}
\end{lemma}
\begin{proof}
For each $I\in\mathscr J$, let $G_j^I(W)$ denote the intersection
graph of the Brownian ranges of the cells in $I$, as in
\eqref{eq:local-kernel}. We compare these graphs with
$G_j^I(\widehat W)$ for
$\widehat W_u=2^{-1/2}W_{2u}$, using the same mesh $1/n_j$ and
normalization $\kappa_j$. Proposition~\ref{prop:local-field}
gives tightness of each family of Brownian paths and internal
distances, including their intersection relations. The two families
are therefore jointly tight, since compact sets for their marginal
laws give compact sets for the product. After retaining the originally
specified limit, we may pass to a further subsequence on which
both families converge jointly. Brownian scaling implies that
their marginal laws agree at each finite index and hence in the
limit as well. Define $\widehat b_j(t,T)$, $\widehat b(t,T)$ and
$\widehat b(t)$ for the second family as in
\eqref{eq:finite-root-distances} and \eqref{eq:root-distances}.

The cell $[k/n_j,(k+1)/n_j]$ for $\widehat W$ corresponds,
up to a spatial scaling, to $[2k/n_j,(2k+2)/n_j]$ for $W$.
Thus a vertex of $G_j^{[0,T]}(\widehat W)$ represents a pair
of consecutive vertices of $G_j^{[0,2T]}(W)$. Projecting the
vertices of a path in the latter graph by
$k\mapsto\lfloor k/2\rfloor$ and deleting consecutive repetitions
gives a path in the former graph with no more edges. Conversely,
for each edge of a path in the coarser graph, choose an intersecting
pair of finer cells in the two corresponding blocks. The entry and
exit cells in any block can be joined by at most one additional edge.
A coarse path of $l$ edges therefore lifts to a fine path with at
most $2l+1$ edges. Since $t<T$ are dyadic, the target sets align
on the two meshes for all sufficiently large $j$. Consequently,
\begin{equation}\label{eq:root-two-mesh-comparison}
 \widehat b_j(t,T)\le b_j(2t,2T)
       \le2\widehat b_j(t,T)+\kappa_j^{-1}.
\end{equation}
For fixed $t,T$, the joint convergence and \eqref{eq:interpolation}
allow passage to the limit in \eqref{eq:root-two-mesh-comparison}.
We then take the infimum over integers $T>t$. Because the minima
are nonincreasing in the terminal time, taking the infimum over
even integer endpoints in the definition of $b(2t)$ gives the same
value as taking it over all integers. It follows that
\[
 \widehat b(t)\le b(2t)\le2\widehat b(t).
\]
Hence $b(2t)>0$ if and only if $\widehat b(t)>0$.
The two minima $\widehat b(t)$ and $b(t)$ have the same law, which
proves \eqref{eq:root-zero-scaling}.
\end{proof}

We next prove separation from time zero. A positive probability
at one scale will give the same probability at every scale $2^{-k}$
by Lemma~\ref{lem:root-zero-scaling}. Lemma~\ref{lem:root-localization}
will then identify the intersection of these events with an event in
the Brownian germ $\sigma$-field $\mathcal F_{0+}$, defined in
Section~\ref{subsec:scalar-zero-one}. Blumenthal's zero--one law
will turn the positive probability into probability one.
\begin{proposition}\label{prop:endpoint-separation}
The inequality $b(t)>0$ holds with probability one simultaneously for
all positive dyadic $t$. In particular,
\begin{equation}\label{eq:endpoint-separation}
 \inf_{\varepsilon\le u\le1}\D^{[0,1]}(0,u)>0,
 \qquad
 \inf_{0\le u\le1-\varepsilon}\D^{[0,1]}(u,1)>0
       \quad(0<\varepsilon\le1).
\end{equation}
Almost surely, for every $I=[a,b]\in\mathscr J$,
$\D^I(a,u)>0$ for all $u\in(a,b]$ and
$\D^I(u,b)>0$ for all $u\in[a,b)$.
\end{proposition}
\begin{proof}
Fix $0<\alpha<\theta$ and take the H\"older constant to be
\[
 H_{[0,1],\alpha}:=
 \sup_{\substack{s,t\in[0,1]\\s\ne t}}
 \frac{\D^{[0,1]}(s,t)}{|s-t|^\alpha},
\]
as permitted in the proof of Proposition~\ref{prop:local-field}.
The positive endpoint mean \eqref{eq:limit-positive-mean} and
the moment bound \eqref{eq:local-holder} give $c,M>0$ such that
\begin{equation}\label{eq:root-positive-probability}
 E_0:=\{\D^{[0,1]}(0,1)>4c,\ H_{[0,1],\alpha}\le M\},
 \qquad \Pp(E_0)>0.
\end{equation}
Choose a dyadic $\varepsilon\in(0,1)$ with $M\varepsilon^\alpha<c$
and put $t_0=1-\varepsilon$. By the H\"older estimate, $E_0$
implies $b(t_0,1)>3c$. Theorem~\ref{thm:measurable} shows that
$\D^{[0,1]}$ is measurable with respect to the completed
$\sigma$-field generated by $W|_{[0,1]}$. Since the chosen
$H_{[0,1],\alpha}$ is a Borel function of $\D^{[0,1]}$, the event
$E_0$ in \eqref{eq:root-positive-probability} is measurable with
respect to this $\sigma$-field as well.

By polarity of a point applied to the reversed increments,
$W_1\notin W([0,t_0])$ almost surely. Lemma~\ref{lem:escape}
then implies that, conditionally on $W|_{[0,1]}$, the probability of
$W([1,\infty))\cap W([0,t_0])=\varnothing$ is strictly positive.
Because $\Pp(E_0)>0$ and $E_0$ is measurable from the path up to
time one, its intersection with this avoidance event has positive
probability. On the intersection, Lemma~\ref{lem:root-localization}
gives $b(t_0)=b(t_0,1)>3c$. Since $t_0<1$ and $b$ is
nondecreasing, $\Pp(b(1)>0)>0$.

To upgrade this to probability one, let
$A:=\bigcap_{k\ge1}\{b(2^{-k})>0\}$. These events decrease with $k$,
and Lemma~\ref{lem:root-zero-scaling} gives each of them the same
probability $p=\Pp(b(1)>0)>0$. Hence $\Pp(A)=p$.
For each fixed positive dyadic $T$, Lemma~\ref{lem:root-localization}
gives, up to a null set,
\begin{equation}\label{eq:root-germ-localization}
 A=\liminf_{k\to\infty}\{b(2^{-k},T)>0\}.
\end{equation}
Here the random variables $b(2^{-k},T)$ on the right-hand side of
\eqref{eq:root-germ-localization} are defined for all sufficiently
large $k$ and are functions of $\D^{[0,T]}$. By
Theorem~\ref{thm:measurable}, they are measurable with respect to
the completed $\sigma$-field generated by $W|_{[0,T]}$.
Therefore $A\in\mathcal F_T$ up to completion for every positive
dyadic $T$, and consequently $A\in\mathcal F_{0+}$ up to completion.
Blumenthal's zero--one law
\cite[Theorem 2.7]{MortersPeres} gives $p=1$. Since $b(t)$ is
nondecreasing in $t$, positivity holds for every positive dyadic
$t$, and the first assertion of \eqref{eq:endpoint-separation} follows. The second assertion of \eqref{eq:endpoint-separation} follows by
time reversal on $[0,1]$ and subtraction of $W_1$. Indeed, the finite
graphs are invariant under reversal; only the convention for cells
at division points can change the distances indexed by time, by at most
$O(\kappa_j^{-1})$. The reversed distances may be included in a joint
limit without changing the original marginal law.

Finally, the internal distance on any $I=[a,b]\in\mathscr J$ has,
after translating time by $a$, the same law as that on $[0,b-a]$,
by Lemma~\ref{lem:local-independence}. Positivity of $b(t)$ therefore
separates $a$ from every other time in $I$; the argument for the reversed path
separates $b$. Since $\mathscr J$ and the positive dyadic gaps are
countable, these conclusions hold on one event of probability one.
\end{proof}

\Needspace{10\baselineskip}
\subsection{Functions supported in a time interval}\label{subsec:supported-functions}

Proposition~\ref{prop:endpoint-separation} gives positive internal
distance between the endpoints of each dyadic interval. We use this
to construct a function $v$ that vanishes outside a smaller interval
$I\subsetneq J$ but satisfies $v(q)>0$ at some $q\in I$. We also
require $|v(s)-v(t)|\le\D^J(s,t)$, which then implies
$\D^J(q,t)\ge v(q)>0$ for every $t\in J\setminus I$.
\begin{proposition}\label{prop:localized-function}
With probability one, for every pair of intervals
$I=[a,b],J\in\mathscr J$ with $I\subsetneq J$, there is a nonnegative,
nonconstant function $v:J\to\R$ such that
\begin{equation}\label{eq:localized-function}
 v=0\text{ on }J\setminus(a,b),\qquad
 |v(s)-v(t)|\le\D^J(s,t)\quad(s,t\in J).
\end{equation}
Moreover, $v(t)=0$ whenever $t\in J$ and $W_t=W_u$ for some
$u\in J\setminus(a,b)$.
\end{proposition}
\begin{proof}
Fix $J=[c,d]\in\mathscr J$ and $I=[a,b]\in\mathscr J$ with
$c\le a<b\le d$ and $I\ne J$. The times in $I$ at which the
Brownian path visits a position attained outside $(a,b)$ form the
compact set
\[
 \mathcal T_I=\{t\in I:W_t\in W([c,a]\cup[b,d])\}.
\]
In particular, $a,b\in\mathcal T_I$. By \eqref{eq:double-times},
$\dimH\mathcal T_I\le1/2$ almost surely. On $I$ define
$u(t)=\D^I(a,t)$ and put $L=u(b)$. Proposition~\ref{prop:endpoint-separation}
gives $L>0$. Moreover, \eqref{eq:local-holder} makes $u$
$\alpha$-H\"older for every $1/2<\alpha<\theta$, so
\[
 \dimH u(\mathcal T_I)\le\frac{1}{2\alpha}<1.
\]
Because $u(\mathcal T_I)$ is compact and contains $0$ and $L$,
we can choose $0<r_-<r_+<L$ such that
$[r_-,r_+]\cap u(\mathcal T_I)=\varnothing$. Set
\[
 g(r)=\max\{0,\min(r-r_-,r_+-r)\},\qquad
 v(t)=\begin{cases}g(u(t)),&t\in I,\\0,&t\in J\setminus I.\end{cases}
\]
The function $g$ is $1$-Lipschitz and vanishes outside $(r_-,r_+)$.
By the choice of $r_-$ and $r_+$, $v$ vanishes on $\mathcal T_I$,
including $a$ and $b$. Since $u$ and $g$ are continuous, defining
$v$ to be zero on $J\setminus I$ therefore gives a continuous
function on $J$.

If $s,t\in I$ and $W_s=W_t$, then
\eqref{eq:local-identifications} gives $u(s)=u(t)$, so $v(s)=v(t)$.
If the same Brownian position is visited both in $I$ and in
$J\setminus I$, its representative in $I$ belongs to
$\mathcal T_I$, where $v$ is zero. Thus $v$ is compatible with
$W$ on $J$. By the triangle inequality for $\D^I$ and the
Lipschitz bound for $g$,
\[
 |v(s)-v(t)|\le\D^I(s,t)\qquad(s,t\in I).
\]
On the complementary intervals $[c,a]$ and $[b,d]$, whenever
nondegenerate, $v$ is identically zero. We can therefore sum the
changes of $v$ along any chain in \eqref{eq:gluing-chain} for the
partition of $J$ at $a,b$: changes at successive pairs with equal
Brownian positions are zero. The gluing identity in
Theorem~\ref{thm:exact-gluing} now gives the distance bound in
\eqref{eq:localized-function}. By construction, $v=0$ on
$J\setminus(a,b)$. Since $v$ is compatible with $W$, it also
vanishes at any $t\in J$ for which $W_t=W_r$ with
$r\in J\setminus(a,b)$.

Finally, $u(a)=0$, $u(b)=L$ and continuity of $u$ imply
$[0,L]\subset u(I)$. Because $g$ is positive on $(r_-,r_+)$,
$v$ is nonconstant. There are countably many pairs $(I,J)$,
so these properties hold simultaneously on one event of probability
one.
\end{proof}

For any $0\le s<t\le1$, choose a nondegenerate dyadic interval
$I=[a,b]\subset(s,t)$ and apply
Proposition~\ref{prop:localized-function} with $J=[0,1]$.
The resulting function $v$ is nonconstant on $I$ and satisfies
$|v(u)-v(w)|\le\D(u,w)$ for $u,w\in[0,1]$. Hence the quotient
image of $[s,t]$ has positive diameter. In particular, on one
event of probability one,
\begin{equation}\label{eq:no-interval-collapse}
 \diam_{d_F}\pi([s,t])>0\qquad(0\le s<t\le1).
\end{equation}

Moreover, if $O\subset K$ is nonempty and relatively open,
its inverse image under $W|_{[0,1]}$ contains a nondegenerate
dyadic interval $I$. Applying
Proposition~\ref{prop:localized-function} with $J=[0,1]$, we obtain
a nonnegative nonconstant function $v$ that is $1$-Lipschitz for
$\D$ and vanishes outside $I$. This induces a function on
$(F,d_F)$ with the same Lipschitz bound, vanishing on
$\Phi(K\setminus O)$.

\subsection{Separation for two-sided Brownian motion}\label{subsec:two-sided-separation}

For a fixed $s\in(0,1)$, a path in the intersection graph
joining the cell assigned to $s$ to another cell may use cells
lying either before or after $s$. We first study this situation
for Brownian motion indexed by the whole real line. The proof
uses the functions $v$ from Proposition~\ref{prop:localized-function},
which vanish outside a chosen interval and satisfy
\eqref{eq:localized-function}, to obtain separation from time zero.

Let $B=(B_u)_{u\in\R}$ be two-sided Brownian motion, with
$B_0=\mathbf0$ and with $(B_u)_{u\ge0}$ and $(B_{-u})_{u\ge0}$
independent standard Brownian motions. We use the cells
$[k/n_j,(k+1)/n_j]$, $k\in\Z$, with the normalizations $\kappa_j$
and include all finite intervals whose endpoints are dyadic,
including negative endpoints. For every such $[a,b]$,
$(B_{a+u}-B_a)_{0\le u\le b-a}$ has the law of standard Brownian
motion. Thus the tightness estimates in the proof of
Proposition~\ref{prop:local-field} apply to this enlarged countable
family, and we write $\D^I$ for the joint subsequential limits.
The repeated visit estimate \eqref{eq:double-times} and the
arguments of Theorems~\ref{thm:exact-gluing} and
\ref{thm:measurable} and Proposition~\ref{prop:endpoint-separation}
then apply to these intervals as well. In particular, gluing,
measurability and endpoint separation hold for their limiting
distances.

If $q$ is dyadic, then $n_jq$ is an integer for all sufficiently
large $j$. The finite intersection graphs and internal distances
are therefore unchanged in law by translating the time coordinates
by $q$ and subtracting $B_q$ from the path. The resulting invariance
passes to the joint limits. Any previously specified marginal limit
is preserved when we enlarge the family of intervals and take a
further subsequence.

For positive dyadic $t<T$ define
\begin{equation}\label{eq:two-sided-root}
 b_2(t,T)=\inf_{u\in[-T,-t]\cup[t,T]}\D^{[-T,T]}(0,u),
 \qquad b_2(t)=\inf_{M\in\N,\,M>t}b_2(t,M).
\end{equation}
For the finite graphs, write $b_{2,j}(t,T)$ for the infimum in
\eqref{eq:two-sided-root} with $\D^{[-T,T]}$ replaced by
$d_j^{[-T,T]}$. As in \eqref{eq:finite-root-distances},
$b_{2,j}(t,T)\to b_2(t,T)$ for fixed $t,T$.
To apply the reasoning of Section~\ref{subsec:initial-separation},
we need two facts. First, if the path near zero does not meet the
path at times with $|u|\ge T$, the infimum defining $b_2(t)$ can
be computed using $\D^{[-T,T]}$. Second, the event $\{b_2(t)>0\}$
has the same probability at the scales $t$ and $2t$. Precisely,
\begin{align}
 B([-t,t])\cap B(\R\setminus(-T,T))=\varnothing
   &\ \Longrightarrow\ b_2(t)=b_2(t,T),\label{eq:two-localization}\\
 \Pp(b_2(2t)>0)&=\Pp(b_2(t)>0).\label{eq:two-scaling}
\end{align}
To prove \eqref{eq:two-localization}, suppose its disjointness
condition holds and fix an integer $M>T$. The sets $B([-t,t])$
and $B([-M,-T]\cup[T,M])$ are disjoint and compact. As in the
proof of Lemma~\ref{lem:root-localization}, convergence of the
Brownian paths on $[-M,M]$ ensures that, for large $j$, no edge
joins a time cell in $[-t,t]$ to a time cell in
$[-M,-T]\cup[T,M]$. Consider a graph path from the cell assigned
to zero to one assigned to a time $u$ with $|u|\ge t$. Stop at its
first cell whose left endpoint is at most $-t$ or at least $t$.
All preceding cells lie in $[-t,t]$, so the stopping cell must
belong to $[-T,T]$. This proves
$b_{2,j}(t,T)\le b_{2,j}(t,M)$; the reverse inequality follows
because the cells of the graph restricted to $[-T,T]$ also occur
in the larger graph. Consequently,
$b_{2,j}(t,M)=b_{2,j}(t,T)$ for all sufficiently large $j$.
Passing to the limit gives $b_2(t,M)=b_2(t,T)$ for integers
$M>T$. For $t<M\le T$, domain monotonicity gives
$b_2(t,M)\ge b_2(t,T)$. Taking the infimum over all integers
$M>t$ proves \eqref{eq:two-localization}.

For fixed $T$, transience and polarity applied separately to
$(B_u)_{u\ge0}$ and $(B_{-u})_{u\ge0}$ yield
$\inf_{|u|\ge T}|B_u|>0$ almost surely. As
$\sup_{|u|\le t}|B_u|\to0$ when $t\downarrow0$, the disjointness
condition in \eqref{eq:two-localization} holds for all sufficiently
small $t$.

To prove \eqref{eq:two-scaling}, pair adjacent time cells
with indices in $\Z$, as in the proof of
Lemma~\ref{lem:root-zero-scaling}, and consider
$\widehat B_u=2^{-1/2}B_{2u}$. Let $\widehat b_2$ denote
the quantities in \eqref{eq:two-sided-root} for the jointly constructed
limiting distances of $\widehat B$. The finite comparison gives
$\widehat b_2(t)\le b_2(2t)\le2\widehat b_2(t)$, and
$\widehat b_2$ has the same law as $b_2$.

These two comparisons allow the germ zero--one argument once
separation has been shown to have positive probability.
\begin{lemma}\label{lem:two-sided-positive}
The inequality $b_2(t)>0$ holds with probability one simultaneously for
all positive dyadic $t$.
\end{lemma}
\begin{proof}
Fix the dyadic intervals $I=[-1/2,1/2]$ and $J=[-1,1]$, and set
\[
 u(t)=\D^I(-1/2,t),\qquad
 \mathcal T_I=\{t\in I:B_t\in B([-1,-1/2]\cup[1/2,1])\}.
\]
Fix $r_-,r_+,c_0\in\mathbb Q$ with $0<r_-<r_+$ and
$c_0>0$, and a dyadic $q\in(-1/2,1/2)$. Define
$v(t)=\max\{0,\min\{u(t)-r_-,r_+-u(t)\}\}$ on $I$, with
$v=0$ on $J\setminus I$. We consider the event $E$ on which
$0<r_-<r_+<u(1/2)$,
$[r_-,r_+]\cap u(\mathcal T_I)=\varnothing$, and $v(q)>c_0$.
Theorem~\ref{thm:measurable} and the measurability of the compact
contact set, as in \eqref{eq:intersection-relation}, show that
$E\in\sigma(B|_J)$.

By Proposition~\ref{prop:endpoint-separation}, $u(1/2)>0$ almost
surely. The dimension argument in the proof of
Proposition~\ref{prop:localized-function} gives a gap in
$u(\mathcal T_I)$, and continuity of $u$ then gives a dyadic
$q\in I$ where the associated cutoff is strictly positive.
Consequently, the countable collection of events $E$ obtained by
varying the rational parameters and dyadic $q$ covers an
event of probability one. At least one deterministic choice therefore
has $\Pp(E)>0$. We keep this choice fixed. On $E$, the proof of
Proposition~\ref{prop:localized-function} gives
$|v(s)-v(t)|\le\D^J(s,t)$ for all $s,t\in J$.

The points $B_{-1}$ and $B_1$ are almost surely outside $B(I)$,
by polarity applied to the appropriate forward and reversed Brownian
increments. Conditional on $B|_J$, the continuations beyond $1$
and before $-1$ are independent Brownian motions. Lemma~\ref{lem:escape}
therefore gives a strictly positive conditional probability that both
avoid $B(I)$ forever. Define the avoidance event
\begin{equation}\label{eq:two-sided-avoidance-event}
 E_{\mathrm{avoid}}=
 \bigl\{B(({-\infty},-1]\cup[1,\infty))\cap B(I)=\varnothing\bigr\}.
\end{equation}
Since $E$ is measurable with respect to $B|_J$ and
$\Pp(E)>0$, we have $\Pp(E\cap E_{\mathrm{avoid}})>0$.

On $E\cap E_{\mathrm{avoid}}$, extend $v$ by zero to $[-M,M]$ for
every integer $M>1$. This extension is compatible with $B$:
$v$ can be nonzero only at times in $I$, and no point of
$B(I)$ is visited outside $J$. It remains $1$-Lipschitz for $\D^J$
on $J$ and is constant on $[-M,-1]$ and $[1,M]$. By the gluing
formula of Theorem~\ref{thm:exact-gluing}, the change of $v$
along each chain is bounded by the sum of its distance costs:
$v$ has equal values at successive times representing the same
Brownian position, and its change within each partition interval
is bounded by the corresponding internal distance. Taking the
infimum over chains gives
\[
 |v(a)-v(b)|\le\D^{[-M,M]}(a,b)\qquad(a,b\in[-M,M]).
\]
Since $v(q)>c_0$ and $v=0$ outside $(-1/2,1/2)$, this gives
\[
 \inf_{u\in[-M,M]\setminus(-1/2,1/2)}\D^{[-M,M]}(q,u)\ge c_0
 \qquad(M\in\N,\ M>1).
\]
Since $|q|<1/2$, every time $u$ with $|u-q|\ge1$ lies outside
$(-1/2,1/2)$. In particular, for every integer $N>1$, we may choose
an integer $M>N+|q|$. The inclusion
$[q-N,q+N]\subset[-M,M]$ and domain monotonicity give
\[
 \inf_{\substack{u\in[q-N,q+N]\\|u-q|\ge1}}
 \D^{[q-N,q+N]}(q,u)
 \ge\inf_{\substack{u\in[-M,M]\\|u-q|\ge1}}
 \D^{[-M,M]}(q,u)\ge c_0.
\]
Let $b_{2,q}(1)$ denote the quantity defined in
\eqref{eq:two-sided-root} for the translated process
$u\mapsto B_{q+u}-B_q$, with its translated internal distances.
The inequalities just proved hold for every integer $N>1$ and give
$b_{2,q}(1)\ge c_0$ on $E\cap E_{\mathrm{avoid}}$.
Because the translated process is two-sided Brownian motion and
the dyadic translation preserves the joint law of its internal
distances, $b_{2,q}(1)$ and $b_2(1)$ have the same distribution.
Since $\Pp(E\cap E_{\mathrm{avoid}})>0$, we conclude that
$\Pp(b_2(1)>0)>0$.

By \eqref{eq:two-scaling}, the decreasing events
$\{b_2(2^{-k})>0\}$ all have the same strictly positive probability. Their
intersection is, modulo a null set, measurable with respect to
$B|_{[-T,T]}$ for every positive dyadic $T$, by the localization
\eqref{eq:two-localization} and Theorem~\ref{thm:measurable}.
It therefore belongs to the germ sigma field of the
six-dimensional Brownian motion $(B_u,B_{-u})_{u\ge0}$.
Blumenthal's zero--one law~\cite[Theorem 2.7]{MortersPeres}
forces its probability to be one. Monotonicity in $t$ then gives
separation at all positive dyadic scales on the same event.
\end{proof}

\subsection{Separation from a prescribed time}\label{subsec:prescribed-separation}

For a fixed deterministic $s\in[0,1]$, we compare the graph
built from the intervals $[k/n_j,(k+1)/n_j]$ with the graph
built from $[s+k/n_j,s+(k+1)/n_j]$. In the second partition,
$s$ is a division point. The comparison allows us to apply
Lemma~\ref{lem:two-sided-positive} to distances from $s$ in
the original graph.
\begin{theorem}\label{thm:deterministic-root}
Fix a joint limiting law from Proposition~\ref{prop:local-field}.
For each fixed deterministic $s\in[0,1]$, with probability one,
\begin{equation}\label{eq:deterministic-root}
 \inf_{\substack{t\in[0,1]\\|t-s|\ge\varepsilon}}
       \D(s,t)>0\qquad\text{for every }\varepsilon>0.
\end{equation}
Furthermore, there is almost surely a Borel set $T_*\subset[0,1]$ of full
Lebesgue measure such that
\begin{equation}\label{eq:typical-root}
 s\in T_*,\quad t\in[0,1]\setminus\{s\}
       \quad\Longrightarrow\quad\D(s,t)>0.
\end{equation}
The measure $\nu$ in \eqref{eq:quotient} is nonatomic with full support, and
$\pi^{-1}(\{z\})$ is a singleton for $\nu$-almost every $z\in F$.
\end{theorem}
\begin{proof}
Fix a deterministic $s\in[0,1]$ and extend $W$ to negative
times using an independent Brownian motion. Then
$B^{(s)}_u=W_{s+u}-W_s$, $u\in\R$, is a two-sided Brownian
motion: the processes $(B^{(s)}_u)_{u\ge0}$ and
$(B^{(s)}_{-u})_{u\ge0}$ are independent standard Brownian
motions. The shifted cells $[s+k/n_j,s+(k+1)/n_j]$ form the
mesh $[k/n_j,(k+1)/n_j]$ for $B^{(s)}$. We use them alongside
the original cells, with the same normalization $\kappa_j$.
The tightness estimates in Lemma~\ref{lem:kernel-moments},
applied to these additional internal distances, allow us to pass
to a further joint subsequence while preserving the previously
fixed law of $(W,(\D^I)_{I\in\mathscr J})$.

For each positive integer $M$, denote the normalized internal distance
on the shifted cells in $[s-M,s+M]$ by
$\widehat d_j^{[-M,M]}$ and its limit by
$\widehat\D^{[-M,M]}$; their arguments are measured relative
to $s$. For $\delta>0$ dyadic, define
\begin{equation}\label{eq:shifted-two-sided-minimum}
 b_{2,s}(\delta)=\inf_{\substack{N\in\N\\N>\delta}}
  \inf_{\substack{u\in[-N,N]\\|u|\ge\delta}}
   \widehat\D^{[-N,N]}(0,u).
\end{equation}
This is the quantity $b_2(\delta)$ from
\eqref{eq:two-sided-root} for the shifted Brownian motion
$B^{(s)}$.

\Needspace{17\baselineskip}
We compare paths in the two finite intersection graphs.
Each original time cell has length $1/n_j$, and its interior
contains at most one division point of the shifted mesh.
Consequently, the common refinement partitions each original
cell into at most two smaller intervals. Fix $t\in[0,1]$ and
choose a shortest path of $l$ edges between the original cells
assigned to $s$ and $t$. For each edge of this path, choose a
Brownian position belonging to the ranges of its two incident
cells. Within each visited original cell, connect the chosen
positions by passing through at most two cells of the common
refinement, using $s$ and $t$ at the two ends. The resulting path
in the refined intersection graph has at most $2(l+1)$ vertices,
and hence at most $2l+1$ edges.

Each refined cell is contained in a shifted cell. Replacing every
refined vertex by its containing shifted cell, and deleting
consecutive repetitions, preserves the graph path and does not
increase its length. If a chosen time is a boundary point of the
shifted mesh, the assigned cell may differ from the projected
cell by one adjacent cell. Connecting the prescribed cells at
$s$ and $t$ therefore adds at most two edges. For every integer $M>2$, all shifted cells used in this
construction lie in $[s-M,s+M]$. Consequently,
\begin{equation}\label{eq:shifted-finite-comparison}
 \widehat d_j^{[-M,M]}(0,t-s)
 \le 2d_j^{[0,1]}(s,t)+\frac{C}{\kappa_j},
 \qquad t\in[0,1],
\end{equation}
where $C$ is a deterministic constant independent of $j$ and $t$.
Passing to the joint limit gives
\begin{equation}\label{eq:shift-comparison}
 \widehat\D^{[-M,M]}(0,t-s)\le2\D(s,t),\qquad t\in[0,1].
\end{equation}
For any $\varepsilon>0$, a positive dyadic $\delta\le\varepsilon$
satisfies, by Lemma~\ref{lem:two-sided-positive},
\[
 \inf_{\substack{t\in[0,1]\\|t-s|\ge\varepsilon}}\D(s,t)
 \ge\frac12b_{2,s}(\delta)>0,
\]
with the usual convention for an empty infimum. This proves
\eqref{eq:deterministic-root} simultaneously for all
$\varepsilon>0$, since only countably many dyadic $\delta$ are used.

For a continuous pseudometric $\D$ on $[0,1]$, the starting times
not separated from every other time form the set
\[
  A(\D)=\{s\in[0,1]:\text{there exists }t\in[0,1]\setminus\{s\}\text{ with }\D(s,t)=0\}.
\]
For every $k\ge1$, the set of pairs $(s,t)$ satisfying
$|s-t|\ge1/k$ and $\D(s,t)=0$ is compact. Projection onto the
first coordinate shows that $A(\D)$ is an $F_\sigma$ set.
This construction is also Borel when $\D$ varies. Indeed,
for each $k$, the set of triples $(\D,s,t)$ satisfying
$|s-t|\ge1/k$ and $\D(s,t)=0$ is closed. Its projection onto
$(\D,s)$ is closed because $t$ belongs to the compact set
$[0,1]$. Taking the countable union shows that
$\{(\D,s):s\in A(\D)\}$ is Borel.
By \eqref{eq:deterministic-root}, each deterministic $s$
belongs to $A(\D)$ with probability zero. Fubini's theorem then
gives $\Leb A(\D)=0$ almost surely, and
$T_*=[0,1]\setminus A(\D)$ satisfies \eqref{eq:typical-root}.

The set $\pi(A(\D))$ consists exactly of the points in $F$
with more than one time representative. It is $F_\sigma$, because
$\pi$ is continuous, and its inverse image under $\pi$ equals
$A(\D)$ by the equivalence relation defined by $\D=0$. Hence
$\nu(\pi(A(\D)))=\Leb A(\D)=0$. Every point outside this set
has a singleton inverse image and therefore zero $\nu$-mass;
thus $\nu$ is nonatomic. Finally, the inverse image under $\pi$
of any nonempty open set in $F$ is a nonempty relatively open subset
of $[0,1]$, with positive Lebesgue measure. This proves full support.
\end{proof}

Together with the previous sections, Theorem~\ref{thm:deterministic-root}
now completes Theorem~\ref{thm:main-metric}. Its normalization and
tightness statements follow from Lemma~\ref{lem:records} and
Proposition~\ref{prop:local-field}, while part (i) follows from
Lemma~\ref{lem:critical-modulus}. The gluing identity in part (ii)
is Theorem~\ref{thm:exact-gluing}, and part (iii) is provided by
Theorem~\ref{thm:measurable} and Corollary~\ref{cor:strong-convergence}.
Part (iv) is precisely the separation proved above.

\subsection{Curve dimensions and chain thresholds}\label{subsec:pinned-dimensions}

We now use the limiting distance to estimate dimensions of
connected sets in $K$. When $\Phi(x)\ne\Phi(y)$, the function
$t\mapsto\D(s,t)$, with $W_s=x$, distinguishes $x$ and $y$ on
the Brownian range. Lemma~\ref{lem:dimension-test} will then give
$\dimH C\ge\gamma_3$ for every connected set $C\subset K$
containing these two points. Together with the curves supplied by
Proposition~\ref{prop:curve-upper}, this yields the exact minimum
curve dimension for such endpoints.
\begin{proposition}\label{prop:point-to-point}
With probability one, every pair $x,y\in K$ with
$\Phi(x)\ne\Phi(y)$ satisfies $p_K(x,y)=\gamma_3$, and there is a
simple curve from $x$ to $y$ in $K$, $\beta$-H\"older for every
$0<\beta<1/\gamma_3$, whose image has Hausdorff and box dimensions
$\gamma_3$. For every such pair $x,y$, each connected subset of $K$
containing $x,y$ has Hausdorff dimension at least $\gamma_3$.
\end{proposition}
\begin{proof}
Let $W_s=x$ for some $s\in[0,1]$ and set $u(t)=\D(s,t)$.
By \eqref{eq:local-identifications}, the function
$\bar u(W_t)=u(t)$ is well defined on $K$. Since
$\Phi(x)\ne\Phi(y)$, it takes different values at $x$ and $y$.
Moreover, \eqref{eq:local-holder} makes $u$ $\alpha$-H\"older for
every $0<\alpha<\theta$. Consequently,
Lemma~\ref{lem:dimension-test} gives $\dimH C\ge2\alpha$ for every
connected $C\subset K$ containing $x,y$. Letting
$\alpha\uparrow\theta$ yields $\dimH C\ge\gamma_3$.
Proposition~\ref{prop:curve-upper} supplies a simple connecting
curve with upper box dimension at most $\gamma_3$, so
\eqref{eq:box-comparison} forces equality of its Hausdorff and
box dimensions with $\gamma_3$.
\end{proof}

\begin{proof}[Proof of Theorem~\ref{thm:main-pinned}]
Fix one subsequence along which the joint convergence
\eqref{eq:joint-local-limit} of Proposition~\ref{prop:local-field}
holds, and write $\D=\D^{[0,1]}$ for its limiting distance.
By Theorem~\ref{thm:measurable}, this distance has the
representation $\D=\Psi(W)$ almost surely for a Borel map
$\Psi$ on the original Wiener space. For each fixed deterministic
$s$, separation in
Theorem~\ref{thm:deterministic-root} implies
$\Phi^{-1}(\{\pi(s)\})=\{W_s\}$ almost surely.
Proposition~\ref{prop:point-to-point} therefore gives the dimension
identity, attainment and the lower bound for connected sets
simultaneously for every $y\in K\setminus\{W_s\}$. By \eqref{eq:typical-root},
the same conclusions hold simultaneously for every $s\in T_*$.
Since $\Psi$ is Borel and the exceptional set $A(\D)$ in the proof
of Theorem~\ref{thm:deterministic-root} is Borel in $(\D,s)$,
the indicator of $\{s\in T_*\}$ is a Borel function of $(W,s)$.

For nonuniqueness, we use the Brownian motions obtained by
rescaling the forward and backward increments at time $s$:
\begin{equation}\label{eq:pinned-rescaled-paths}
 \begin{aligned}
 t&\longmapsto\frac{W_{s+(1-s)t}-W_s}{\sqrt{1-s}}
   &&(s<1),\\
 t&\longmapsto\frac{W_{s-st}-W_s}{\sqrt{s}}
   &&(s>0),\qquad 0\le t\le1.
 \end{aligned}
\end{equation}
For each fixed $s$, each of these paths that is defined is a
standard Brownian motion. Since every $y\in K\setminus\{W_s\}$ lies in
$W([s,1])$ or $W([0,s])$, Proposition~\ref{prop:nonunique-curves}
applied to the corresponding rescaled path gives continuum many
simple curves from $W_s$ to $y$ with distinct images. Scaling back
preserves finite $p$-variation for all $p>\gamma_3$. By
Lemma~\ref{lem:holder-arc}, each curve therefore admits a single
parametrization that is $\beta$-H\"older for every
$0<\beta<1/\gamma_3$, without changing its image.
Together with the lower bound for connected sets from
Proposition~\ref{prop:point-to-point}, these H\"older bounds imply
that its Hausdorff and box dimensions equal $\gamma_3$, by
\eqref{eq:box-comparison}.

To prove the assertion for almost every starting time, choose a
Borel set $E_0\subset C_0([0,1],\R^3)$ such that
\begin{equation}\label{eq:nonuniqueness-borel-event}
 \Pp\bigl(W|_{[0,1]}\in E_0\bigr)=1
\end{equation}
and such that every path in $E_0$ has the conclusion of
Proposition~\ref{prop:nonunique-curves} for all of its target
points. Such a Borel set exists because the proposition holds
under the completed Wiener law, which is the distribution of
$W|_{[0,1]}$ on $C_0([0,1],\R^3)$. For $0<s<1$,
both rescaling maps in \eqref{eq:pinned-rescaled-paths} are jointly
continuous in $(W,s)$. Their inverse images of $E_0$ are therefore Borel, and
Fubini's theorem gives, almost surely, a Borel set of starting
times of full Lebesgue measure for which both rescaled paths belong
to $E_0$. Replacing $T_*$ from \eqref{eq:typical-root} by its intersection
with this Borel set of full Lebesgue measure preserves all
the conclusions of Theorem~\ref{thm:deterministic-root}
and completes the simultaneous curve assertions.
\end{proof}

Finally, separation of the limiting distance also identifies
the exponent threshold for the chain pseudometrics
$\delta_\alpha^W$ at each predetermined starting time.
\begin{corollary}\label{cor:fixed-threshold}
For each fixed deterministic $s\in[0,1]$, almost surely and simultaneously for
$0<\alpha<\theta$ and $\varepsilon>0$,
\begin{equation}\label{eq:fixed-threshold}
 \inf_{\substack{t\in[0,1]\\|t-s|\ge\varepsilon}}
       \delta_\alpha^W(s,t)>0.
\end{equation}
Consequently, for every $t\in[0,1]\setminus\{s\}$,
\[
 \sup\{\alpha\in(0,1):\delta_\alpha^W(s,t)>0\}=\theta.
\]
On one event of probability one, these conclusions hold simultaneously
for Lebesgue almost every $s\in[0,1]$.
\end{corollary}
\begin{proof}
Fix $0<\alpha<\theta$ and $s,t\in[0,1]$.
For any chain $(a_l,b_l)_{l=1}^m$ admissible in
\eqref{eq:holder-chain}, we have $a_1=s$, $b_m=t$ and
$W_{b_l}=W_{a_{l+1}}$ for $l<m$. By
\eqref{eq:local-identifications},
$\D(s,b_l)=\D(s,a_{l+1})$. Thus the changes of the function
$u\mapsto\D(s,u)$ along successive pairs add up to its total
change from $s$ to $t$. The triangle inequality and
\eqref{eq:local-holder} give
\[
 \begin{aligned}
  \D(s,t)
  &\le\sum_{l=1}^m
       |\D(s,a_l)-\D(s,b_l)|\\
  &\le H_{[0,1],\alpha}
       \sum_{l=1}^m|a_l-b_l|^\alpha.
 \end{aligned}
\]
Taking the infimum over these chains proves
\begin{equation}\label{eq:metric-chain-comparison}
 \D(s,t)\le H_{[0,1],\alpha}\delta_\alpha^W(s,t).
\end{equation}
For each fixed $s$, \eqref{eq:deterministic-root} and
\eqref{eq:metric-chain-comparison} yield \eqref{eq:fixed-threshold}.
For $s\in T_*$, \eqref{eq:typical-root} supplies pointwise
positivity of $\D(s,\cdot)$, and continuity makes its minimum
positive over every nonempty set $\{t:|t-s|\ge\varepsilon\}$.
The upper bound on the threshold follows from
Corollary~\ref{cor:chain-threshold}, which states that
$\delta_\alpha^W\equiv0$ for $\alpha>\theta$.
To obtain the assertions for all subcritical exponents and all gaps
on the same event, we first use countable sequences of exponents
increasing to $\theta$ and of positive dyadic gaps, and then apply
monotonicity. This also applies on the set of full Lebesgue measure of times
from \eqref{eq:typical-root}.
\end{proof}

\section{Ball exit partitions}\label{sec:comparisons}\label{sec:exits}

In addition to the equal time partitions considered above, we study
partitions determined by successive exits of Brownian motion from balls.
For $r>0$, define
\begin{equation}\label{eq:exit-times}
 \tau_0^{(r)}=0,\qquad
 \tau_{k+1}^{(r)}=\inf\{t>\tau_k^{(r)}:
               |W_t-W_{\tau_k^{(r)}}|=r\}.
\end{equation}
Let $N_r=\min\{k\ge1:\tau_k^{(r)}\ge1\}$. These stopping times give
an interval partition of $[0,1]$ with cells
\[
 [\tau_k^{(r)},\tau_{k+1}^{(r)}]\quad(0\le k<N_r-1),
 \qquad [\tau_{N_r-1}^{(r)},1].
\]
The intersection graph is defined as in \eqref{eq:graph-definition},
using the Brownian ranges of these cells. Its endpoint distance and diameter are denoted by
$L_r^{\rm ex}$ and $V_r^{\rm ex}$, respectively.
Almost surely, both are positive for all sufficiently small $r$.

\Needspace{7\baselineskip}
\begin{theorem}\label{thm:main-exit}
The endpoint distance $L_r^{\rm ex}$ and diameter $V_r^{\rm ex}$ satisfy
\begin{gather}
 \lim_{r\downarrow0}\frac{\log L_r^{\rm ex}}{\log(1/r)}
 =\lim_{r\downarrow0}\frac{\log V_r^{\rm ex}}{\log(1/r)}
 =\gamma_3,\label{eq:main-exit}\\
 \liminf_{r\downarrow0}r^{\gamma_3}\log(1/r)L_r^{\rm ex}>0
       \label{eq:main-exit-lower}
\end{gather}
almost surely. For every $p>0$, there is a deterministic $c_p>0$
such that, for all sufficiently small $r>0$,
\begin{equation}\label{eq:main-exit-moments}
 \frac{c_p r^{-p\gamma_3}}{[\log(1/r)]^p}
 \le\E[(L_r^{\rm ex})^p]\le\E[(V_r^{\rm ex})^p]
 =r^{-p\gamma_3+o(1)}.
\end{equation}
\end{theorem}

We compare the exit partition at radius $r$ with the equal time
partition whose intersection graph $G_n$ is defined in
\eqref{eq:graph-definition}, taking $n=\lceil r^{-2}\rceil$.
The required distance bounds come from comparing the graph distances
for two time partitions and controlling the durations between
successive exits.

\subsection{Comparison of time partitions}\label{subsec:partition-comparison}

The comparison of graph distances for two time partitions is
independent of the Brownian law. To formulate it, let
$w:[0,1]\to\R^3$ be continuous, and let $\mathcal P$ and
$\mathcal Q$ be finite partitions of $[0,1]$ into nondegenerate closed
intervals with disjoint interiors. For $\mathcal A\in\{\mathcal P,\mathcal Q\}$, let
$G_{\mathcal A}(w)$ be the graph whose vertices are the cells of
$\mathcal A$, with distinct cells $I,J$ adjacent when
$w(I)\cap w(J)\ne\varnothing$. At an internal partition point,
the time is assigned to the cell on its right; time $1$ is assigned
to the last cell. The notation $d_{\mathcal A}(s,t)$ refers to the
graph distance between the cells assigned to $s$ and $t$. List all distinct partition points of $\mathcal P$ and $\mathcal Q$
in increasing order as $0=t_0<t_1<\cdots<t_m=1$. Their
\emph{common refinement}, denoted by $\mathcal P\vee\mathcal Q$,
is the partition with cells $[t_{j-1},t_j]$, $1\le j\le m$.

\Needspace{7\baselineskip}
\begin{lemma}\label{lem:partition-comparison}
Suppose every cell of $\mathcal P$ contains at most $A$ cells of the common
refinement $\mathcal P\vee\mathcal Q$. Then
\begin{equation}\label{eq:partition-comparison}
 1+d_{\mathcal Q}(s,t)
 \le A\,[1+d_{\mathcal P}(s,t)]\qquad(s,t\in[0,1]).
\end{equation}

\end{lemma}
\begin{proof}
Let $I_0,\ldots,I_\ell$ be a shortest path in $G_{\mathcal P}(w)$
from the cell assigned to $s$ to the cell assigned to $t$.
For each $1\le j\le\ell$, there are times $a_j\in I_{j-1}$ and
$b_j\in I_j$ with $w(a_j)=w(b_j)$. Set $b_0=s$ and
$a_{\ell+1}=t$.

Within each $I_j$, $0\le j\le\ell$, join $b_j$ to $a_{j+1}$
by the consecutive cells of the common refinement between these
two times. If a chosen time is an endpoint of $I_j$, use the
refinement cell adjacent to that endpoint from within $I_j$;
it contains the chosen time since the refinement cells are closed.
At $s$ and $t$, use the refinement cells assigned by the
boundary convention specified above. As each $I_j$ contains at most
$A$ refinement cells, these connections use at most $A(\ell+1)$
cells in total.

Consecutive refinement cells within an $I_j$ share a time endpoint.
At the transition from $I_{j-1}$ to $I_j$, the last and first
chosen cells have intersecting images under $w$, since they
contain $a_j$ and $b_j$, respectively, and $w(a_j)=w(b_j)$.
Thus these cells form a path in the refinement graph, with at most
$A(\ell+1)$ vertices. Each refinement cell lies in a unique cell
of $\mathcal Q$ when interiors are considered. Replacing each
refinement cell by this $\mathcal Q$ cell and deleting consecutive
repetitions gives a path in $G_{\mathcal Q}(w)$ from the cell
assigned to $s$ to the cell assigned to $t$. Hence
\[
 1+d_{\mathcal Q}(s,t)\le A(\ell+1)
   =A[1+d_{\mathcal P}(s,t)],
\]
as claimed.
\end{proof}

\subsection{Bounds on exit times}\label{subsec:exit-times}

For a fixed radius $r$, the strong Markov property and Brownian
scaling show that the full exit durations
$(\tau_{i+1}^{(r)}-\tau_i^{(r)})_{i\ge0}$ are independent
and each has the law of $r^2\tau$, where $\tau$ is the first
exit time from the unit ball for Brownian motion started at the
origin. The comparison of partitions requires bounds that hold
simultaneously for all sufficiently small $r$.

\Needspace{13\baselineskip}
\begin{lemma}\label{lem:exit-durations}
There are $c,C>0$ such that, for $u\ge1$,
\begin{equation}\label{eq:exit-duration-tails}
 \Pp(\tau<u^{-1})\le Ce^{-cu},\qquad
 \Pp(\tau>u)\le Ce^{-cu}.
\end{equation}
Almost surely, for every sufficiently small $r>0$,
\begin{equation}\label{eq:exit-comparability}
 \frac{r^2}{C\log(1/r)}
 \le \tau_{i+1}^{(r)}-\tau_i^{(r)}
 \le Cr^2\log(1/r),\qquad 0\le i<N_r.
\end{equation}
For the last index $i=N_r-1$, the inequalities apply to
$\tau_{N_r}^{(r)}-\tau_{N_r-1}^{(r)}$, even though the last cell
of the partition ends at time $1$. Moreover, for every $a>0$
there is $C_a>0$ such that, for each sufficiently small fixed $r$,
\eqref{eq:exit-comparability} holds with $C_a$ in place of $C$
outside an event of probability at most $C_ar^a$.
\end{lemma}
These estimates are standard; see
\cite[Theorems 1.14, 2.16 and 2.21]{MortersPeres} and
\cite[Theorem 1.6.1]{CsorgoRevesz}. The probability bound for a
fixed radius also follows by a union bound using
\eqref{eq:exit-duration-tails} and the exponential estimate for
$N_r$ below. We omit the proof.

The number $N_r$ of exit intervals satisfies a corresponding moment
upper bound. Indeed, if $q_0=\E[e^{-\tau}]<1$, independence of the full durations
and the exponential Markov inequality give, for integers $m\ge1$,
\[
 \Pp(N_r>m)=\Pp(\tau_m^{(r)}<1)
 \le e^{r^{-2}}q_0^m.
\]
Splitting the tail sum at a sufficiently large multiple of $r^{-2}$
then yields
\begin{equation}\label{eq:exit-count-moments}
 \E N_r^q\le C_qr^{-2q}\qquad(0<r<1,\ q>0).
\end{equation}

\subsection{Distance growth for ball exit partitions}\label{subsec:exit-growth}

For the exit partition $\mathcal P_r$, write
$d_r^{\rm ex}(s,t)=d_{\mathcal P_r}(s,t)$.
The estimates from the preceding subsections compare this distance
with the distance in the equal time partition at
$n=\lceil r^{-2}\rceil$.

\Needspace{7\baselineskip}
\begin{proposition}\label{prop:exit-comparison}
Almost surely, for every sufficiently small $r>0$, with
$n=\lceil r^{-2}\rceil$,
\begin{equation}\label{eq:exit-distance-comparison}
 \frac{1+d_{G_n}(\iota_n(s),\iota_n(t))}{C\log(1/r)}
 \le 1+d_r^{\rm ex}(s,t)
 \le C\log(1/r)\,[1+d_{G_n}(\iota_n(s),\iota_n(t))]
\end{equation}
for every $s,t\in[0,1]$. For every $a>0$, this comparison holds at
each sufficiently small fixed $r$, with $C_a$ in place of $C$,
outside an event of probability at most $C_ar^a$.
\end{proposition}
\begin{proof}
Let $\mathcal Q_n$ be the partition of $[0,1]$ into the $n$
intervals $[j/n,(j+1)/n]$, $0\le j<n$, where
$n=\lceil r^{-2}\rceil$. On the probability one event in
Lemma~\ref{lem:exit-durations}, let
$r>0$ be sufficiently small. Any two consecutive exit times
strictly before time $1$ are separated by at least
$r^2/[C\log(1/r)]$. An interval of length $1/n\le r^2$
therefore contains at most $C\log(1/r)+1$ exit times in its
interior. Thus each cell of $\mathcal Q_n$ contains at most
$C\log(1/r)+2$ cells of the common refinement
$\mathcal P_r\vee\mathcal Q_n$.

Conversely, the length of each exit cell is at most
$Cr^2\log(1/r)$ by \eqref{eq:exit-comparability}; this includes
the last cell, whose right endpoint is $1$. Since
$n=\lceil r^{-2}\rceil\le 2r^{-2}$ for $r\le1$, each such cell
contains at most $C\log(1/r)+2$ cells of the common refinement.
Lemma~\ref{lem:partition-comparison}, applied first to
$(\mathcal P_r,\mathcal Q_n)$ and then with their roles reversed,
now gives \eqref{eq:exit-distance-comparison}, after enlarging $C$.
The probability statement for each fixed $r$ follows in the
same way from the last assertion of Lemma~\ref{lem:exit-durations}.
\end{proof}

The modulus estimate \eqref{eq:critical-modulus} and the endpoint
separation in Theorem~\ref{thm:deterministic-root} give
quantitative lower bounds for $L_n$.

\Needspace{7\baselineskip}
\begin{corollary}\label{cor:quantitative-lower-bounds}
Almost surely,
\begin{equation}\label{eq:time-distance-lower}
 \inf_{n\ge2}\frac{[\log(en)]^{1-\theta}L_n}{n^\theta}>0.
\end{equation}
For every $p>0$, there is a deterministic $c_p>0$ such that
\begin{equation}\label{eq:time-moment-lower}
 \E[L_n^p]\ge
 \frac{c_p n^{p\theta}}{[\log(en)]^{p(1-\theta)}}
 \qquad(n\ge2).
\end{equation}
\end{corollary}
\begin{proof}
By Theorem~\ref{thm:main-metric}(iii), we may realize
$\D=\D^{[0,1]}$ on the original Wiener space. Theorem~\ref{thm:main-metric}(i),
(iv) then give a finite random variable $H>0$ such that
\[
 \begin{gathered}
 \omega(0)=0,\qquad
 \omega(h)=h^\theta[\log(e/h)]^{1-\theta}\quad(0<h\le1),\\
 \D(s,t)\le H\omega(|s-t|),\qquad
 c_W:=\frac{\D(0,1)}H\in(0,1].
 \end{gathered}
\]
The function $\omega$ is increasing on $(0,1]$ because
$\theta>1/2$. Furthermore, $u(t)=\D(0,t)$ is compatible with $W$
by \eqref{eq:main-local}, and the triangle inequality gives
$|u(s)-u(t)|\le\D(s,t)$.

Consider any finite partition $\mathcal P$ as in
Section~\ref{subsec:partition-comparison}, with maximum cell length
$h\le1$. Along a graph path from its first to its last cell,
choose a common Brownian position at each transition and time
representatives in the two incident cells. Compatibility makes the
change in $u$ at each transition zero, whereas the change within
any one cell is at most $H\omega(h)$. Summing the changes along a
shortest path therefore gives
\begin{equation}\label{eq:general-partition-lower}
 \D(0,1)\le H[1+d_{\mathcal P}(0,1)]\omega(h),\qquad
 1+d_{\mathcal P}(0,1)\ge\frac{c_W}{\omega(h)}.
\end{equation}
For the equal time partition $h=1/n$, this implies
\[
 L_n\ge\frac{c_W}{2}
 \frac{n^\theta}{[\log(en)]^{1-\theta}}\qquad(n\ge2),
\]
since $L_n+1\le2L_n$. This proves
\eqref{eq:time-distance-lower}. Taking $p$-th moments gives
\eqref{eq:time-moment-lower}, with $c_p=2^{-p}\E[c_W^p]>0$.

\end{proof}

\begin{proof}[Proof of Theorem~\ref{thm:main-exit}]
For $n=\lceil r^{-2}\rceil$, applying
\eqref{eq:exit-distance-comparison} with $s=0,t=1$ compares the
endpoint distances; taking the supremum over $s,t\in[0,1]$
compares the diameters. Theorem~\ref{thm:main-growth} and
$\log n\sim2\log(1/r)$ therefore give \eqref{eq:main-exit}.

For the quantitative lower bounds, choose a deterministic $C>0$ and
put
\[
 E_r=\bigl\{\tau_{i+1}^{(r)}-\tau_i^{(r)}\le
 Cr^2\log(1/r)\text{ for all }0\le i<N_r\bigr\}.
\]
By Lemma~\ref{lem:exit-durations}, $C$ may be chosen so that
$\Pp(E_r^c)\le Cr$ for all sufficiently small $r$, and
almost surely $E_r$ holds for every sufficiently small $r$.
On $E_r$, every cell of the exit partition has length at most
$h_r=Cr^2\log(1/r)$. For all sufficiently small $r$,
$\omega(h_r)\le C'r^{\gamma_3}\log(1/r)$, where $C'$ is
deterministic. Thus \eqref{eq:general-partition-lower} yields
\begin{equation}\label{eq:exit-lower-on-good-event}
 1+L_r^{\rm ex}\ge
 \frac{c_W}{C'r^{\gamma_3}\log(1/r)}\qquad\text{on }E_r.
\end{equation}
This proves \eqref{eq:main-exit-lower}, since $c_W>0$ almost surely.

Fix $p>0$ and choose $\delta>0$ such that
$q:=\Pp(c_W\ge\delta)>0$. For all sufficiently small $r$,
\[
 \Pp(E_r\cap\{c_W\ge\delta\})\ge q/2.
\]
On this event, after decreasing the upper bound on $r$ if
necessary, \eqref{eq:exit-lower-on-good-event} gives
\[
 L_r^{\rm ex}\ge
 \frac{\delta}{2C'}\frac{r^{-\gamma_3}}{\log(1/r)}.
\]
Taking $p$-th moments proves the lower bound in
\eqref{eq:main-exit-moments}.

For the upper bound, choose $a>4p$ in
Proposition~\ref{prop:exit-comparison}. On the event where
\eqref{eq:exit-distance-comparison} holds,
$1+V_r^{\rm ex}\le C_a\log(1/r)(1+V_n)$; on its complement,
$1+V_r^{\rm ex}\le N_r$. By Cauchy--Schwarz and
\eqref{eq:exit-count-moments}, the exceptional contribution is at most
$C_{a,p}r^{a/2-2p}$. Consequently,
\[
 \E[(1+V_r^{\rm ex})^p]
 \le [C_a\log(1/r)]^p\E[(1+V_n)^p]
       +C_{a,p}r^{a/2-2p}
 =r^{-p\gamma_3+o(1)},
\]
where we used \eqref{eq:main-moments}. Combined with the lower
bound and $L_r^{\rm ex}\le V_r^{\rm ex}$, this proves
\eqref{eq:main-exit-moments}.
\end{proof}

\section{Further questions}\label{sec:questions}

In this section, we present several questions left open by the present work.

\subsection{Comparison with simple random walk}\label{subsec:discrete-comparison}

Let $S=(S_k)_{k\ge0}$ be simple random walk on $\Z^d$,
$d\in\{2,3\}$, started at the origin, and write
$\Pp_d^{\mathrm{RW}}$ and $\E_d^{\mathrm{RW}}$ for its law and expectation.
The graph $\mathcal R_n$ has vertex set $\{S_0,\ldots,S_n\}$ and edge set
$\{\{S_{k-1},S_k\}:1\le k\le n\}$, with unit edge lengths. Set
\[
 A_n^{\mathrm{RW}}=d_{\mathcal R_n}(0,S_n).
\]
For $R\ge1$, define
\[
 T_R=\inf\{k\ge0:|S_k|>R\},\qquad
 \calG_R^{\mathrm{RW}}=\mathcal R_{T_R},\qquad
 D_R^{\mathrm{RW}}=d_{\calG_R^{\mathrm{RW}}}(0,S_{T_R}).
\]
Thus $\calG_R^{\mathrm{RW}}$ includes the edge by which the random walk leaves
the ball. In dimension three, transience implies that
$A_n^{\mathrm{RW}}>0$ for all sufficiently large $n$, almost surely.
We ask whether the graph distances admit deterministic exponents and
whether their three-dimensional values agree with the Brownian exponents.

\begin{question}\label{qu:discrete-time-exponent}
In dimension three, is there a deterministic exponent
$\theta_{\mathrm{RW}}$ such that
\[
 \E_3^{\mathrm{RW}}[A_n^{\mathrm{RW}}]
     =n^{\theta_{\mathrm{RW}}+o(1)},\qquad
 \lim_{n\to\infty}\frac{\log A_n^{\mathrm{RW}}}{\log n}
     =\theta_{\mathrm{RW}}
 \quad\Pp_3^{\mathrm{RW}}\text{-almost surely}?
\]
For each $d\in\{2,3\}$, is there a deterministic exponent $\beta_d$
such that
\[
 \E_d^{\mathrm{RW}}[D_R^{\mathrm{RW}}]=R^{\beta_d+o(1)},\qquad
 \lim_{R\to\infty}\frac{\log D_R^{\mathrm{RW}}}{\log R}
     =\beta_d
 \quad\Pp_d^{\mathrm{RW}}\text{-almost surely}?
\]
If $\theta_{\mathrm{RW}}$ and $\beta_3$ exist, do we have
$\theta_{\mathrm{RW}}=\theta$ and
$\beta_3=2\theta_{\mathrm{RW}}=\gamma_3$?
\end{question}

Movshovitz and Havlin~\cite[Section 4, equations (4.2) and (4.6)]{MovshovitzHavlin}
use a Flory argument to predict that the mean square Euclidean
radius at chemical distance $\ell$ scales as $\ell^{4/d}$.
Although this concerns a different observable, it suggests the
spatial exponents $d/2$.
\begin{question}\label{qu:discrete-exponent-values}
If the exponents $\beta_2$ and $\beta_3$ exist, are
$\beta_2=1$ and $\beta_3=3/2$? For the Brownian exponent,
is $\gamma_3=3/2$?
\end{question}

Zhou, Deng and Hou~\cite{ZhouDengHou} study the mean maximal graph
distance from the starting point in the trace of a two-dimensional
random walk run for $L^2$ steps on a torus of side length $L$.
Their simulations support the scaling $L(\log L)^{1/4}$ for this
observable. The distance $D_R^{\mathrm{RW}}$ considered here differs from the quantity studied in~\cite{ZhouDengHou}, and it remains open whether
\[
 \E_2^{\mathrm{RW}}[D_R^{\mathrm{RW}}]=R(\log R)^{1/4+o(1)}.
\]

\subsection{Spatial separation of the limiting distance}\label{subsec:spatial-separation-question}

We ask whether the limiting distance $\D=\D^{[0,1]}$ from
Theorem~\ref{thm:main-metric} satisfies the following spatial separation
property simultaneously for all $s,t\in[0,1]$, almost surely:
\begin{equation}\label{eq:full-separation-question}
 \D(s,t)=0\quad\Longleftrightarrow\quad W_s=W_t
                   \qquad(s,t\in[0,1]).
\end{equation}
The implication $W_s=W_t\Rightarrow\D(s,t)=0$ follows from
\eqref{eq:main-local}. For each fixed deterministic $s$,
Theorem~\ref{thm:deterministic-root} gives $\D(s,t)>0$ for every
$t\ne s$, almost surely. The event of probability one may depend on $s$;
this does not establish the reverse implication in
\eqref{eq:full-separation-question} simultaneously for every pair
$s,t\in[0,1]$.
If it does, the map $\Phi:K\to F$ in \eqref{eq:spatial-map}
is a homeomorphism. If it fails, there is a $z\in F$ such that
$\Phi^{-1}(\{z\})$ has positive Euclidean diameter; this set is
compact and connected by Proposition~\ref{prop:quotient-geometry}.

\begin{question}\label{qu:critical-separation}
Does every subsequential limit of the normalized graph distances along
$(n_j)$ satisfy \eqref{eq:full-separation-question} almost surely?
\end{question}

An affirmative answer would also give $p_K(x,y)=\gamma_3$ for every
distinct $x,y\in K$ on one event of probability one, by
Proposition~\ref{prop:point-to-point}.

\vspace{2em}

\noindent\textbf{AI statement.}
Commercial large language models were used to assist with
polishing the English, preparing parts of the manuscript
in \LaTeX{}, and checking some proofs.
All mathematical arguments were verified by the author,
who takes full responsibility for the content of the paper.

\vspace{0.5em}

\noindent\textbf{Acknowledgments.}
The author was supported in part by JSPS KAKENHI Grant Numbers
JP26K00609 and JP23K22399.

\end{document}